\documentclass[11pt, oneside]{article}
\usepackage[margin=2cm]{geometry}
\usepackage[parfill]{parskip}
\usepackage{graphicx}
\usepackage[utf8]{inputenc}
\usepackage[english]{babel} 
\usepackage{amsmath,amssymb,amsthm,mathrsfs,bm,bbm} 
\usepackage[margin=2cm]{geometry}
\usepackage[color=yellow]{todonotes}
\usepackage[colorlinks]{hyperref}
\usepackage{enumerate}
\usepackage{nicefrac}
\usepackage{comment}

\usepackage{pgfplots}
\usepackage{float}
\usepackage{tikz,tikz-cd}
\usepackage{graphicx}
\usepackage{caption}
\usepackage{multicol}
\usepackage{subcaption}
\extrafloats{100}
\numberwithin{equation}{section}

\newtheorem{theorem}{Theorem}[section]
\newtheorem{lemma}[theorem]{Lemma}
\newtheorem{corollary}[theorem]{Corollary}
\newtheorem{proposition}[theorem]{Proposition}
\newtheorem{remark}[theorem]{Remark}

\theoremstyle{definition}
\newtheorem{definition}[theorem]{Definition}

\pgfplotsset{compat=1.16}

\newcommand{\vertiii}[1]{{\left\vert\kern-0.25ex\left\vert\kern-0.25ex\left\vert #1
    \right\vert\kern-0.25ex\right\vert\kern-0.25ex\right\vert}}

\newcommand{\R}{\mathbb{R}}
\renewcommand{\P}{\mathbb{P}}
\newcommand{\N}{\mathbb{N}}
\newcommand{\E}{\mathbb{E}}
\newcommand{\tr}{\operatorname{tr}}

\begin{document}

\title{\textbf{Stochastic processes on reductive homogeneous and symmetric spaces}}
\author{Sonja Cox \thanks{\href{mailto:s.g.cox@uva.nl}{s.g.cox@uva.nl} Korteweg--de Vries Institute of Mathematics, University of Amsterdam}
\and Erwin Luesink \thanks{\href{mailto:e.luesink@uu.nl}
{e.luesink@uu.nl} Mathematical Institute, Utrecht University}
}
\date{}

\maketitle

\begin{abstract}
In this paper we present a method to effectively and efficiently provide pathwise approximations of the Riemannian Brownian motion on a reductive homogeneous space $M\cong G/H$ endowed with a $G$-invariant metric. This setting includes spheres, hyperbolic spaces, the space of symmetric positive-definite matrices, flag manifolds and Stiefel manifolds. The simulation is based on Wong--Zakai approximations of an appropriately chosen SDE in $G$ that is projected down onto $M\cong G/H$. The effectiveness lies in the fact that well-established results on the convergence rate of the Wong--Zakai equation immediately carry over to the homogeneous space thanks to Nash's embedding. The efficiency lies in the fact that explicit representations of the Wong--Zakai approximations can be obtained in this setting using the Lie exponential map.

\end{abstract}
\section{Introduction}

Brownian motion on a Riemannian manifold $(M,\bm{g})$ -- the diffusion generated by one half of the Laplace--Beltrami operator $\frac{1}{2}\Delta_M $ -- is a fundamental object in stochastic analysis, geometry and, increasingly, in applications: it underlies heat-kernel methods, it is the driving noise in Langevin-type sampling algorithms on manifolds, it serves as the basic building block for statistical models of manifold-valued data, and it appears as the natural random motion in physical and biological models constrained to curved state spaces (rotations, orientations, covariance matrices, subspaces). In many of these applications one needs to \emph{simulate} Riemannian Brownian motion, and to do so \emph{pathwise}: one is interested not merely in the marginal law at a fixed time, but in the joint behaviour of a whole trajectory, for instance in hitting times, in functionals of the path, in the coupling of the noise across several processes driven by the same Brownian motion, or in multilevel and variance-reduction schemes that require strong (rather than weak) convergence of the approximation.

Pathwise simulation of the Riemannian Brownian motion on a general Riemannian manifold is notoriously awkward. The intrinsic constructions of the Riemannian Brownian motion -- the martingale problem for $\frac{1}{2}\Delta_M$, or the Eells--Elworthy--Malliavin constrution by `rolling without slipping', in which a Riemannian Brownian motion is obtained as the projection of a horizontal Brownian motion on the orthonormal frame bundle (see e.g.~\cite[Chapter 3]{hsu2002stochastic}) -- are elegant but are not directly implementable. Of course, if the manifold can be covered by a single chart, one can numerically approximate the appropriate SDE in coordinates -- but if multiple charts are needed, issues arise when passing from one chart to the next. Extrinsic constructions require an isometric embedding and a projection back onto the manifold, which typically destroys the geometric structure one would like to preserve. \par 

The classical directly implementable approximation for a Riemannian Brownian motion is the \emph{geodesic random walk}, which at each step draws a Gaussian increment in the tangent space and follows the geodesic in that direction; it converges in law to Riemannian Brownian motion by J{\o}rgensen's central limit theorem~\cite{jorgensen1975central}, see also~\cite{SchwarzSturmWardetzky:2025} for efficient variants and~\cite{bharath2023sampling} for (weak) convergence rates. Convergence in law is, however, all that this construction provides: the walk is driven by independent increments drawn afresh at each step and is not, a priori, coupled to any given Brownian path, so it does not constitute a pathwise approximation. 

In our work, we demonstrate that given reductive homogeneous space $M=G/H$ endowed with a $G$-invariant metric, there exists a natural Stratonovich equation in $G$ that defines the Riemannian Brownian motion in $M$. Moreover, the Wong--Zakai approximation of this Stratonovich equation is easily implementable and its projection onto $M$ inherits convergence rates from the Euclidean setting. We thus obtain a pathwise approximation of a fixed Riemannian Brownian motion in $M=G/H$, with almost sure rates of convergence and, for compact $G$, rates in $L^p$. Moreover, on \emph{naturally} reductive homogeneous spaces (a class of spaces that include spheres, hyperbolic spaces and the space of symmetric positive-definite matrices) the scheme is precisely a geodesic random walk, see Remark~\ref{rem:WZ_is_GRW}.

\paragraph{Setting and main idea.} In this paper we consider the large and important class of \emph{reductive homogeneous spaces} $M\cong G/H$, where $G$ is a Lie group acting transitively on $M$, $H$ is the (closed) isotropy group of a base point, and the Lie algebra admits an $\operatorname{Ad}_H$-invariant splitting $\mathfrak{g}=\mathfrak{h}\oplus\mathfrak{m}$. Every Lie group is trivially such a space, and so are spheres, hyperbolic spaces, the space of symmetric positive-definite matrices, Grassmannians, flag manifolds and Stiefel manifolds. We endow $M$ with a $G$-invariant Riemannian metric, which is the same as fixing an $\operatorname{Ad}_H$-invariant inner product $B$ on $\mathfrak{m}$ (all of this is explained in detail in Section~\ref{sec:redhomspaces} below to assist the reader who is not yet familiar with homogeneous spaces).

The key observation is that on such a space the Riemannian Brownian motion admits a pathwise construction that bypasses the frame bundle altogether. Choosing a $B$-orthonormal basis $e_1,\ldots,e_r$ of $\mathfrak{m}$ and independent standard Brownian motions $W_1,\ldots,W_r$, we consider the left-invariant Stratonovich SDE
\begin{equation}\label{eq:intro_sde}
    \bm{d}g(t) = (\textnormal{d}L_{g(t)})_e\Big(\sum_{i=1}^r e_i\circ\bm{d}W_i(t)\Big),\qquad g(0)=e,
\end{equation}
on the group $G$ and project its solution to $M$ via the natural map $\pi\colon G\to G/H$ (here $L_g$ denotes the left group action on $G$). We show (Proposition~\ref{prop:unimodularredhombm}) that $X=\pi(g)$ is a Riemannian Brownian motion on $M$ provided that either $G$ is unimodular or $(M,B)$ is a naturally reductive homogeneous space. The first condition covers, among others, all compact, semisimple and reductive groups, and hence all the examples listed above with \emph{any} $G$-invariant metric. The second condition is a classical sufficient conditon to ensure that the Riemannian exponential map at the base point $o\in M$ is induced by the (projected) Lie group exponential -- this property allows us to interpret the Wong--Zakai approximation as a geodesic random walk. In both cases the argument for demonstrating that the process $X=\pi(g)$ is a Riemannian Brownian motion is the same: the projected generator differs from $\tfrac12\Delta_M$ by a first-order drift term of the form $\tfrac12\sum_{i=1}^{r} U(e_i,e_i)$, where $U$ is a symmetric bilinear map determined by the metric, and this drift term vanishes under either hypothesis.

\paragraph{Pathwise approximation.} The construction~\eqref{eq:intro_sde} is tailor-made for simulation. Approximating the driving Brownian motion by its piecewise linear interpolation with step size $h$, i.e., applying a Wong--Zakai approximation, turns~\eqref{eq:intro_sde} into an ODE with left-invariant, piecewise constant coefficients, whose solution is explicit in terms of the Lie group exponential:
\begin{equation}\label{eq:intro_WZ}
    X_h(kh) = \pi\left(\exp\Big(\sum_{i=1}^r \Delta W_i^{1,h} e_i\Big)\cdots \exp\Big(\sum_{i=1}^r \Delta W_i^{k,h} e_i\Big)\right),\qquad k\in\N,
\end{equation}
where $\Delta W_i^{k,h}=W_i(kh)-W_i((k-1)h)$; see~\eqref{eq:prewongzakai}. Note that the scheme is \emph{structure preserving} and requires no projection step (the `projection' $\pi\colon G\rightarrow G/H$ is obtained by simply letting the argument act on the base point $o\in M$, see also Remark~\ref{rem:projection_on_M}). 

On the group $G$ itself, using the product of exponentials in~\eqref{eq:intro_WZ} to approximate a Brownian motion is classical: it is McKean's `injection' construction of Brownian motion on a Lie group~\cite[Section 4.8]{McKean:1969}, going back to~\cite{McKean:1960, Gangolli:1964}, in which the process is \emph{defined} as the limit of such products along dyadic partitions; in the numerical literature it appears as the geometric Euler--Maruyama scheme for SDEs in matrix Lie groups~\cite{PiggottSolo:2016, MarjanovicSolo:2018}. What is new here is the passage to the homogeneous space: for every $G$-invariant metric on $M\cong G/H$, the projection of the left-invariant diffusion~\eqref{eq:intro_sde} is the Riemannian Brownian motion of that metric, the projected scheme~\eqref{eq:intro_WZ} is a consistent, structure-preserving approximation of it, and for naturally reductive metrics it is precisely a geodesic random walk on $M$.

Our second contribution is the observation that the strong and almost sure convergence theory for Wong--Zakai approximations in the Euclidean setting transfers to the setting of reductive homogeneous spaces, allowing us to establish convergence rates for the proposed approximation scheme~\eqref{eq:intro_WZ}. More specifically, we first show that the classical $L^p$ and almost sure rates of order $h^{1/2-\varepsilon}$ for Wong--Zakai approximations of Stratonovich SDEs in $\R^N$~\cite{shmatkov2006rate,FrizRiedel:2011,BayerEtAl:2016} carry over to SDEs on arbitrary Riemannian manifolds via Nash's embedding theorem (Corollaries~\ref{cor:WZ rates compact RM} and~\ref{cor:WZ as rate RM}), and hence, since $\pi$ is a Riemannian submersion for a suitable left-invariant metric on $G$ (Lemma~\ref{lem:submersion}), to the approximation~\eqref{eq:intro_WZ} of the Riemannian Brownian motion on $M=G/H$ (Corollary~\ref{cor:WZ_rates_homogeneous}). When considering $M=G/H$ with $G$ compact we obtain $L^p$ rates of order $\frac{1}{2}-\varepsilon$; for non-compact $G$ we obtain almost sure rates of order $\frac{1}{2}-\varepsilon$ along any sequence $h_n\to 0$, $L^p$ rates being unavailable in general in this setting. Such convergence rates for discretisations of SDEs on a Lie group are not new: an almost sure rate of order $h^{1/2-\varepsilon}$ along the dyadic partitions $h_n=2^{-n}$, in the norm of a faithful matrix representation, is already contained in~\cite[Section 4.8]{McKean:1969}; $L^2$ rates of order $\tfrac12$ for~\eqref{eq:intro_WZ} on $\operatorname{SO}(n)$ and $\operatorname{SE}(n)$ were obtained in~\cite{PiggottSolo:2016}. Other related convergence results are discussed in the related work paragraph below. The Wong--Zakai route we take recovers and slightly extends these results; moreover, it is applicable in general Riemannian manifolds (see Section~\ref{ssec:WZ_RM}).

\paragraph{Examples.} Our third contribution lies in the fact that we work out the details of the construction of the Riemannian Brownian motion for spheres, hyperbolic spaces, symmetric positive-definite matrices, flag manifolds and Stiefel manifolds, in increasing order of geometric complexity (see Section~\ref{sec:examples}). For each example we identify the reductive decomposition $\mathfrak{g} = \mathfrak{h} \oplus \mathfrak{m}$ and decompose the isotropy representation of $H$ on $\mathfrak{m}$ into irreducibles. By Schur’s lemma this yields the full family of $G$-invariant metrics: a one-parameter family for spheres and hyperbolic spaces, a two-parameter family for symmetric positive-definite matrices (with $G = \mathrm{GL}(n)$), and larger families of anisotropic metrics for flag and Stiefel manifolds. For these last two, the resulting Riemannian Brownian motions differ from those induced by the normal metric. We provide explicit orthonormal bases of $\mathfrak{m}$ for these metrics, ready for implementation, and illustrate the numerical schemes with simulations of trajectories and of the associated (geodesic or extrinsic) distance processes.

\paragraph{Relation to existing work.}
To the best of our knowledge, we are the first to systematically exploit the framework of reductive homogeneous spaces to construct and analyse numerical schemes for SDEs on such spaces. Closest in spirit to our approach is~\cite{JensenJoshiSommer:2022}, where Brownian motion on a Lie group with a left-invariant metric is simulated and pushed forward to a homogeneous space $G/H$ in order to simulate bridges and estimate the metric from data (with the sphere $S^2$ and the $3\times 3$ positive-definite matrices as examples). However, that work does not address numerical approximations or the question for which metrics on $G/H$ this pushforward is itself a Riemannian Brownian motion. \par 
Regarding the analysis of  numerical schemes for stochastic differential equations on Riemannian manifolds, until recently most contributions have concentrated on weak convergence (rates), some of which have already been mentioned above. In~\cite{bharath2023sampling}, weak error rates of order one are obtained for geodesic (and retraction-based) discretisations of the Langevin diffusion on compact manifolds. In~\cite{bronasco2025high} higher weak order integrators are constructed for stochastic differential equations on Riemannian manifolds based on frozen flow methods, and in~\cite{LaurentVilmart:2022} order conditions are derived for sampling the invariant measure of ergodic SDEs on manifolds. \par Strong rates of order $\tfrac12$ for the geometric Euler--Maruyama scheme on matrix Lie groups (in essence the same scheme we consider) were obtained in~\cite{PiggottSolo:2016}; higher-order schemes, which add L\'evy-area corrections involving the brackets $[e_i,e_j]$ and are no longer of Wong--Zakai type, are obtained by truncating the exponential Lie series further~\cite{CastellGaines:1995, MalhamWiese:2008} or by Milstein-type corrections~\cite{MunizEhrhardtGuentherWinkler:2022, WangSolo:2026milstein}. Strong convergence rates for intrinsic schemes on general Riemannian manifolds have appeared only very recently: in~\cite{wang2026geometry} strong $L^2$ error rates of order $h^{(1-\varepsilon)/2}$ are proved for an exponential Euler--Maruyama scheme for stochastic differential equations on \emph{embedded} Riemannian manifolds, and in~\cite{ZhanSugiyama:2026} strong $L^p$ error rates of order $\tfrac12$ are established for a geometric Euler--Maruyama scheme on \emph{closed embedded submanifolds with bounded extrinsic geometry} -- note however that what the authors call ``geometric Euler--Maruyama'' in this paper is a scheme that involves the \emph{Riemannian} exponential map, whereas~\eqref{eq:intro_WZ} requires the Lie group exponential of $G$. In particular, the schemes coincide on naturally reductive homogeneous spaces, but not e.g.\ in the flag and Stiefel manifolds, when endowed with and anisotropic metric.

\paragraph{Outline and reading guide.} Section~\ref{sec:notation} fixes notation and collects preliminaries. Sections~\ref{ssec:HS} and~\ref{ssec:nrHS} introduce reductive and naturally reductive homogeneous spaces and recall Nomizu's description of invariant connections. In Section~\ref{ssec:RM_natredHS} it is proven that the projection of~\eqref{eq:intro_sde} is a Riemannian Brownian motion on a reductive homogeneous space (under aforementioned assumptions). Section~\ref{ssec:WZ_RM} establishes the convergence rates of Wong--Zakai approximations, first on general Riemannian manifolds and then for the Riemannian Brownian motion on reductive homogeneous spaces, and derives the explicit representation~\eqref{eq:intro_WZ}. Section~\ref{sec:examples} contains the examples. The application-oriented reader who is primarily interested in the simulation algorithm may skip directly to Section~\ref{ssec:WZ_LG_HS} and to the examples in Section~\ref{sec:examples}, referring back to Proposition~\ref{prop:unimodularredhombm} for the conditions under which the algorithm produces a Riemannian Brownian motion.

\section{Notation and preliminaries}\label{sec:notation}
We take $\N=\{1,2,\ldots\}$ and $\N_0 = \N\cup \{0\}$. For $k\in \N$, we denote by $I_k$ the $k\times k$ identity matrix.

We generally follow~\cite{knapp2002lie} regarding terminology and notation in the context of Lie groups, in particular, the key results provided below can be found in that book. Let $G$ be a Lie group with $e$ its identity element and $\mathfrak{g}$ its Lie algebra. We denote the left translation by $g\in G$ on $G$ by $L_g$ and the right translation by $R_g$, i.e., $L_g\colon G \rightarrow G$, $L_g h = gh$ and $R_g \colon G \rightarrow G$, $R_g h = hg$. Recall that the adjoint representation $\operatorname{Ad}\colon G \rightarrow \operatorname{GL}(\mathfrak{g})$ is given by $\operatorname{Ad}_g = \textnormal{d}(L_g\circ R_{g^{-1}})_e$, $g\in G$. Moreover, $\operatorname{ad}\colon \mathfrak{g}\rightarrow \operatorname{End}(\mathfrak{g})$ is given by $\operatorname{ad}_X = (\frac{\bm{d}}{\bm{d}t} \operatorname{Ad}_{\exp(t X)})_{t=0}$ and satisfies $\operatorname{ad}_X Y = [X,Y]$, $X,Y\in \mathfrak{g}$. For any Lie group $G$, a left Haar measure is a nonzero Borel measure that is invariant under left translations; right Haar measures are defined analogously. A Haar measure is unique up to a multiplicative constant (\cite[Theorem 8.23]{knapp2002lie}). Moreover, left and right Haar measures are related by the modular function, which is given in terms of the adjoint representation of $G$ on $\mathfrak{g}$. A Lie group is \emph{unimodular} if its left Haar measure coincides with its right Haar measure. The following proposition provides a characterisation of unimodular Lie groups in terms of the adjoint representation (see~\cite[Proposition 8.27]{knapp2002lie} for the first statement, the second follows by differentiation):

\begin{proposition}\label{prop:unimodular_Ad}
    A Lie group $G$ is \emph{unimodular} if and only if $|\det\operatorname{Ad}_g|=1$ for all $g\in G$. For a connected Lie group, this is equivalent to $\operatorname{tr}\operatorname{ad}_X =0$ for all $X\in\mathfrak{g}$.
\end{proposition}

Unimodular Lie groups include all abelian, compact, semisimple, reductive and nilpotent Lie groups, see e.g. \cite[Corollary 8.31]{knapp2002lie}. Note that a reductive Lie group is \emph{not} the same thing as a reductive homogeneous space, see Remark~\ref{remark:reductive Lie group}.

We generally follow~\cite{lee2012introductionsmooth, lee2018introduction} regarding terminology and notation in the context of smooth and Riemannian manifolds. In particular, \emph{smooth} means `of class $C^{\infty}$', and for $M$ a smooth manifold, the vector space of all smooth vector fields on $M$ is denoted by $\mathfrak{X}(M)$. Throughout our paper, we consider manifolds without a boundary. The $N$-dimensional Euclidean space, when interpreted as a Riemannian manifold, is denoted by $(\R^N, \bm{g}_{\textnormal{Eucl}})$ (in this paper, we use the bold font $\bm{g}$ for metrics in order to distinguish them from group elements). For $M$ and $N$ smooth manifolds and $U\subseteq M$ and $V\subseteq N$ open, a smooth map $\phi\colon U \rightarrow V$ is called a \emph{diffeomorphism} if it is a bijection and its inverse is also smooth.

Given a filtered probability space $(\Omega,\mathcal{F},(\mathcal{F}_t)_{t\geq 0},\P)$, a stopping time $\tau \colon \Omega \rightarrow [0,\infty]$, and a measurable space $(A,\mathcal{A})$, we say that $(X(t))_{t\in [0,\tau)}$ is a \emph{stochastic process} if, for all $t\geq 0$, $X(t) \colon \{ t <\tau \} \rightarrow A$ is a random variable with respect to the trace $\sigma$-algebra $\mathcal{F}|_{\{t<\tau\}}$ (and we call the process \emph{adapted} if $X(t)$ is $\mathcal{F}_t|_{\{t<\tau\}}$-measurable for all $t\geq 0$). A stochastic process $(X(t))_{t\in [0,\tau]}$ is defined analogously. We say that $(X(t))_{t\in [0,\tau)}$ has continuous sample paths if the mapping $[0,\tau(\omega))\ni t \rightarrow X(t,\omega)$ is continuous for all $\omega \in \Omega$. Note that wherever there is no risk of confusion, the dependence on $\Omega$ is omitted, i.e., we typically write $X(t)$ and not $X(t,\omega)$.\par

In Sections~\ref{ssec:RM_natredHS} and~\ref{ssec:WZ_RM} below we will need some extension results for vector fields defined on a submanifold of $\R^N$. While the existence of these extensions is entirely standard, we provide the exact statement and proof for the readers' convenience.

\begin{lemma}\label{lem:extension}
Let $n,N\in \N$ with $n < N$, let $M$ be a smooth, properly embedded,
$n$-dimensional submanifold of $(\R^N,\bm{g}_{\textnormal{Eucl}})$,
and let $X\in \mathfrak{X}(M)$. Then there exists an
$\hat{X}\in \mathfrak{X}(\R^N)$ such that $\hat{X}|_{M}=X$. Moreover, for every
compact set $K\subseteq \R^N$ there exists an $\hat{X}_K\in \mathfrak{X}(\R^N)$
with compact support such that $\hat{X}_K|_K = \hat{X}|_K$. In particular, if $M$ is compact, then $\hat{X}$ can be constructed to
have compact support.
\end{lemma}

\begin{proof}
Recall that a properly embedded submanifold of $\R^N$ is closed in $\R^N$
(see, e.g.,~\cite[Proposition 5.5]{lee2012introductionsmooth}).

First, observe that it suffices to provide an open set $U \subseteq \R^N$
containing $M$ and a smooth extension $X_U\colon U\rightarrow \R^N$ of $X$.
Indeed, since $M$ is closed in $\R^N$, by
e.g.~\cite[Proposition 2.25]{lee2012introductionsmooth} there exists a
$\psi \in C^\infty(\R^N)$ with $\psi \equiv 1$ on $M$ and
$\operatorname{supp}(\psi) \subseteq U$; setting $\hat{X} = \psi \cdot X_U$ on
$U$ and $\hat{X}=0$ outside $U$ yields the desired extension. Given moreover a
compact set $K\subseteq \R^N$, pick $\chi \in C_c^{\infty}(\R^N)$ with
$\chi \equiv 1$ on $K$ and set
$\hat{X}_K = \chi \cdot \hat{X}$: this vector field has compact support and
agrees with $\hat{X}$ on $K$. If $M$ is compact, then
taking $K=M$ shows $\hat{X}$ can be constructed to have compact support.

Next, recall that by~\cite[Theorem 5.8]{lee2012introductionsmooth}, for every
$x \in M$ there exist an open neighbourhood $U_x \subseteq \R^N$ of $x$, an
open set $V_x\subseteq \R^N$, and a diffeomorphism
$\phi_x \colon U_x \rightarrow V_x$ such that
\begin{equation}
\phi_x(M \cap U_x) = \{ y \in V_x \colon y_{n+1} = \dots = y_N = 0 \};
\end{equation}
the pair $(U_x, \phi_x)$ is called a \emph{smooth slice coordinate chart for
$M$ at $x$}. Each smooth slice coordinate chart provides a local smooth
extension $\hat{X}_x \in \mathfrak{X}(U_x)$ of $X|_{M \cap U_x}$: regarding $X$
as a smooth map $M\rightarrow \R^N$, in chart
coordinates one extends by holding the last $N-n$ coordinates constant, i.e.,
$\hat{X}_x = (X|_{M\cap U_x})\circ \phi_x^{-1} \circ P \circ \phi_x$, where
$P\colon \R^N \rightarrow \R^N$ is given by
$P(y)=(y_1,\ldots,y_n,0,\ldots,0)$ (shrinking $U_x$ if necessary so that
$P(V_x)\subseteq V_x$).

Since $M$ is second countable, it is $\sigma$-compact, so there exist
$x_1,x_2,\ldots \in M$ such that $M \subseteq \cup_{i\in \N} U_{x_i}$. Set
$U = \cup_{i\in \N} U_{x_i}$ and let $\{\rho_i\}_{i\in \N}$ be a smooth
partition of unity on $U$ subordinate to the cover $(U_{x_i})_{i\in \N}$. The
vector field
\begin{equation*}
X_U(y) = \sum_{i\in \N} \rho_i(y) \hat{X}_{x_i}(y), \quad y \in U,
\end{equation*}
defines a smooth extension on $U$ satisfying $X_U|_M = X$.
\end{proof}

\section{The Riemannian Brownian motion on a reductive homogeneous space}\label{sec:redhomspaces}
Our setting is that of a \emph{reductive homogeneous space}, which is a manifold that allows for a high degree of symmetry (it generalises the concept of a Lie group). The goal of this section is to provide a brief introduction to reductive homogeneous spaces and to demonstrate that they allow for a natural and elegant representation of the Riemannian Brownian motion. More specifically, we begin by introducing reductive homogeneous spaces in Section~\ref{ssec:HS}. In Section~\ref{ssec:nrHS} we discuss connections on reductive homogeneous spaces. We also identify a special class of reductive homogeneous spaces (the naturally reductive homogeneous spaces) that can be equipped with a Riemannian metric for which the Riemannian exponential coincides with a Lie group exponential. Finally, in Section~\ref{ssec:RM_natredHS}, we construct the Riemannian Brownian motion on a Riemannian reductive homogeneous space.

\subsection{Reductive homogeneous spaces}\label{ssec:HS}
A \emph{homogeneous space} is a smooth manifold $M$ that is acted upon smoothly by a Lie group $G$ in a transitive manner. We will discuss this definition in more detail below, for a comprehensive discussion of homogeneous spaces see e.g.~\cite[Section 10.1]{hilgert2011structure} and~\cite[Chapter 21]{lee2012introductionsmooth}.

Let $M$ be a smooth manifold and let $G$ be a Lie group that acts smoothly on $M$; we denote this action by $\Phi \colon G \times M \rightarrow M$. For the definition of a smooth action of $G$ on $M$ we refer to~\cite[Definition 9.1.11]{hilgert2011structure}. Throughout this paper, we assume $\Phi$ to be a left action, i.e., $\Phi_{gh} = \Phi_g \circ \Phi_h$, $g,h\in G$, though an appropriate analogue of all results presented here holds for right actions. \emph{Transitivity} of the action means that for every pair of points $x,y\in M$, there exists a group element $g\in G$ such that $y = \Phi_g(x)$.

As the transitivity of the action ensures that every point can be reached from every other point, we can pick any base point $o\in M$ to be our distinguished `point of reference'. The \emph{stabiliser group} or \emph{isotropy group} is then defined to be the (closed) Lie subgroup
\begin{equation}
H =\{g\in G \,|\, \Phi_g( o)=o\}.
\end{equation}
Note that the isotropy group depends on the base point $o\in M$ that is chosen. The Lie algebra $\mathfrak{h}$ of $H$ is called the \emph{isotropy algebra}.

The fact that the action is transitive implies that $G/H$ is diffeomorphic to $M$ under the mapping $gH \mapsto \Phi_g(o)$; \cite[Theorem 3.62]{warner1983foundations}, see also Proposition~\ref{prop:HPviaH} below. Note that if the action $\Phi$ is free, then $H$ is trivial and it immediately follows that $M$ is diffeomorphic to $G$. In particular, every Lie group is itself a homogeneous space, albeit a trivial one. Note moreover that $G/H$ is itself a group if and only if $H$ is normal, see \cite[Chapter 1, Section 3]{lang2012algebra} for abstract groups $G$ and $H$ and see \cite[Theorem 3.64]{warner1983foundations} for Lie groups $G$ and $H$.

The above construction can also be reversed in a sense, as the following proposition demonstrates (see, e.g.~\cite[Theorem 10.1.10]{hilgert2011structure} or~\cite[Theorem 21.17]{lee2012introductionsmooth}):

\begin{proposition}\label{prop:HPviaH}
Let $G$ be a Lie group and let $H\subseteq G$ be a closed subgroup. Then the coset space $G/H$, endowed with the quotient topology, carries a natural manifold structure for which the quotient map $\pi \colon G \rightarrow G/H$, $g\mapsto gH$ is a submersion. Moreover, $\Phi \colon G \times G/H \rightarrow G/H$, $(g_1,g_2H)\mapsto g_1g_2H$ defines a smooth transitive action of $G$ on $G/H$.
\end{proposition}

Note that under the identification $M\cong G/H$ one has $\Phi_{\exp(tX)}(o)=\pi(\exp(tX))$ for $X\in \mathfrak{g}$, $t\in\R$, so that $\textnormal{d}\pi_e\colon \mathfrak{g}\rightarrow T_o(G/H)$ is given by $\textnormal{d}\pi_e X = \frac{\bm{d}}{\bm{d}t}\big|_{t=0}\Phi_{\exp(tX)}(o)$. As $\pi$ is a submersion and $\pi(\exp(tY))=o$ for $Y\in \mathfrak{h}$, the kernel of $\textnormal{d}\pi_e$ is $\mathfrak{h}$, so that the restriction of $\textnormal{d}\pi_e$ to \emph{any} linear complement of $\mathfrak{h}$ in $\mathfrak{g}$ is an isomorphism onto $T_oM$.

A \emph{reductive} homogeneous space is a homogeneous space for which we can identify an $\operatorname{Ad}_H$-invariant subspace of $\mathfrak{g}$ that is a linear complement of $\mathfrak{h}$:

\begin{definition}\label{def:rhs}We call a homogeneous space $M\cong G/H$ a \emph{reductive homogeneous space} (see \cite[Definition 23.8]{gallier2020differential}) if the Lie algebra $\mathfrak{g}$ corresponding to $G$ can be split as
\begin{equation}\label{eq:reductive}
    \mathfrak{g}= \mathfrak{h}\oplus \mathfrak{m},
\end{equation}
where $\mathfrak{m}$ is invariant under $\operatorname{Ad}_h$ for all $h\in H$, i.e.,
\begin{equation}\label{eq:stability}
\operatorname{Ad}_h (\mathfrak{m})\subseteq\mathfrak{m},\quad h\in H.
\end{equation}
\end{definition}

By the discussion preceding the above definition, for a reductive homogeneous space $M\cong G/H$ with reductive decomposition $\mathfrak{g}=\mathfrak{h}\oplus \mathfrak{m}$ the map $\textnormal{d}\pi_e$ restricts to a natural isomorphism
\begin{equation}\label{eq:link_m_ToM}
    \mathfrak{m}\ni X \mapsto X^*(o) := \textnormal{d}\pi_e X = \frac{\bm{d}}{\bm{d}t}\Big|_{t=0}\Phi_{\exp(t X)}( o) \in T_oM
\end{equation}
(see also \cite[p.\ 695]{gallier2020differential}). Following~\cite[Section 7]{nomizu1954invariant}, we extend $X^*(o)$ to a vector field near $o$ by means of the group action: by the inverse function theorem there exists an open neighbourhood $V\subseteq \mathfrak{m}$ of $0$ such that $V\ni Z\mapsto \pi(\exp(Z))$ is a diffeomorphism onto an open neighbourhood $N^*$ of $o$, and for $X\in \mathfrak{m}$ we define $X^*\in \mathfrak{X}(N^*)$ by
\begin{equation}\label{eq:local_frame}
    X^*\big(\pi(\exp(Z))\big) := \textnormal{d}(\Phi_{\exp(Z)})_o\, X^*(o),\quad Z\in V.
\end{equation}
Note that $X^*$ is \emph{not} the fundamental vector field $x\mapsto \frac{\bm{d}}{\bm{d}t}|_{t=0}\Phi_{\exp(tX)}(x)$ induced by the action (the two coincide at $o$ and along the curve $t\mapsto \pi(\exp(tX))$, but not on all of $N^*$); in the case that $M=G$ is a Lie group acting on itself by left translation, $X^*$ is the \emph{left}-invariant vector field determined by $X$, whereas the fundamental vector field is the right-invariant one. If $e_1,\ldots,e_r$ is a basis of $\mathfrak{m}$, then $e_1^*,\ldots,e_r^*$ is a local frame on $N^*$, and for $X,Y\in \mathfrak{m}$ one has (see \cite[(7.3)]{nomizu1954invariant})
\begin{equation}\label{eq:bracket_local_frame}
    [X^*,Y^*](o) = ([X,Y]_{\mathfrak{m}})^*(o),
\end{equation}
where $[X,Y]_{\mathfrak{m}}$ denotes the $\mathfrak{m}$-component of $[X,Y]$ with respect to the decomposition $\mathfrak{g}=\mathfrak{h}\oplus\mathfrak{m}$. Moreover, for $X\in \mathfrak{m}$ the curve $t\mapsto \pi(\exp(tX))$ is (for $tX\in V$) an integral curve of $X^*$, as $X^*(\pi(\exp(tX))) = \textnormal{d}(\Phi_{\exp(tX)})_o \frac{\bm{d}}{\bm{d}s}|_{s=0}\pi(\exp(sX)) = \frac{\bm{d}}{\bm{d}s}|_{s=0}\pi(\exp((t+s)X))$.

Condition \eqref{eq:stability} is also referred to as the stability of $\mathfrak{m}$ under the adjoint action of $H$. Equivalently, the adjoint representation of $H$ restricts to a representation of $H$ on $\mathfrak{m}$, the canonical \emph{isotropy representation}
\begin{equation}\label{eq:isotropyrep}
     \rho:H\to \operatorname{GL}(\mathfrak{m}), \quad \rho(h) X = \operatorname{Ad}_h X.
\end{equation}
It is important to note that the subspace $\mathfrak{m}$ in the decomposition~\eqref{eq:reductive} need not be a Lie subalgebra. Moreover, observe that the $\operatorname{Ad}_H$-invariance of $\mathfrak{m}$ implies, by differentiating $\operatorname{Ad}_{\exp(tY)}(\mathfrak{m})\subseteq \mathfrak{m}$ at $t=0$, that $[\mathfrak{h},\mathfrak{m}]\subseteq \mathfrak{m}$. Conversely, if $\mathfrak{m}$ is invariant under $\operatorname{ad}_Y$ for all $Y\in \mathfrak{h}$, then $\mathfrak{m}$ is invariant under $\operatorname{Ad}_{\exp(Y)}=\exp(\operatorname{ad}_Y)$ for all $Y\in \mathfrak{h}$, which is equivalent to $\mathfrak{m}$ being invariant under $\operatorname{Ad}_{H_0}$, where $H_0$ is the identity component of $H$ (as $H_0$ is generated by $\exp(\mathfrak{h})$). In particular, we have:

\begin{lemma}\label{lem:connected_H_reductive} Let $M\cong G/H$ be a homogeneous space with $H$ connected, and let $\mathfrak{m}\subseteq \mathfrak{g}$ be a linear subspace such that $\mathfrak{g}=\mathfrak{h}\oplus \mathfrak{m}$ and $[\mathfrak{h},\mathfrak{m}]\subseteq \mathfrak{m}$. Then $M$ is a reductive homogeneous space with reductive decomposition $\mathfrak{g}=\mathfrak{h}\oplus \mathfrak{m}$.
\end{lemma}

Another well-known sufficient condition is the following (see~\cite[p.\ 41]{nomizu1954invariant}):

\begin{lemma}\label{prop:properaction_redhomospace}
A homogeneous space $M\cong G/H$ whose isotropy group $H$ is compact is a reductive homogeneous space.
\end{lemma}

\begin{proof}
    Let $\langle \cdot,\cdot\rangle_0 \colon \mathfrak{g}\times\mathfrak{g}\to\mathbb{R}$ be an arbitrary inner product. By the compactness of $H$, we can average over $H$ to obtain the $H$-invariant inner product:
    \begin{equation}
        \langle X, Y\rangle := \int_H \langle \operatorname{Ad}_h X, \operatorname{Ad}_h Y\rangle_0 \, \mu(dh),\quad X,Y\in \mathfrak{g},
    \end{equation}
    where $\mu$ is the (finite, bi-invariant) Haar measure on the compact group $H$ and $\operatorname{Ad}\colon G\to\mathrm{GL}(\mathfrak{g})$ is the adjoint representation (see Section~\ref{sec:notation}). The subspace $\mathfrak{h} \subseteq \mathfrak{g}$ is naturally invariant under $\operatorname{Ad}|_H$. Since the inner product $\langle\,\cdot\,,\,\cdot\,\rangle$ is $H$-invariant, it follows that the orthogonal complement $\mathfrak{m}=\mathfrak{h}^\perp$ is also invariant under $\operatorname{Ad}|_H$, that is, $\operatorname{Ad}_h(\mathfrak{m})\subseteq \mathfrak{m}$ for all $h\in H$.
\end{proof}

\begin{remark}
Recall that a group action $\Phi \colon G \times M \to M$ is said to be \emph{proper} if the map $G\times M\to M\times M$, $(g,x)\mapsto(x,\Phi_g(x))$ is a proper map; that is, the preimage of any compact set is compact (see e.g.~\cite[p.542]{lee2012introductionsmooth}). In particular, the isotropy group $H$ of a proper action is compact, and thus the homogeneous space $M\cong G/H$ is reductive by Lemma~\ref{prop:properaction_redhomospace}. Proper actions are discussed in~\cite[Chapter 21]{lee2012introductionsmooth}.
\end{remark}

\begin{remark}\label{remark:reductive Lie group}
Note that every Lie group is (trivially) a reductive homogeneous space. In particular, the concept of a reductive homogeneous space should not be confused with the concept of a reductive Lie group: the latter is a Lie group whose Lie algebra admits the splitting
\begin{equation}
    \mathfrak{g} = \mathfrak{s}\oplus\mathfrak{a},
\end{equation}
where $\mathfrak{s}$ is semisimple and $\mathfrak{a}$ is abelian.
\end{remark}

\subsection{Connections and naturally reductive homogeneous spaces}\label{ssec:nrHS}
The goal of our work is to efficiently construct a \emph{Riemannian} Brownian motion on a reductive homogeneous space $M\cong G/H$ by exploiting the reductive structure. Thus we need to delve into connections on reductive homogeneous spaces. Specifically, we begin by recalling that $G$-invariant connections can be described by so-called connection functions. Next, we identify the connection function that corresponds to the Levi-Civita connection arising from a $G$-invariant metric. Finally, we identify a special class of reductive homogeneous spaces known as \emph{naturally reductive homogeneous spaces}; these are the spaces for which the Riemannian exponential map is directly linked to the Lie group exponential associated with $G$.

\paragraph{Connections on reductive homogeneous spaces} Let $M\cong G/H$ be a reductive homogeneous space with base point $o$, let $\Phi\colon G\times M\rightarrow M$ denote the associated group action, and let $\mathfrak{g}=\mathfrak{h}\oplus \mathfrak{m}$ be the associated reductive decomposition of the Lie algebra associated with $G$ (see Definition~\ref{def:rhs}). Recall that $\pi \colon G \rightarrow G/H$ denotes the natural projection and that $\mathfrak{m}$ is identified with $T_oM$ via $X\mapsto X^*(o)$, see~\eqref{eq:link_m_ToM}. Note that for $t\in\R,\ X\in\mathfrak{m}$ and $h\in H$ one has
\begin{equation}
\Phi_h\big(\pi(\exp(tX))\big)
  = \pi\big(h\exp(tX)\big)
  = \pi\big(h\exp(tX)h^{-1}\cdot h\big)
  = \pi\big(h\exp(tX)h^{-1}\big)
  = \pi\big(\exp(t\operatorname{Ad}_hX)\big).
\end{equation}
Differentiating at $t=0$, we obtain
\begin{equation}\label{eq:dPhi_h_Ad}
(\textnormal{d}\Phi_h)_o( X^*(o) ) = (\operatorname{Ad}_h X)^*(o),\quad X\in \mathfrak{m},\ h\in H.
\end{equation}
As a consequence, there is a one-to-one correspondence between $G$-invariant mappings on $TM$ and $\operatorname{Ad}_H$-invariant mappings on $\mathfrak{m}$. More explicitly, if $\alpha \colon \mathfrak{m}\times \mathfrak{m}\rightarrow \mathfrak{m}$ is $\textrm{Ad}_H$-invariant, i.e., if
\begin{equation}
\alpha(\textrm{Ad}_hX,\textrm{Ad}_h Y) = \textrm{Ad}_h\alpha(X,Y),\quad X,Y\in \mathfrak{m},\, h\in H,
\end{equation}
then we can define a $G$-invariant mapping $\tilde{\alpha}\colon TM \times TM \rightarrow TM$ by setting
\begin{equation}
\tilde{\alpha}_o(X^*(o),Y^*(o)) = \alpha(X,Y)^*(o), \quad X,Y\in \mathfrak{m},
\end{equation}
and extending by
\begin{equation}
\tilde{\alpha}_{gH}(X,Y)=\textnormal{d}\Phi_{g} \Big(\tilde{\alpha}_o(\text{d}\Phi_{g^{-1}} X,\text{d}\Phi_{g^{-1}} Y)\Big),\quad X,Y\in T_{gH}M.
\end{equation}
This extension is well-defined as $\alpha$ is $\textrm{Ad}_H$-invariant, see also \cite[Theorem 8.1]{nomizu1954invariant}. An $\textrm{Ad}_H$-invariant bilinear function $\alpha\colon \mathfrak{m}\times\mathfrak{m} \rightarrow \mathfrak{m}$ is called a \emph{connection function}. Since a connection is not a tensor field, the above correspondence does not directly apply to connections; nevertheless, by \cite[Theorem 8.1]{nomizu1954invariant} there exists a one-to-one correspondence between the set of $G$-invariant connections $\nabla$ on $TM$ and the set of connection functions, given by
\begin{equation}\label{eq:connection_function}
    \alpha(X,Y)^*(o) = (\nabla_{X^*}Y^*)(o),\quad X,Y\in \mathfrak{m},
\end{equation}
where $X^*,Y^*$ are the local vector fields defined in~\eqref{eq:local_frame}.

Torsion and curvature of $G$-invariant connections can be given purely in terms of operations and functions on $\mathfrak{m}$: recall that for $X,Y\in \mathfrak{g}$ we denote by $[X,Y]_{\mathfrak{h}}\in \mathfrak{h}$ and $[X,Y]_{\mathfrak{m}}\in \mathfrak{m}$ the components of the Lie bracket $[X,Y]_{\mathfrak{g}}$, i.e.,
\begin{equation}
    [X,Y]_{\mathfrak{g}} = [X,Y]_{\mathfrak{h}}+[X,Y]_{\mathfrak{m}}.
\end{equation}
It was shown in~\cite[equations (9.1) and (9.6)]{nomizu1954invariant} (using~\eqref{eq:connection_function} and~\eqref{eq:bracket_local_frame}) that the torsion $T$ of the $G$-invariant connection with connection function $\alpha$ is given by
\begin{equation}\label{eq:torsion}
  T(X^*,Y^*)(o) = \big(\alpha(X,Y) - \alpha(Y,X) - [X,Y]_{\mathfrak{m}}\big)^*(o),\quad X,Y \in \mathfrak{m},
\end{equation}
and the curvature tensor field $R$ can be expressed as
\begin{equation}\label{eq:curvature}
    (R(X^*,Y^*)\cdot Z^*)(o) = \big(\alpha(X,\alpha(Y,Z))-\alpha(Y,\alpha(X,Z)) - \alpha([X,Y]_\mathfrak{m},Z) - [[X,Y]_\mathfrak{h},Z]\big)^*(o),\quad X,Y,Z\in\mathfrak{m}.
\end{equation}
Henceforth, by invariance, we drop the notation $(R(X^*,Y^*)\cdot Z^*)(o)$ in favour of $R(X,Y)\cdot Z\in \mathfrak{m}$ with $X,Y,Z\in\mathfrak{m}$ and similarly for the torsion.

\paragraph{Metric connections on reductive homogeneous spaces}
Let $M\cong G/H$ be a reductive homogeneous space. Again identifying $T_o M$ with $\mathfrak{m}$ we see that there is a one-to-one correspondence between the set of $\operatorname{Ad}_H$-invariant inner products $B:\mathfrak{m}\times\mathfrak{m}\to\mathbb{R}$ and the set of $G$-invariant Riemannian metrics on $M$. Explicitly, given an $\operatorname{Ad}_H$-invariant inner product $B$ on $\mathfrak{m}$, the \emph{$G$-invariant Riemannian metric $\bm{g}_B$ induced by $B$} is defined by
\begin{equation}\label{eq:gB}
    \bm{g}_B\big(\Phi_g(o);\, \textnormal{d}(\Phi_g)_o X^*(o),\, \textnormal{d}(\Phi_g)_o Y^*(o)\big) := B(X,Y),\quad X,Y\in \mathfrak{m},\ g\in G,
\end{equation}
where we write $\bm{g}_B(x;V,W)$ for the inner product of $V,W\in T_xM$. As $\Phi_g\colon M\rightarrow M$ is a diffeomorphism and $X\mapsto X^*(o)$ is an isomorphism $\mathfrak{m}\rightarrow T_oM$, every pair of tangent vectors at $\Phi_g(o)$ is of the form appearing in~\eqref{eq:gB}. The right-hand side of~\eqref{eq:gB} does not depend on the choice of $g$: if $\Phi_g(o)=\Phi_{g'}(o)$, then $h:=g^{-1}g'\in H$ and $\textnormal{d}(\Phi_{g'})_oX^*(o) = \textnormal{d}(\Phi_g)_o(\operatorname{Ad}_hX)^*(o)$ by~\eqref{eq:dPhi_h_Ad}, so the $\operatorname{Ad}_H$-invariance of $B$ ensures that both representatives yield the same value. The metric $\bm{g}_B$ is $G$-invariant by construction, and conversely any $G$-invariant metric $\bm{g}$ on $M$ equals $\bm{g}_B$ for the $\operatorname{Ad}_H$-invariant inner product $B(X,Y):=\bm{g}(o;X^*(o),Y^*(o))$ on $\mathfrak{m}$. In particular, if $e_1,\ldots,e_r$ is a $B$-orthonormal basis of $\mathfrak{m}$, then $\textnormal{d}(\Phi_g)_oe_1^*(o),\ldots,\textnormal{d}(\Phi_g)_oe_r^*(o)$ is a $\bm{g}_B$-orthonormal basis of $T_{\Phi_g(o)}M$ for every $g\in G$. By \cite[Theorem 13.1]{nomizu1954invariant} and \cite[Chapter X, Theorem 3.3]{kobayashi1996foundations2} (note that the book corrects a sign error in the paper), we have that the connection function associated with the Levi-Civita connection of the $G$-invariant metric $\bm{g}_B$ induced by the $\operatorname{Ad}_H$-invariant inner product $B:\mathfrak{m}\times\mathfrak{m}\to\mathbb{R}$ is given by
\begin{equation}\label{eq:alphaB}
    \alpha_B(X,Y) = \tfrac{1}{2}[X,Y]_\mathfrak{m} + U(X,Y), \quad X,Y\in \mathfrak{m},
\end{equation}
where the symmetric bilinear map $U\colon \mathfrak{m}\times \mathfrak{m}\rightarrow \mathfrak{m}$ is uniquely determined (since $B$ is non-degenerate) from the relation
\begin{equation}\label{eq:defB}
    2B(U(X,Y),Z) = B(X,[Z,Y]_\mathfrak{m}) + B([Z,X]_\mathfrak{m},Y),\quad X,Y,Z\in \mathfrak{m}.
\end{equation}

\paragraph{Naturally reductive homogeneous spaces.}
We now single out a connection function (not necessarily metric) with important further properties: by \cite[Theorem 10.1]{nomizu1954invariant}, for every reductive homogeneous space $M\cong G/H$ there exists a unique $G$-invariant connection satisfying the following two properties:\label{def:natredHS}
\begin{enumerate}
    \item $T(X,Y) = 0$ for all $X,Y\in \mathfrak{m}$;
    \item $t \mapsto \pi(\exp(t X)) $ is a geodesic for all $X\in \mathfrak{m}$ (where $\pi\colon G \rightarrow G/H \cong M$ is the natural projection and $\exp$ is the Lie group exponential);
\end{enumerate}
this connection is defined in terms of the connection function $\alpha_1\colon \mathfrak{m}\times \mathfrak{m}\rightarrow \mathfrak{m}$ satisfying
\begin{equation}
    \alpha_{1}(X,Y)=\tfrac{1}{2}[X,Y]_\mathfrak{m},\quad X,Y\in \mathfrak{m},
\end{equation}
and is called the \emph{canonical connection of the first kind}.

Note that the first property is the easy one: it follows immediately from~\eqref{eq:torsion} that the connection defined by $\alpha_1$ is torsion free (the second property follows from \cite[(10.2)]{nomizu1954invariant}, i.e., from $\alpha_1(X,X)=0$ for all $X\in \mathfrak{m}$). Moreover, since the connection is $G$-invariant, any geodesic is the $G$-translate of a geodesic through $o$, and by construction every geodesic through $o$ is of the form $t\mapsto \pi(\exp(tX))$ for some $X\in \mathfrak{m}$. Since the map $t\mapsto \exp(tX)$ is defined for all $t\in\mathbb{R}$, all geodesics of this connection are defined for all time. Thus, the canonical connection of the first kind is geodesically complete. Finally, we note that while it is possible to obtain a formula for the curvature of this connection by means of \eqref{eq:curvature}, this does not lead to a particularly nice expression.

It follows from~\eqref{eq:alphaB} and~\eqref{eq:defB} that the Levi-Civita connection of the $G$-invariant metric induced by an $\operatorname{Ad}_H$-invariant inner product $B\colon \mathfrak{m}\times \mathfrak{m}\rightarrow \R$ is the canonical connection of the first kind if and only if $U$ is identically zero. This gives rise to the following concept:

\begin{definition}\label{def: naturally reductive space}
    A reductive homogeneous space $M\cong G/H$ with decomposition $\mathfrak{g}=\mathfrak{h}\oplus \mathfrak{m}$ is called \emph{naturally reductive} if $\mathfrak{m}$ is equipped with an $\operatorname{Ad}_{H}$-invariant inner product $B:\mathfrak{m}\times\mathfrak{m}\to\mathbb{R}$ such that
        \begin{equation}\label{eq:naturallyreductive}
            B(X,[Z,Y]_\mathfrak{m}) + B([Z,X]_\mathfrak{m},Y) = 0,\quad \text{for all } X,Y,Z\in\mathfrak{m}.
        \end{equation}
\end{definition}

By the above discussion, every naturally reductive homogeneous space $M\cong G/H$ is a Riemannian \mbox{manifold} with a $G$-invariant metric $\bm{g}_B$ induced by $B$ (see~\eqref{eq:gB}), and the associated Levi-Civita connection is the canonical connection of the first kind. In particular, the Riemannian exponential map at the base point $o$ is given in terms of the Lie group exponential: $\exp_o(X^*(o)) = \pi(\exp(X))$ for all $X\in \mathfrak{m}$.

\paragraph{Examples of (naturally) reductive homogeneous spaces.}
Identity~\eqref{eq:naturallyreductive} is not easily verified in general. However, one readily identifies a situation in which this identity holds trivially, namely when $[\mathfrak{m},\mathfrak{m}] \subseteq \mathfrak{h}$,
as this implies that $[X,Y]_{\mathfrak{m}}=0$ for any $X,Y\in \mathfrak{m}$. In summary, we have:

\begin{lemma}\label{lemma: LTS}
Let $M\cong G/H$ be a reductive homogeneous space with decomposition $\mathfrak{g}=\mathfrak{h}\oplus \mathfrak{m}$ and suppose
\begin{equation}\label{eq: LTS}
[\mathfrak{m},\mathfrak{m}]\subseteq \mathfrak{h}.
\end{equation}
Then any $\operatorname{Ad}_H$-invariant inner product $B\colon \mathfrak{m}\times\mathfrak{m} \rightarrow \R$ renders $M$ naturally reductive.
\end{lemma}

This situation is not uncommon: spheres, hyperbolic spaces and symmetric positive-definite matrices can all be endowed with a reductive homogeneous structure $\mathfrak{h}\oplus \mathfrak{m}$ such that $[\mathfrak{m},\mathfrak{m}]\subseteq \mathfrak{h}$, and are thus naturally reductive homogeneous spaces (see Sections~\ref{sec:spheres}, \ref{sec:hyperbolics} and~\ref{sec:spd} below). More generally, one has $[\mathfrak{m},\mathfrak{m}] \subseteq \mathfrak{h}$ when $\mathfrak{h}$ and $\mathfrak{m}$ are the eigenspaces of an involutive automorphism of $\mathfrak{g}$ corresponding to the eigenvalues $+1$ and $-1$, respectively (see~\cite[Chapter II, Theorem 2.2]{loos1969symmetric}). Involutive automorphisms arise naturally in the context of symmetric spaces, see~\cite[Chapter II]{loos1969symmetric} for details.

\begin{remark}
When $M\cong G/H$ is a reductive homogeneous space with decomposition $\mathfrak{g}=\mathfrak{h}\oplus \mathfrak{m}$ satisfying $[\mathfrak{m},\mathfrak{m}]\subseteq \mathfrak{h}$, then the connection function $\alpha_B\equiv 0$ (see \eqref{eq:alphaB}). In particular, the connection function corresponds to the so-called \emph{canonical connection of the second kind}: the unique connection for which the parallel transport of the tangent vector $Y^*(o)$ at $o$ along the curve $t \mapsto \pi(\exp(t X)) $ is equal to the translation of $Y^*(o)$ by the 1-parameter subgroup $t\mapsto\exp(t X)$, i.e., equal to $\textnormal{d}(\Phi_{\exp(tX)})_o Y^*(o) = Y^*(\pi(\exp(tX)))$ (for all $X,Y\in \mathfrak{m}$, cf.~\eqref{eq:local_frame}), see \cite[Theorem 10.2]{nomizu1954invariant}. The canonical connection of the second kind has torsion and curvature given by
\begin{equation}
    T_2(X,Y)=-[X,Y]_\mathfrak{m}, \quad R_2(X,Y)\cdot Z = -[[X,Y]_\mathfrak{h},Z],\quad X,Y,Z\in \mathfrak{m}.
\end{equation}
Note that under the assumption $[\mathfrak{m},\mathfrak{m}]\subseteq \mathfrak{h}$ one has $T_2\equiv 0$, i.e., the canonical connections of the first and second kind coincide.
\end{remark}

We now identify one more class of naturally reductive homogeneous spaces. For this we make use of the notion of a bi-invariant metric on a Lie group. This is a left-invariant Riemannian metric on $G$ that is also right-invariant; equivalently, it is the left-invariant metric induced by an inner product $Q:\mathfrak{g}\times\mathfrak{g}\to\mathbb{R}$ that satisfies
\begin{equation}\label{eq:biinvariance}
    Q(X,Y)=Q(\operatorname{Ad}_gX,\operatorname{Ad}_gY), \quad  X,Y\in\mathfrak{g} \text{ and } g\in G.
\end{equation}
Recall from \cite[Lemma 7.5]{milnor1976curvatures} that a \emph{connected} Lie group admits a bi-invariant metric if and only if it is isomorphic to the Cartesian product of a compact group and an additive vector group.
When $M=G/H$ is a homogeneous space, where $G$ is a Lie group allowing for a bi-invariant metric, the following proposition identifies a naturally reductive structure:

\begin{proposition}\label{prop:normalhomogeneous}
    Let $G$ be a Lie group with a bi-invariant metric, let $Q$ be the associated $\operatorname{Ad}_G$-invariant inner product on $\mathfrak{g}$ (see~\eqref{eq:biinvariance}), and let $H$ be a closed subgroup of $G$. Let $M\cong G/H$ be the homogeneous space (see Proposition~\ref{prop:HPviaH}) equipped with $G$-invariant metric $\bm{g}_{B}$ induced by $B:=Q|_{\mathfrak{m}\times \mathfrak{m}}$ (see~\eqref{eq:gB}); this metric is called the \emph{normal metric} associated with $Q$. Then $(M,B)$ is a naturally reductive homogeneous space.
\end{proposition}

\begin{proof}
Note that $\mathfrak{h}$ is $\operatorname{Ad}_H$-invariant as $\mathfrak{h}$ is the Lie algebra associated with $H$. Recalling $Q$ is $\operatorname{Ad}_G$-invariant, we thus have that $\mathfrak{m}$ is $\operatorname{Ad}_H$-invariant, too, rendering $G/H$ a reductive homogeneous space. Also note that $B$ is $\operatorname{Ad}_H$-invariant by construction. It remains to verify that $G/H$ is \emph{naturally} reductive, i.e., that~\eqref{eq:naturallyreductive} is satisfied. Since the metric $Q$ is bi-invariant, it satisfies the cyclic skew-symmetry property: it follows by differentiation of \eqref{eq:biinvariance} that
    \begin{equation}\label{eq:bi-inv}
        Q([Z, X], Y) + Q(X, [Z, Y]) = 0,   \qquad X, Y, Z \in \mathfrak{g}.
    \end{equation}
Using that $Q(X, [Z,Y]_{\mathfrak{h}}) = 0$ due to orthogonality of $\mathfrak{m}$ and $\mathfrak{h}$ we thus obtain that
\begin{align*}
   B( X, [Z,Y]_{\mathfrak{m}}) & = Q(X, [Z,Y]_{\mathfrak{m}}) = Q(X, [Z,Y])
   \\
   & = Q(Y,[X,Z]) = Q(Y,[X,Z]_{\mathfrak{m}}) = B(Y, [X,Z]_{\mathfrak{m}}) = - B([Z,X]_{\mathfrak{m}},Y), \quad X,Y,Z\in \mathfrak{m}.
\end{align*}
Hence~\eqref{eq:naturallyreductive} is satisfied and $(M,B)$ is naturally reductive.
\end{proof}

Homogeneous spaces equipped with a metric constructed from the bi-invariant metric of the group that supplies the transitive action are called \emph{normal homogeneous spaces} and Proposition \ref{prop:normalhomogeneous} shows that normal homogeneous spaces are naturally reductive. Flag manifolds (see Section \ref{sec:flag}) and Stiefel manifolds (see Section \ref{sec:stiefel}) when equipped with the normal metric are examples of normal homogeneous spaces. Both flag manifolds and Stiefel manifolds can be equipped with more general metrics that fail condition \eqref{eq:naturallyreductive}. In this case, they are only reductive homogeneous spaces, but the unimodularity of the Lie group implies that Riemannian Brownian motion can still be constructed in the same way as for naturally reductive homogeneous spaces.

\subsection{The Riemannian Brownian motion on a reductive homogeneous space}\label{ssec:RM_natredHS}
There are multiple approaches to introducing a Riemannian Brownian motion that can be shown to be equivalent. One approach is to define the Riemannian Brownian motion as the solution to the martingale problem associated with the Laplace--Beltrami operator; however, this does not provide a pathwise construction. The Eells--Elworthy--Malliavin construction, also known as `rolling without slipping', \emph{does} provide a pathwise approach: the Riemannian Brownian motion is defined as the semimartingale whose anti-development (after lifting to the frame bundle) is a standard Euclidean Brownian motion. For details on these equivalent constructions we refer to e.g.~\cite[Section 3.2]{hsu2002stochastic}. \par
In this section, we demonstrate that reductive homogeneous spaces allow for a more direct pathwise construction of the Riemannian Brownian motion that avoids the lifting to the frame bundle. Before doing so, we recall several important definitions from stochastic analysis on general smooth manifolds. Following \cite[pp.19--20]{hsu2002stochastic}, we introduce the notion of a manifold-valued semimartingale and explain how to interpret stochastic differential equations (SDEs) on smooth manifolds. The key result is Proposition~\ref{prop:unimodularredhombm} below, which identifies a Riemannian Brownian motion on a reductive homogeneous space endowed with a $G$-invariant metric provided either $G$ is unimodular or the space is naturally reductive.

Let $(\Omega, \mathscr{F}, (\mathscr{F}_t)_{t\geq 0}, \mathbb{P})$ be a filtered probability space and let $\tau$ be an $(\mathscr{F}_t)_{t\geq 0}$-stopping time. Following \cite[pp.~19--20]{hsu2002stochastic} (see also \cite[Chapter~3]{Emery:1989}), a real-valued adapted process $(Z(t))_{t\in [0,\tau)}$ with continuous sample paths is called a \emph{(continuous) semimartingale on $[0,\tau)$} if there exists a sequence of stopping times $(\tau_n)_{n\in\N}$ with $\tau_n\uparrow\tau$ and $\tau_n<\tau$ on $\{\tau>0\}$ such that for every $n\in \N$ the stopped process $(Z(t\wedge\tau_n))_{t\geq 0}$ is a continuous semimartingale on $[0,\infty)$. For such $Z$ and a locally bounded predictable process $\phi$ on $[0,\tau)$, the stochastic (It\^o or Stratonovich) integral $\int_0^t \phi(s)\,\bm{d}Z(s)$ is defined on $\{t<\tau\}$ by setting
\begin{equation}\label{eq:localised_integral}
    \int_0^t \phi(s)\,\bm{d}Z(s) := \int_0^{t} \phi(s\wedge \tau_n)\,\bm{d}Z(s\wedge \tau_n)\quad \text{on } \{t<\tau_n\},\ n\in \N;
\end{equation}
by locality of the stochastic integral the right-hand side does not depend on $n$, nor on the choice of the sequence $(\tau_n)_{n\in\N}$, and the resulting process is again a continuous semimartingale on $[0,\tau)$. Note that a localising sequence $(\tau_n)_{n\in\N}$ is only needed when a process fails to stay within a compact set: otherwise $\tau=\infty$, one can take $\tau_n=n$, and the integral above is well-defined for all $t\geq 0$ by the usual theory. Recalling that any manifold $M$ is $\sigma$-compact, for an $M$-valued process $X$ one typically builds the localising sequence by setting $\tau_n$ to be the first exit time of a process $X$ from $K_n$ (or $n$, whichever is smaller), where $(K_n)_{n\in\N}$ is an exhaustion of $M$ by compact sets; this is how we proceed for the maximal solutions considered below.

\begin{definition}
Let $M$ be a smooth manifold, let $(\Omega, \mathscr{F}, (\mathscr{F}_t)_{t\geq 0}, \mathbb{P})$ be a filtered probability space, and let $\tau$ be an $(\mathscr{F}_t)_{t\geq 0}$-stopping time. A continuous, adapted, $M$-valued process $X$ defined on $[0,\tau)$ is called an \emph{$M$-valued semimartingale on $[0,\tau)$} if $f\circ X$ is a real-valued semimartingale on $[0,\tau)$ in the above sense for all $f\in C^\infty(M)$.
\end{definition}

\begin{definition}\label{def:sol_SDE}
Let $M$ be a smooth manifold and let $(\Omega, \mathscr{F}, (\mathscr{F}_t)_{t\geq 0}, \mathbb{P})$ be a filtered probability space. Let $V_1,\hdots, V_r \in \mathfrak{X}(M)$ and let $Z$ be an $\R^{r}$-valued continuous semimartingale. Let $\xi\colon \Omega \rightarrow M$ be an $\mathscr{F}_0$-measurable random variable and let $\tau\colon \Omega \rightarrow (0,\infty]$ be a stopping time.
An $M$-valued stochastic process $(X(t))_{t\in [0,\tau)}$  is called a \emph{solution up to $\tau$ to the SDE} formally given by
\begin{equation}\label{eq:semimartingalesde}
    \bm{d}X(t) = \sum_{i=1}^r V_i(X(t))\circ \bm{d}Z_i(t), \quad t\geq 0; \quad X(0) = \xi,
\end{equation}
if it is an $M$-valued semimartingale on $[0,\tau)$ and moreover
\begin{equation}\label{eq:sol_SDE_def}
    f(X(t)) = f(\xi)+\int_0^{t} \sum_{i=1}^r (V_i f)(X(s))\circ \bm{d}Z_i(s)\quad \text{on } \{t<\tau\},\ t\geq 0,
\end{equation}
for all $f\in C^\infty(M)$, where the integral is a Stratonovich integral on $[0,\tau)$ in the sense of~\eqref{eq:localised_integral}. A solution $X$ up to a stopping time $\tau$ to~\eqref{eq:semimartingalesde} is also called a \emph{local solution}. Moreover, a solution $X$ up to a stopping time $\tau$ is called a \emph{maximal solution} if for any other solution $\tilde{X}$ up to a stopping time $\tilde{\tau}$ satisfying $X|_{[0,\tau\wedge \tilde{\tau})} = \tilde{X}|_{[0,\tau\wedge \tilde{\tau})}$ a.s.\ one has $\P(\tau\geq \tilde{\tau})=1$. If there exists a local solution $X$ up to $\tau\equiv \infty$, then we call $X$ a \emph{global} solution.
\end{definition}

\par
Existence and uniqueness of a maximal solution to~\eqref{eq:semimartingalesde} follow from the Euclidean theory: by Whitney's embedding theorem there exists a proper smooth embedding $\iota\colon M\rightarrow \R^N$ for some $N\in \N$ (see e.g.~\cite[Theorem 6.15]{lee2012introductionsmooth}), so that $\iota(M)$ is a closed submanifold of $\R^N$. By Lemma~\ref{lem:extension}, the vector fields $\textnormal{d}\iota\circ V_i\circ\iota^{-1}\in \mathfrak{X}(\iota(M))$ extend to smooth vector fields $\hat{V}_1,\ldots,\hat{V}_r\in \mathfrak{X}(\R^N)$. The SDE in $\R^N$ driven by $\hat{V}_1,\ldots,\hat{V}_r$ with initial value $\iota(\xi)$ has a unique maximal solution $(\hat{X}(t))_{t\in[0,\tau^{\max})}$, and since the $\hat{V}_i$ are tangent to $\iota(M)$, this solution stays in $\iota(M)$ (see e.g.~\cite[Proposition 1.2.8]{hsu2002stochastic}). Consequently $X:=\iota^{-1}\circ \hat{X}$ is a solution to~\eqref{eq:semimartingalesde} up to $\tau^{\max}$, and it is maximal: any solution to~\eqref{eq:semimartingalesde} is mapped by $\iota$ to a solution of the SDE in $\R^N$, so uniqueness in $\R^N$ yields uniqueness in $M$, and as $\iota$ is proper, $\hat{X}$ leaving every compact subset of $\R^N$ as $t\uparrow \tau^{\max}$ on $\{\tau^{\max}<\infty\}$ is equivalent to $X$ leaving every compact subset of $M$ (see also \cite[Proposition 1.2.6]{hsu2002stochastic}; in particular, the sequence of stopping times in~\eqref{eq:localised_integral} can be taken as described there). Note that if $M$ is compact (so that $\hat{V}_1,\ldots,\hat{V}_r$ can be taken with compact support, hence with globally bounded derivatives), then~\eqref{eq:semimartingalesde} allows for a unique global solution.

Now, let $M\cong G/H$ be a reductive homogeneous space, i.e., we have the decomposition $\mathfrak{g} = \mathfrak{h} \oplus \mathfrak{m}$, see~\eqref{eq:reductive}. Let $\mathfrak{m}$ be endowed with an $\operatorname{Ad}_H$-invariant inner product $B$; this induces a $G$-invariant Riemannian metric $\bm{g}_B$ on $M$ (see~\eqref{eq:gB}). Let $r=\dim \mathfrak{m}$ and let $e_1,\ldots,e_r$ be an orthonormal basis of $(\mathfrak{m}, B)$, and let $W_{1},\ldots, W_{r}$ be independent standard $(\mathscr{F}_t)_{t\geq 0}$-Brownian motions. We first define a Brownian motion $(W_{\mathfrak{m}}(t))_{t\geq 0}$ in $\mathfrak{m}$ by
\begin{equation}\label{eq:mbrownianmotion}
    W_\mathfrak{m}(t) := \sum_{i=1}^{r} e_i W_i(t),\quad t\geq 0.
\end{equation}
Since $B$ is $\operatorname{Ad}_H$-invariant, $\operatorname{Ad}_h|_{\mathfrak{m}}$ is an orthogonal transformation of $(\mathfrak{m},B)$ for every $h\in H$, whence
\begin{equation}\label{eq:adjointlaw}
(\operatorname{Ad}_h W_\mathfrak{m}(t))_{t\geq 0} \stackrel{\text{law}}{=} (W_\mathfrak{m}(t))_{t\geq 0},\quad h\in H.
\end{equation}
In order to insert this Brownian motion into our reductive homogeneous space $M$, we first insert it into the group $G$ and then project it onto $M$. As $M\cong G/H$ is a space of \emph{left} cosets, on which $G$ acts from the left, the appropriate SDE on $G$ is the \emph{left}-invariant one: its projection to $M$ will turn out to depend on the starting point in $G$ only through its coset, and the resulting process on $M$ will be $G$-invariant. (For the space of right cosets $H\backslash G$ one uses the right-invariant version instead, see Remark~\ref{rem:right_cosets}.)

Recall from Section~\ref{sec:notation} that $L_g$ denotes the left translation by $g$ on $G$. We define a left-invariant diffusion on $G$ starting at $e$, driven by $W_{\mathfrak{m}}$, as the solution to the following Stratonovich SDE
\begin{equation}\label{eq:stratsde}
    \bm{d}g(t) = (\textnormal{d}L_{g(t)})_e\left(\sum_{i=1}^{r} e_i\circ \bm{d}W_i(t)\right) = (\textnormal{d}L_{g(t)})_e \circ \bm{d}W_{\mathfrak{m}}(t),\quad t\geq 0; \quad g(0)=e.
\end{equation}
It follows from the chain rule for the Stratonovich integral that this SDE is left-invariant in the sense that if $g=(g(t))_{t \in [0,\tau)}$ is a solution to~\eqref{eq:stratsde} on $[0,\tau)$ for some stopping time $\tau\colon \Omega \rightarrow [0,\infty]$, then for any $a\in G$, $ag=(a g(t))_{t \in [0,\tau)}$ is a local solution to an SDE with the same dynamics but with initial condition $a$. Also note that unless $\mathfrak{m}=\mathfrak{g}$ (i.e., unless $H$ is discrete), the diffusion~\eqref{eq:stratsde} is degenerate: it only has noise in the $r=\dim\mathfrak{m}$ directions of $\mathfrak{m}$, and thus it is in general \emph{not} a (Riemannian) Brownian motion on $G$.

Recall that the flow of a left-invariant vector field on a Lie group is given by right-multiplication with the associated Lie exponential map (see, e.g.,~\cite[Chapter 20]{lee2012introductionsmooth}).
More specifically, defining the left-invariant vector fields $\tilde{e}_1,\ldots,\tilde{e}_r\in \mathfrak{X}(G)$ by
\begin{equation}\label{eq:def_tilde_e}
\tilde{e}_i(g) = (\textnormal{d}L_{g})_e e_i,\quad g\in G,
\end{equation}
the integral curve $\gamma_i$ of the left-invariant vector field $\tilde{e}_i$ starting at $g$ is given by
\begin{equation}
    \gamma_i(t) = g\exp(t e_i), \quad t\in \R.
\end{equation}
Thus the vector fields $\tilde{e}_1,\ldots,\tilde{e}_r$ are complete. Additionally, the Lie algebra generated by the driving vector fields $\{\tilde{e}_1,\hdots, \tilde{e}_r\}$ via iterated Lie brackets is a Lie subalgebra of the Lie algebra of all left-invariant vector fields on $G$, and thus finite-dimensional. It follows that the conditions of \cite[Chapter 3, Theorem 5.1]{kunita1984stochastic} are satisfied, and we obtain the strong completeness of the stochastic flow associated with~\eqref{eq:stratsde}, guaranteeing that this equation allows for a unique global solution $(g(t))_{t\geq 0}$.

The $G/H$-valued process $\pi(g)$ is our candidate for a Riemannian Brownian motion on $(G/H,\bm{g}_B)$. In order to verify that, under certain circumstances, $\pi(g)$ is indeed a Riemannian Brownian motion, we first establish that it is a Markov process and then proceed to identify the generator. We begin by considering the conjugated process $y=hg h^{-1}$ (with $h\in H$). By the chain rule, $y$ satisfies the following SDE:
\begin{equation}
    \bm{d}y(t) = \textnormal{d}(L_h\circ R_{h^{-1}})_{g(t)}\circ \bm{d}g(t) = \textnormal{d}(L_h\circ R_{h^{-1}})_{g(t)}(\textnormal{d}L_{g(t)})_e\circ \bm{d}W_\mathfrak{m}(t), \quad t\geq 0;\quad y(0)=e.
\end{equation}
Using that $\textnormal{d}(L_h\circ R_{h^{-1}})_{g(t)}= (\textnormal{d}L_{h g(t) h^{-1}})_e \operatorname{Ad}_h (\textnormal{d}L_{g(t)^{-1}})_{g(t)}$, we see that $y$ satisfies
\begin{equation}
    \bm{d}y(t) = (\textnormal{d}L_{y(t)})_e(\operatorname{Ad}_h\circ \bm{d}W_\mathfrak{m}(t)),\quad t\geq 0; \quad y(0)=e.
\end{equation}
Since we have the law equivalence \eqref{eq:adjointlaw}, and since pathwise uniqueness for~\eqref{eq:stratsde} implies uniqueness in law, it follows that $(hg(t)h^{-1})_{t\geq 0} \stackrel{\textnormal{law}}{=}(g(t))_{t\geq 0}$ (here equality in law is to be understood as equality in the law of the processes). This implies that if $g$ is the solution to~\eqref{eq:stratsde} and $h\in H$, then
\begin{equation}\label{eq:shiftlemma}
(hg(t))_{t\geq 0} \stackrel{\text{law}}{=} (g(t) h)_{t\geq 0}.
\end{equation}
Next, observe that the law of the projected process $t\mapsto \pi(xg(t))$ only depends on the coset $xH$ and not on the specific representative $x\in G$: for $h\in H$, \eqref{eq:shiftlemma} yields
\begin{equation}\label{eq:law_coset}
    (\pi(x h g(t)))_{t\geq 0} = (xhg(t)H)_{t\geq 0} \stackrel{\text{law}}{=} (xg(t) hH)_{t\geq 0} = (x g(t) H)_{t\geq 0} = (\pi(xg(t)))_{t\geq 0}.
\end{equation}
Next, fix $s\geq 0$. By left-invariance of~\eqref{eq:stratsde} and the independence of the increments of $W_{\mathfrak{m}}$, one has $g(s+t) = g(s)\,g^{(s)}(t)$, $t\geq 0$, where $g^{(s)}$ is the solution to~\eqref{eq:stratsde} driven by $(W_{\mathfrak{m}}(s+t)-W_{\mathfrak{m}}(s))_{t\geq 0}$; in particular, $g^{(s)}$ is independent of $\mathscr{F}_s$ and has the same law as $g$. Hence, for any bounded measurable functional $F$ on the path space of $M$,
\begin{equation*}
    \E\big[F\big((\pi(g(s+t)))_{t\geq 0}\big)\,\big|\,\mathscr{F}_s\big] = \Psi_F(g(s)),\qquad \Psi_F(x) := \E\big[F\big((\pi(xg(t)))_{t\geq 0}\big)\big],\quad x\in G,
\end{equation*}
and by~\eqref{eq:law_coset} the function $\Psi_F$ is constant on cosets, i.e., $\Psi_F(g(s))$ is a measurable function of $\pi(g(s))$. This shows that $\pi(g)$ is a Markov process on $M\cong G/H$ (with respect to $(\mathscr{F}_t)_{t\geq 0}$), whose transition probabilities are given by the laws of $\pi(xg(t))$, $xH\in M$.

Having established that the process $\pi(g)$ is a Markov process on $M\cong G/H$, we now identify two situations for which this Markov process is in fact a Riemannian Brownian motion:

\begin{proposition}\label{prop:unimodularredhombm}
Let $M\cong G/H$ be a reductive homogeneous space with reductive decomposition $\mathfrak{g}=\mathfrak{h}\oplus\mathfrak{m}$, and let $B\colon \mathfrak{m}\times \mathfrak{m}\rightarrow \R$ be an $\operatorname{Ad}_H$-invariant inner product. Suppose at least one of the following two properties is satisfied:
\begin{enumerate}
\item $G$ is a unimodular Lie group (i.e., the left Haar measure coincides with the right Haar measure);
\item $M\cong G/H$ is naturally reductive with respect to $B$ (see Definition~\ref{def: naturally reductive space}).
\end{enumerate}
Then the process $X=\pi(g)$, with $g$ defined by \eqref{eq:stratsde}, is a Riemannian Brownian motion on $(M,\bm{g}_B)$ starting at $o$, where $\bm{g}_B$ is the $G$-invariant Riemannian metric induced by $B$ (see~\eqref{eq:gB}).
\end{proposition}

\begin{proof}
We first provide a proof for the unimodular case, as this is the more difficult one to treat.
    Consider the Stratonovich SDE on $G$ given by \eqref{eq:stratsde}. This defines a left-invariant diffusion on $G$ with noise in the directions of $\mathfrak{m}$. Using the Stratonovich formulation avoids It\^o correction terms, so its generator on $G$ is simply
    \begin{equation}\label{eq:generatorG}
        \mathcal{L}_G = \frac{1}{2}\sum_{i=1}^{r}\tilde{e}_i^2,
    \end{equation}
    where $\tilde{e}_i$ are the corresponding left-invariant vector fields on $G$ (see~\eqref{eq:def_tilde_e}). For $h\in H$, $X\in \mathfrak{g}$ and $F\in C^\infty(G)$ one has $\tilde{X}(F\circ R_h) = \big(\widetilde{\operatorname{Ad}_{h^{-1}}X}\,F\big)\circ R_h$ (where $\widetilde{\operatorname{Ad}_{h^{-1}}X}\in\mathfrak{X}(G)$ is the left-invariant vector field associated with $\operatorname{Ad}_{h^{-1}}X\in \mathfrak{g}$). Inserting this observation into~\eqref{eq:generatorG} shows that $\mathcal{L}_G (F \circ R_h) = \frac{1}{2} \left(\sum_{i=1}^{r} (\widetilde{\operatorname{Ad}_{h^{-1}}e_i})^2F\right) \circ R_h$. Since $\operatorname{Ad}_{h^{-1}}e_1,\ldots,\operatorname{Ad}_{h^{-1}}e_r$ is again a $B$-orthonormal basis of $\mathfrak{m}$ by the $\operatorname{Ad}_H$-invariance of $\mathfrak{m}$ and $B$, and since $\sum_{i=1}^r \tilde{f}_i^{\,2}$ does not depend on the choice of the $B$-orthonormal basis $f_1,\ldots,f_r$ of $\mathfrak{m}$, it follows that $\mathcal{L}_G(F\circ R_h) = (\mathcal{L}_G F)\circ R_h$, $h\in H$. Now observe that the pullbacks $f\circ \pi$, $f\in C^\infty(M)$, are precisely the smooth functions $F$ on $G$ satisfying $F\circ R_h = F$ for all $h\in H$ (such an $F$ is constant on cosets, hence of the form $f\circ\pi$, and $f$ is smooth as locally $f = F\circ \sigma$ for a smooth local section $\sigma$ of the submersion $\pi$). Thus, $\mathcal{L}_G$ maps pullbacks of smooth functions on $M$ to pullbacks of smooth functions on $M$. This guarantees the existence of a projected generator $\mathcal{L}_M$ on $M$, a second-order differential operator satisfying $\mathcal{L}_G(f \circ \pi) = (\mathcal{L}_M f) \circ \pi$.

It is immediate that $\mathcal{L}_M$ is the generator of $X=\pi(g)$: indeed, for $f\in C^\infty(M)$ and $X=\pi(g)$, It\^o's formula on $G$ (i.e., Definition~\ref{def:sol_SDE} applied to $\tilde f = f\circ \pi$ together with the Stratonovich-to-It\^o conversion) shows that
\begin{equation*}
    f(X(t)) - f(o) - \int_0^t (\mathcal{L}_M f)(X(s))\,\bm{d}s = \tilde f(g(t)) - \tilde f(e) - \int_0^t (\mathcal{L}_G \tilde f)(g(s))\,\bm{d}s,\quad t\geq 0,
\end{equation*}
is a local martingale. Hence, in order to prove that $X$ is a Riemannian Brownian motion on $M$, it remains to prove that $\mathcal{L}_M=\frac{1}{2}\Delta_M$ (see~\cite[Section 3.2]{hsu2002stochastic}).
\par
In order to show that $\mathcal{L}_M$ coincides with $\frac{1}{2}\Delta_M$, we evaluate the operators at the basepoint $o = eH$. For $f \in C^\infty(M)$, let $\tilde{f} = f \circ \pi$. As $t\mapsto \exp(te_i)$ is an integral curve of the left-invariant vector field $\tilde{e}_i$, one has $(\tilde{e}_i\tilde{e}_i \tilde f)(e) = \frac{\bm{d}^2}{\bm{d}t^2}|_{t=0}\tilde{f}(\exp(te_i))$, and thus
    \begin{equation}
        (\mathcal{L}_G \tilde{f})(e) = \frac{1}{2}\sum_{i=1}^{r} \left.\frac{\bm{d}^2}{\bm{d}t^2}\right|_{t=0} \tilde{f}(\exp(t e_i)) = \frac{1}{2}\sum_{i=1}^{r} \left.\frac{\bm{d}^2}{\bm{d}t^2}\right|_{t=0} f(\pi(\exp(t e_i))).
    \end{equation}
    Let $e_1^*,\ldots,e_r^*$ be the local vector fields on the neighbourhood $N^*$ of $o$ associated with $e_1,\ldots, e_r$, as defined in~\eqref{eq:local_frame}; recall that $V\subseteq \mathfrak{m}$ is the neighbourhood of $0$ for which $V\ni Z \mapsto \pi(\exp(Z)) \in N^*$ is a diffeomorphism. By~\eqref{eq:gB}, $\bm{g}_B(e_i^*,e_j^*)(\pi(\exp(Z))) = B(e_i,e_j) = \delta_{i,j}$ for $Z\in V$, i.e., $e_1^*,\ldots,e_r^*$ is a local orthonormal frame on $N^*$ (see also \cite[Section 13]{nomizu1954invariant}). Moreover, as $t\mapsto \pi(\exp(te_i))$ is an integral curve of $e_i^*$ (see Section~\ref{ssec:HS}), we have
    \begin{equation}
        (e_i^* e_i^* f)(o) = \left.\frac{\bm{d}^2}{\bm{d}t^2}\right|_{t=0} f(\pi(\exp(t e_i))),\quad i\in\{1,\ldots,r\},
    \end{equation}
    whence $(\mathcal{L}_G \tilde{f})(e) = \frac{1}{2} \sum_{i=1}^{r} (e_i^*e_i^* f)(o)$. On the other hand, expressing the Laplace--Beltrami operator $\Delta_M = \mathrm{div} \circ \nabla$ in the local orthonormal frame $e_1^*,\ldots,e_r^*$ yields
    \begin{equation}
        \frac{1}{2}\Delta_M f(o) = \frac{1}{2}\sum_{i=1}^{r} \Big( (e_i^*e_i^* f)(o) - \big((\nabla_{e_i^*} e_i^*) f\big)(o) \Big),
    \end{equation}
    where $\nabla$ is the Levi-Civita connection of $\bm{g}_B$. By~\eqref{eq:connection_function} and \cite[Theorem 13.1]{nomizu1954invariant} (see~\eqref{eq:alphaB}), the connection function of $\nabla$ is $\alpha_B(X,Y) = \frac{1}{2}[X,Y]_\mathfrak{m} + U(X,Y)$, $X,Y\in \mathfrak{m}$, with $U:\mathfrak{m}\times\mathfrak{m}\to\mathfrak{m}$ the symmetric bilinear map determined by~\eqref{eq:defB}. Since $[e_i, e_i] = 0$, we obtain $(\nabla_{e_i^*} e_i^*)(o) = \alpha_B(e_i,e_i)^*(o) = U(e_i, e_i)^*(o)$, so that
    \begin{equation}
        (\mathcal{L}_M f)(o) = \frac{1}{2}\Delta_M f(o) + \frac{1}{2}\,\bm{d}f\Big(\Big(\sum_{i=1}^r U(e_i,e_i)\Big)^*(o)\Big).
    \end{equation}
    It thus remains to show that the drift $\sum_{i=1}^r U(e_i,e_i)\in \mathfrak{m}$ vanishes. Taking the inner product of the drift with an arbitrary vector $Z \in \mathfrak{m}$ and using~\eqref{eq:defB} yields
    \begin{equation}
        \sum_{i=1}^{r} B\big(U(e_i, e_i), Z\big) = \sum_{i=1}^{r} B\big([Z, e_i]_\mathfrak{m}, e_i\big) = \tr\big(P_{\mathfrak{m}}\circ \operatorname{ad}_Z|_\mathfrak{m}\big),
    \end{equation}
    where $P_{\mathfrak{m}}\colon \mathfrak{g}\rightarrow \mathfrak{m}$ denotes the projection onto $\mathfrak{m}$ along $\mathfrak{h}$. Because the reductive splitting guarantees that $\operatorname{ad}_Z$ maps $\mathfrak{h}$ into $\mathfrak{m}$ (recall $[\mathfrak{h},\mathfrak{m}]\subseteq \mathfrak{m}$), the $\mathfrak{h}$-$\mathfrak{h}$ block of $\operatorname{ad}_Z$ with respect to the decomposition $\mathfrak{g}=\mathfrak{h}\oplus \mathfrak{m}$ vanishes, meaning $\tr(P_{\mathfrak{m}}\circ \operatorname{ad}_Z|_\mathfrak{m}) = \tr(\operatorname{ad}_Z)$. Since $G$ is a unimodular Lie group, $\tr(\operatorname{ad}_Z) = 0$ for all $Z \in \mathfrak{g}$ (see Proposition~\ref{prop:unimodular_Ad}; note that $|\det \operatorname{Ad}_{\exp(tZ)}|=1$ for all $t$ gives $\operatorname{tr}\operatorname{ad}_Z=0$ also when $G$ is not connected). Consequently, the drift $\sum_{i=1}^{r} U(e_i, e_i)$ vanishes, and therefore $(\mathcal{L}_M f)(o) = \frac{1}{2}\Delta_M f(o)$.

    Since the SDE on $G$ is left-invariant and the metric on $M$ is $G$-invariant, both $\mathcal{L}_M$ and $\Delta_M$ are $G$-invariant differential operators. As they agree at the basepoint $o$, they must agree everywhere on $M$. \par
Finally, we consider the case that $M\cong G/H$ is a naturally reductive homogeneous space. In this case, the function $U$ in the above proof satisfies $U\equiv 0$ (see Definition~\ref{def: naturally reductive space} and \eqref{eq:defB}). Consequently, we immediately obtain that $\pi(g)$ is a Riemannian Brownian motion.
\end{proof}

\begin{remark}\label{rem:alt_proof_unimodular}
We can also give an algebraic proof of Proposition~\ref{prop:unimodularredhombm} in the unimodular case. The quadratic element $\Omega = \sum_{i=1}^{r} e_i^2 \in \mathcal{U}(\mathfrak{g})^{\operatorname{Ad}_H}$ projects to a well-defined $G$-invariant second-order differential operator $2\mathcal{L}_M$ on $M$ (here $\mathcal{U}(\mathfrak{g})$ is the universal enveloping algebra, whose elements act as left-invariant differential operators on $G$). Since $2\mathcal{L}_M$ and $\Delta_M$ share the same principal symbol, $2\mathcal{L}_M - \Delta_M = V + c$ for a vector field $V$ and a function $c$ on $M$; evaluating at $f\equiv 1$ gives $c=0$, since $\Omega$, $\Delta_M$, and $V$ all annihilate constants. Thus $2\mathcal{L}_M = \Delta_M + V$.

Because $B$ is $\operatorname{Ad}_H$-invariant, $\bm{g}_B$ is a $G$-invariant metric on $M$, so its Riemannian volume $d\mathrm{vol}_{\bm{g}_B}$ is a $G$-invariant measure on $M$. Moreover $|\det \operatorname{Ad}_h|_{\mathfrak{m}}| = 1$ for $h\in H$, so that the modular functions of $G$ and $H$ agree on $H$; consequently, Weil's integration formula holds with $d\mathrm{vol}_{\bm{g}_B}$ as the invariant measure on $G/H$ and suitably normalised Haar measures $dg,dh$ on $G,H$: $\int_G \phi\,dg = \int_M\big(\int_H \phi(gh)\,dh\big)\,d\mathrm{vol}_{\bm{g}_B}$, $\phi\in C_c^\infty(G)$.

Since $G$ is unimodular, its left Haar measure is bi-invariant, so each $\tilde e_i$ is skew-adjoint on $L^2(G,dg)$ and thus $\Omega^\dagger = \sum_i(-\tilde e_i)^2 = \Omega$. As $\Omega\in\mathcal U(\mathfrak g)^{\operatorname{Ad}_H}$, the associated operator commutes with right translation by $H$. Together with Weil's formula above, this transports self-adjointness of $\Omega$ on $L^2(G,dg)$ to self-adjointness of $2\mathcal{L}_M$ on $L^2(M,d\mathrm{vol}_{\bm{g}_B})$, which is the standard mechanism by which $\operatorname{Ad}(H)$-invariant elements of $\mathcal{U}(\mathfrak{g})$ descend to self-adjoint invariant operators on $G/H$ (see e.g.\ \cite{helgason1962differential}). When $H$ is compact this reduces to Fubini applied to pullbacks; for $H$ non-compact, pullbacks of compactly supported functions on $M$ are not themselves compactly supported on $G$, so that the argument via Weil's formula is necessary.

For any vector field $V$, the formal adjoint satisfies $V^\dagger = - V - \operatorname{div} V$. Self-adjointness of $2\mathcal{L}_M = \Delta_M + V$ and of $\Delta_M$ then forces $2V = -\operatorname{div}V$; applying this to $f\equiv 1$ gives $\operatorname{div}V = 0$, hence $V=0$. Thus $\mathcal{L}_M = \tfrac12\Delta_M$, recovering $\pi(g)$ as Riemannian Brownian motion.

This argument uses unimodularity of $G$ at the step where $\Omega$ is shown to be self-adjoint on $G$, so it does not cover the naturally reductive case. There, $V=0$ follows instead exactly as in the proof above: natural reductivity means $U\equiv 0$ by definition (Definition~\ref{def: naturally reductive space}, \eqref{eq:defB}), so the drift $\sum_i U(e_i,e_i)$ vanishes identically.
\end{remark}

\begin{remark}\label{rem:right_cosets}
    By an analogous argument, \textit{mutatis mutandis}, under the assumptions of Proposition~\ref{prop:unimodularredhombm} the right-invariant process projected onto the space of right cosets $M \cong H\backslash G$ also has the Laplace--Beltrami operator $\frac{1}{2}\Delta_M$ as its infinitesimal generator. Consequently, it too defines a Riemannian Brownian motion on the homogeneous space.

    In fact, one can go beyond Riemannian Brownian motion. Assuming left- and right-invariant diffusions with $\operatorname{Ad}_H$-invariant generators, we have
\begin{equation*}
    \begin{aligned}
    \text{left-invariant diffusions on }G &\quad    \Longrightarrow \quad G\text{-invariant diffusions on }M\cong G/H,\\
    \text{right-invariant diffusions on }G &\quad   \Longrightarrow \quad G\text{-invariant diffusions on }M\cong H\backslash G.
    \end{aligned}
\end{equation*}
\end{remark}

\section{Wong--Zakai approximations for SDEs on a Riemannian manifold}\label{sec:WZ}

The goal of this section is to demonstrate how the representation of a Riemannian Brownian motion provided in the previous section can be exploited for simulations. As announced in the introduction, we employ Wong--Zakai approximations for this purpose. Below, we first discuss the available convergence results for Wong--Zakai approximations in the Euclidean setting, then translate these results to the general Riemannian manifold setting (see Section~\ref{ssec:WZ_RM}), and finally discuss the implications in the setting of reductive homogeneous spaces (see Section~\ref{ssec:WZ_LG_HS}).

Let $(M,\bm{g})$ be a Riemannian manifold, let $(\Omega, \mathscr{F}, (\mathscr{F}_t)_{t\geq 0}, \mathbb{P})$ be a filtered probability space, and let $W_1,\ldots,W_r$ be independent $(\mathscr{F}_t)_{t\geq 0}$ standard Brownian motions. Let $V_0,V_1,\hdots, V_r \in \mathfrak{X}(M)$ and let $\xi\colon \Omega \rightarrow M$ be an $\mathscr{F}_0$-measurable random variable. Finally, let the stochastic process $(X(t))_{t\in [0,\tau^{\max})}$ be the maximal solution (in the sense of Definition~\ref{def:sol_SDE}) to the following SDE:
\begin{equation}\label{eq:RMsde}
    \bm{d}X(t) = V_0(X(t))\,\bm{d}t + \sum_{i=1}^r V_i(X(t))\circ \bm{d}W_i(t), \quad t\geq 0; \qquad X(0)=\xi.
\end{equation}
Note that although Definition~\ref{def:sol_SDE} concerns the driftless equation~\eqref{eq:semimartingalesde}, it covers~\eqref{eq:RMsde} upon taking $Z=(Z^0,Z^1,\ldots,Z^r)$ with $Z^0(t)=t$ and $Z^i=W_i$, $i\in\{1,\ldots,r\}$.

\begin{definition}
Let $h>0$, $f\in C([0,\infty),\R)$. The \emph{piecewise linear approximation of $f$ with step-size $h$} is given by
\begin{equation}
f_h(t) = f(kh) + (t/h-k)\big(f((k+1)h)-f(kh)\big),\quad t\in[kh, (k+1)h),\, k\in \N_0.
\end{equation}
\end{definition}
\begin{definition}
Let $h>0$ and for $i\in \{1,\ldots,r\}$ let $W_{i,h}$ be the piecewise linear approximation of $W_i$ with step-size $h$. Let $(X_h(t))_{t\in [0,\tau_h^{\max})}$ be the maximal solution\footnote{This `maximal solution' is defined recursively: one first constructs the maximal solution $(X_h(t))_{t\in [0,h\wedge \tau^{\max}_{h,1})}$ in the sense of Definition~\ref{def:sol_SDE}, but on the interval $[0,h]$ instead of $[0,\infty)$ (where the coefficients are smooth). Subsequently, restricting to the set $\{\tau^{\max}_{h,1}>h\}$, one constructs the maximal solution $(X_h(t))_{t\in [h, \tau^{\max}_{h,2} \wedge 2h)}$ on $[h,2h]$ with initial condition $\xi = \lim_{t\nearrow h} X_h(t)$ (the limit exists on $\{\tau^{\max}_{h,1}>h\}$), etc. The stopping time $\tau^{\max}_h$ is set to be equal to $\tau^{\max}_{h,k}$ whenever $\tau^{\max}_{h,k}<hk$.} to the equation
\begin{equation}\label{eq:WZ-RMsde}
    \bm{d}X_h(t) = V_0(X_h(t))\,\bm{d}t + \sum_{i=1}^r V_i(X_h(t)) \dot{W}_{i,h}(t) \,\bm{d}t, \quad t\geq 0; \qquad X_h(0)=\xi.
\end{equation}
Then we call $X_h$ the \emph{Wong--Zakai approximation of $X$ with step-size $h$}.
\end{definition}

The obvious question to ask is whether $X_h$ converges to $X$, and if so, with what rate. Indeed, there is a vast literature on convergence (rates) of Wong--Zakai approximations for SDEs in $\R^N$, some via `classical' stochastic analysis methods, see e.g.~\cite{shmatkov2006rate}, but most based on a rough path approach, see e.g.~\cite{HuNualart:2009, FrizRiedel:2011, BayerEtAl:2016} and references therein. These translate to the Riemannian setting thanks to Nash's theorem. We provide details below, first citing the relevant available theorems in the Euclidean setting and then performing the translation to the Riemannian setting.

\begin{remark}
The paths of the Brownian motion can be approximated by smooth(er) functions in a variety of ways; piecewise linear functions are only one of many approximations that have been investigated under the umbrella of Wong--Zakai approximations. However, in the context of Lie groups and left-invariant vector fields, the piecewise linear functions provide a particularly nice representation for the Wong--Zakai approximation, see Equation~\eqref{eq:sol_WZ_LG} below.
\end{remark}

\subsubsection*{Strong \texorpdfstring{($L^p$)}{(Lp)} convergence rates for $(M,\bm{g})=(\R^N,\bm{g}_{\textnormal{Eucl}})$}
When reducing~\cite[Theorem 13]{BayerEtAl:2016} to the case that the driving process is a Brownian motion (i.e., when, in their setting, one assumes $\rho =1$) one obtains:

\begin{theorem}[Theorem 13 in~\cite{BayerEtAl:2016} (simplified version)] \label{thm: WZ Lp rate}
Assume, in the setting of equations~\eqref{eq:RMsde} and~\eqref{eq:WZ-RMsde}, that $(M,\bm{g})=(\R^N, \bm{g}_{\textnormal{Eucl}})$ for some $N\in \N$. Assume moreover that the vector fields $V_0,\ldots,V_r \in \mathfrak{X}(\R^N)$ have bounded derivatives of all orders\footnote{Note that having bounded first derivatives already suffices to ensure that equations~\eqref{eq:RMsde} and~\eqref{eq:WZ-RMsde} allow for unique global solutions $(X(t))_{t\geq 0}$ and $(X_{h}(t))_{t\geq 0}$, $h>0$.\label{fn:global solutions}}. Then for all $T>0$, all $\eta \in (0,\frac{1}{2})$, and all $p\in [1,\infty)$ for which $\| \xi \|_{L^p(\Omega,\R^N)}< \infty$ there exists a constant $C>0$ such that
\begin{equation}
\left\| \sup_{t\in [0,T]} \| X(t) - X_h(t) \|_{\R^N} \right\|_{L^{p}(\Omega)} \leq C h^{\eta},\quad h>0.
\end{equation}
\end{theorem}

\begin{remark}
The assertions in~\cite[Theorem 13]{BayerEtAl:2016} are significantly more general than the result presented above: for example, the authors only assume that the first $m$ derivatives of the vector fields are bounded, at the price of a lower rate of convergence.\par
Also note that~\cite[Theorem 13]{BayerEtAl:2016} is stated in the context of rough differential equations (and not Stratonovich differential equations). The equivalence of these settings follows from e.g.~\cite[Theorem 9.1]{FrizHairer:2020}.
\end{remark}

\subsubsection*{Almost sure convergence rates for $(M,\bm{g})=(\R^N,\bm{g}_{\textnormal{Eucl}})$}

One can obtain almost sure rates from $L^p$ rates using localisation and a classical Borel--Cantelli argument, see e.g.~\cite{Gyongy:1998, KloedenNeuenkirch:2007, CoxEtAl:2021}. However, almost sure convergence rates of the Wong--Zakai approximations were actually established prior to the $L^p$-rates and under slightly weaker conditions, see~\cite{shmatkov2006rate, FrizRiedel:2011}. Indeed, one has:

\begin{theorem}[Theorem 2.3.1 in \cite{shmatkov2006rate}]\label{thm: WZ as rate}
Assume, in the setting of equations~\eqref{eq:RMsde} and~\eqref{eq:WZ-RMsde}, that $(M,\bm{g})=(\R^N, \bm{g}_{\textnormal{Eucl}})$ for some $N \in \N$. Assume moreover that $V_0$ is globally Lipschitz continuous and that the vector fields $V_1,\ldots,V_r \in C^3(\R^N,\R^N)$ have bounded first, second, and third derivatives\footref{fn:global solutions}. Then for all $T>0$ and all $\eta \in (0,\frac{1}{2})$ there exists an almost surely finite random variable $K$ such that
\begin{equation}
\sup_{t\in [0,T]} \| X(t) - X_h(t) \|_{\R^N} \leq K h^{\eta},\quad h>0.
\end{equation}
\end{theorem}

\begin{remark}
In~\cite{shmatkov2006rate} the author considers vector fields $V_0,\ldots,V_r$ that depend on the time parameter; with respect to the time parameter these vector fields only need to be once continuously differentiable with bounded derivatives.
\end{remark}

\subsection{Convergence rates in Riemannian manifolds}\label{ssec:WZ_RM}

Now, let $(M,\bm{g})$ be any Riemannian manifold.
By the Nash embedding theorem, see \cite{nash1956imbedding}, there exists a smooth isometric embedding $\iota:M\to\mathbb{R}^N$ for some sufficiently large $N$ (for compact $M$, smooth isometric embeddings are guaranteed for $N = (n+2)(n+3)/2$ where $n$ is the dimension of $M$, see \cite[p.223]{gromov1986partial}). This means that the map $\iota\colon M \rightarrow \mathbb{R}^N$ is a diffeomorphism onto its range and that for every point $x\in M$, the derivative $d\iota_x$ is a linear map from $T_x M$ to $\R^N$ that is an isometry with respect to the metric $\bm{g}$ on $M$ and the Euclidean inner product in $\R^N$.
In particular, one has
\begin{equation}\label{eq:metric_lowerbound}
 \| \iota(x) - \iota(y) \|_{\R^N} \leq d_{\bm{g}}(x,y), \quad x,y\in M.
\end{equation}
Moreover, as the embedding is isometric and smooth, for every $x\in M$ there exists an $r_x>0$ such that $ d_{\bm{g}}(y,z) \leq 2  \| \iota(y) - \iota(z) \|_{\R^N}$ for all $y,z\in M$ satisfying $d_{\bm{g}}(x,y),d_{\bm{g}}(x,z) \leq r_x$. In particular, by a compactness argument (the local estimate above controls pairs close to the diagonal, while $d_{\bm{g}}$ is bounded and $\|\iota(x)-\iota(y)\|_{\R^N}$ is bounded from below on the remaining compact set of pairs), for \emph{compact} Riemannian manifolds there exists a $C\geq 1$ such that for all $x,y$ in the same connected component of $M$ one has
\begin{equation}\label{eq:metric_upperbound}
d_{\bm{g}}(x,y) \leq C \| \iota(x) - \iota(y) \|_{\R^N}.
\end{equation}

This allows us to transfer the convergence rates of Wong--Zakai approximations from the Euclidean setting to the setting of compact Riemannian manifolds:

\begin{corollary}\label{cor:WZ rates compact RM}
Let $(M,\bm{g})$ be a compact Riemannian manifold and let $V_0,V_1,\ldots,V_r\in \mathfrak{X}(M)$, and let $X$ and $X_h$, $h>0$, be global solutions\footnote{Note that global existence is automatic on compact manifolds, see also\footref{fn:global solutions}.} to equations~\eqref{eq:RMsde} and~\eqref{eq:WZ-RMsde} with $\mathscr{F}_0$-measurable initial value $\xi$. Then for all $T>0$, all $\eta \in (0,\frac{1}{2})$, and all $p\in [1,\infty)$ there exists a constant $C>0$ such that
\begin{equation}
\| \sup_{t\in [0,T]} d_{\bm{g}}( X(t) , X_h(t) ) \|_{L^p(\Omega)} \leq C h^{\eta}, \quad h>0,
\end{equation}
and there exists an almost surely finite random variable $K$ such that
\begin{equation}
\sup_{t\in [0,T]} d_{\bm{g}}( X(t) , X_h(t) ) \leq K h^{\eta},\quad h>0.
\end{equation}
\end{corollary}

\begin{proof} First of all note that as $X$ and $X_h$, $h>0$, are continuous stochastic processes with the same (possibly random) initial value $\xi$, for each $\omega\in \Omega$ they all remain in the same connected component of $M$ (namely the component containing $\xi$).\par
Let $N\in \N$ be such that there exists a Nash embedding $\iota \colon M \rightarrow \R^N$, i.e., $\iota$ is a smooth isometric embedding (see \cite[Theorem 2]{nash1956imbedding}). In particular, the image of $M$ under $\iota$ is a compact (and thus bounded) subset of $\R^N$. The images of $V_0,V_1,\ldots,V_r$ under $\iota$ can be extended to smooth vector fields $\hat{V}_0,\hat{V}_1,\ldots,\hat{V}_r$ on $\R^N$ with compact support (see Lemma~\ref{lem:extension}). Moreover, the equations~\eqref{eq:RMsde} and~\eqref{eq:WZ-RMsde} in $\R^N$ with vector fields $\hat{V}_0,\hat{V}_1,\ldots,\hat{V}_r$ allow for a global solution $\hat{X}(t)=\iota(X(t))$ and $\hat{X}_h = \iota(X_h(t))$. The assertion thus follows from Theorems~\ref{thm: WZ Lp rate} and~\ref{thm: WZ as rate} combined with~\eqref{eq:metric_upperbound}.
\end{proof}

For non-compact Riemannian manifolds we face the problem that even if one assumes the vector fields $V_0,V_1,\ldots,V_r$ to have globally bounded derivatives on $M$, this does not suffice to ensure that their images under the Nash embedding can be extended to vector fields with globally bounded derivatives in the Euclidean space. Moreover, the upper bound~\eqref{eq:metric_upperbound} is no longer necessarily valid (and e.g.\ bounded curvature does not suffice to ensure it).
Thus, strong convergence rates are not straightforward to establish in the non-compact setting (and presumably will fail to exist in general). However, almost sure convergence rates can be obtained by a standard localisation argument, as shown below. Recall that a manifold is called \emph{complete} if it is geodesically complete; by the Hopf--Rinow theorem this is equivalent to stating that each connected component is Cauchy complete (see e.g.~\cite[Section 6.3]{lee2018introduction}).

\begin{corollary}\label{cor:WZ as rate RM}
Let $(M,\bm{g})$ be a complete Riemannian manifold and let $V_0,V_1,\ldots,V_r\in \mathfrak{X}(M)$. Let $(h_n)_{n\in \N}$ be a real-valued sequence such that $\lim_{n\rightarrow \infty} h_n = 0$. Assume that equation~\eqref{eq:RMsde} allows for a global solution $(X(t))_{t\geq 0}$. Let $(X_{h_n}(t))_{t \in [0,\tau^{\max}_{n})}$ denote the maximal solution to~\eqref{eq:WZ-RMsde} with $h=h_n$, $n\in \N$. Then for all $T>0$, $\eta \in (0,\frac{1}{2})$ there exist almost surely finite random variables $N_0\colon \Omega \rightarrow \N\cup \{\infty\}$ and $K_0\colon \Omega \rightarrow [0,\infty]$ such that for all $\omega\in \Omega$ and all $n\geq N_0(\omega)$
it holds that $\tau_n^{\max}(\omega) >T $ and
\begin{equation}\label{eq:WZ as rate RM}
\sup_{t\in [0,T]} d_{\bm{g}}( X(t,\omega) , X_{h_n}(t,\omega) ) \leq K_0(\omega) h_n^{\eta}.
\end{equation}
If moreover~\eqref{eq:WZ-RMsde} allows for a global solution for all $h=h_n$, $n\in \N$, then one can take $N_0\equiv 1$.
\end{corollary}

\begin{proof}[Proof of Corollary~\ref{cor:WZ as rate RM}]
Let $N\in \N$ be such that one has a Nash embedding $\iota \colon M \rightarrow \R^N$, i.e., $\iota$ is a smooth isometric embedding. By~\cite{Muller:2009} and the fact that $M$ is complete, we can assume the Nash embedding $\iota\colon M\rightarrow \R^N$ to be closed, and thus proper (see~\cite[Proposition 5.5]{lee2012introductionsmooth}). Also note again that as $X$ and $X_{h_n}$, $n\in \N$, are continuous stochastic processes with the same (possibly random) initial value, for each $\omega \in \Omega$ they all remain in the same connected component of $M$. \par  Define $B_R=\{x\in \R^N\colon \| x \|\leq R \}$ and $M_R = \iota^{-1}(B_R)$. As $\iota$ is proper, $M_R\subseteq M$ is compact. By arguing as for~\eqref{eq:metric_upperbound} (note that $d_{\bm{g}}$ is continuous and hence bounded on pairs of points of the compact set $M_R$ lying in the same connected component of $M$) one obtains for all $R>0$ the existence of a constant $C_R\geq 1$ such that for all $x,y\in M_R$ in the same connected component of $M$ one has
\begin{equation}\label{eq:metric_upperbound_localized}
d_{\bm{g}}(x,y) \leq C_R \| \iota(x) - \iota(y) \|_{\R^N}.
\end{equation}

Next, we define the stopping times $\tau_{R}\colon \Omega \rightarrow [0,\infty]$ and $\tau_{R,n}\colon \Omega \rightarrow [0,\infty]$ by
\begin{align*}
\tau_{R} & = \inf\{t\geq 0\colon \| \iota(X(t))\| \geq R\},\\
\tau_{R,n} & = \inf\{t\geq 0\colon \| \iota(X_{h_n}(t))\|\geq R\}.
\end{align*}
Note that as one necessarily has that $\limsup_{t\rightarrow \tau^{\max}_n} \| \iota(X_{h_n}(t)) \| = \infty$ on $\{\tau^{\max}_n<\infty\}$ (here we use that $\iota$ is proper), one has $\tau_{R,n} < \tau^{\max}_n$ on $\{\tau^{\max}_n<\infty\}$. \par

Next, consider the vector fields $\bar{V}_{i} \in \mathfrak{X}(\iota(M))$ given by $\bar{V}_{i} = \textnormal{d}\iota\,\circ\, V_i \,\circ\, \iota^{-1}$, $i=0,1,\ldots,r$. By Lemma~\ref{lem:extension} these vector fields can be extended to vector fields $\hat{V}_i \in \mathfrak{X}(\R^N)$, $i=0,1,\ldots,r$, and for all $R\in \N$ there exist compactly supported $\hat{V}_{i,R}\in \mathfrak{X}(\R^N)$ such that $\hat{V}_{i,R}|_{B_R} = \hat{V}_i|_{B_R}$. In particular, the vector field $\hat{V}_{i,R}$ has bounded derivatives of all orders. Let $(\hat{X}_{R,\infty}(t))_{t\geq 0}$ be the solution to the following SDE in $\R^N$
\begin{equation}\label{eq:RMsde_embedded}
    \bm{d}\hat{X}_{R,\infty}(t) = \hat{V}_{0,R}(\hat{X}_{R,\infty}(t))\,\bm{d}t + \sum_{i=1}^r \hat{V}_{i,R}(\hat{X}_{R,\infty}(t))\circ \bm{d}W_i(t), \quad t\geq 0; \qquad \hat{X}_{R,\infty}(0)=\iota \circ \xi,
\end{equation}
and let the stochastic processes $(\hat{X}_{R,n}(t))_{t\geq 0}$, $n\in \N$, be the solutions to
\begin{equation}\label{eq:WZ-RMsde_embedded}
    \bm{d}\hat{X}_{R,n}(t) = \hat{V}_{0,R}(\hat{X}_{R,n}(t))\,\bm{d}t + \sum_{i=1}^r \hat{V}_{i,R}(\hat{X}_{R,n}(t)) \dot{W}_{i,h_n}(t) \,\bm{d}t, \quad t\geq 0; \qquad \hat{X}_{R,n}(0)=\iota \circ \xi.
\end{equation}

As $\hat{V}_{i,R}$ and $\hat{V}_i$ coincide on $B_R$ we have, from the definition of $\tau_R$ and $\tau_{R,n}$, that $(\iota(X(t)))_{t\in [0,\tau_R)}$ is a local solution to~\eqref{eq:RMsde_embedded} and $(\iota(X_{h_n}(t)))_{t\in [0,\tau_{R,n})}$ is a local solution to~\eqref{eq:WZ-RMsde_embedded}. By uniqueness and continuity of solutions there exists a set $\Omega_0\in \mathcal{F}$ with $\P(\Omega_0)=1$ such that for all $n,R\in \N$
one has
\begin{align}   \label{eq:Xn_tildeXn_coincide}
    (\iota(X(t)))_{t \in [0,\tau_R]}
    & = (\hat{X}_{R,\infty}(t)))_{t\in [0,\tau_R]} &\text{on }\Omega_0, & &\text{and} & &
    (\iota(X_{h_n}(t)))_{t \in [0,\tau_{R,n}]}
    & = (\hat{X}_{R,n}(t))_{t\in [0,\tau_{R,n}]} &\text{on }\Omega_0.
\end{align}

For the remainder of this proof we fix $T>0$ and $\eta\in (0,\frac{1}{2})$. By Theorem~\ref{thm: WZ as rate} there exists, for all $R\in \N$, an almost surely finite random variable $K_R$ such that $\sup_{t\in [0,T]} \| \hat{X}_{R,\infty}(t) - \hat{X}_{R,n}(t) \|_{\R^N} \leq K_R h_n^{\eta}$ for all $n\in \N$; combined with~\eqref{eq:Xn_tildeXn_coincide} and~\eqref{eq:metric_upperbound_localized} (recall that $X(t),X_{h_n}(t)\in M_R$ for $t\leq \tau_R\wedge\tau_{R,n}$, and that $X(t)$ and $X_{h_n}(t)$ lie in the same connected component of $M$) this yields, for all $R,n\in \N$,
\begin{equation*}
\sup_{t\in [0,\tau_{R}\wedge \tau_{R,n} \wedge T]} d_{\bm{g}}(X(t), X_{h_n}(t)) \leq C_R \sup_{t\in [0,T]} \| \hat{X}_{R,\infty}(t) - \hat{X}_{R,n}(t) \|_{\R^N} \leq C_R K_R h_n^{\eta} \quad \text{on }\Omega_0.
\end{equation*}
By redefining $\Omega_0$ we can assume $K_R$ is finite on $\Omega_0$ for all $R\in \N$, whilst maintaining the property that $\P(\Omega_0)=1$. This ensures that for every $\omega \in \Omega_0$ and all $R\in \N$ there exists an $N_R(\omega)\in \N$ such that $K_R(\omega) h_n^{\eta} < R/2$ for all $n\geq N_R(\omega)$. As $\| \hat{X}_{R,\infty}(t)\| = \|\iota(X(t))\| \leq R/2$ for $t\leq \tau_{R/2}$, the triangle inequality and~\eqref{eq:Xn_tildeXn_coincide} then imply $\|\iota(X_{h_n}(t))\| < R$ for $t\leq \tau_{R/2}\wedge \tau_{R,n}\wedge T$, so that necessarily one has $\tau_{R,n}(\omega)\geq \tau_{R/2}(\omega)\wedge T$ whenever $n\geq N_R(\omega)$. In conclusion, we now have for all $\omega \in \Omega_0$ and all $n\geq N_R(\omega)$ that
\begin{equation}\label{eq:WZ as rate RM local}
\sup_{t\in [0,\tau_{R/2}\wedge T]} d_{\bm{g}}(X(t,\omega), X_{h_n}(t,\omega)) \leq C_R K_R(\omega) h_n^{\eta}.
\end{equation}

\par

Finally, for $R\in \N$ we define $A_R \in \mathcal{F}$ (with the convention $\tau_0 = 0$) by
\begin{equation*}
A_R = \{ \tau_{(R-1)/2} < T \leq \tau_{R/2}\}
\end{equation*}
and define the random variables $N_0$ and $K_0$ by setting $N_0=N_R$ and $K_0=C_R K_R$ on $A_R\cap \Omega_0$, and $N_0 = \infty$ and $K_0=\infty$ on $\Omega_0^c \cup (\Omega \setminus \cup_{R\in \N} A_R)$. The fact that $X$ is a global solution with continuous sample paths ensures that $\P(\cup_{R\in \N} A_R) = 1$, so that $N_0$ and $K_0$ are indeed almost surely finite. The desired bound~\eqref{eq:WZ as rate RM} follows from~\eqref{eq:WZ as rate RM local} and the fact that $N_0=N_R$, $K_0=C_R K_R$, and $\tau_{R/2}\geq T$ on $A_R$. Moreover, for $\omega\in A_R\cap \Omega_0$ and $n\geq N_0(\omega)$ we have shown that $\tau_{R,n}(\omega)\geq \tau_{R/2}(\omega)\wedge T = T$, and as $\tau_{R,n}<\tau_n^{\max}$ on $\{\tau_n^{\max}<\infty\}$, this implies $\tau_n^{\max}(\omega)>T$.

Regarding the final assertion of the corollary: suppose~\eqref{eq:WZ-RMsde} allows for a global solution for all $h=h_n$, $n\in \N$, and let $N_0\colon \Omega \rightarrow \N\cup \{\infty\}$ and $K_0\colon \Omega \rightarrow [0,\infty]$ be almost surely finite random variables such that~\eqref{eq:WZ as rate RM} holds for all $\omega \in \Omega$. The global existence of solutions with continuous sample paths implies that
\begin{equation*}
 \sup_{n\in \{1,\ldots,N_0\}}
    \sup_{t\in [0,T]}
        d_{\bm{g}}( X(t) , X_{h_n}(t) ) < \infty\quad \text{a.s.},
\end{equation*}
whence by (possibly) adjusting $K_0$ one can take $N_0\equiv 1$.\par
\end{proof}

\begin{remark}
We consider a sequence $(h_n)_{n\in\N}$ in Corollary~\ref{cor:WZ as rate RM} instead of all $h>0$ because the union of null-sets excluded in~\eqref{eq:Xn_tildeXn_coincide} must be countable.
\end{remark}

\subsection{Wong--Zakai approximations in Lie groups and their implications for approximating the Riemannian Brownian motion in reductive homogeneous spaces}\label{ssec:WZ_LG_HS}

\subsubsection{Convergence rates of the Wong--Zakai approximations in reductive homogeneous spaces}

Recall from Section~\ref{ssec:RM_natredHS} that a Riemannian Brownian motion on a reductive homogeneous space $M\cong G/H$ endowed with a $G$-invariant metric is obtained by first constructing an appropriate SDE in the Lie group $G$, and then factoring out $H$ via the natural projection $\pi \colon G \rightarrow G/H$. Also note that Corollary~\ref{cor:WZ rates compact RM} applies directly to Wong--Zakai approximations of SDEs in \emph{compact} Lie groups endowed with a left-invariant Riemannian metric, and Corollary~\ref{cor:WZ as rate RM} applies to all Lie groups endowed with a left-invariant Riemannian metric (left-invariance of the metric ensures geodesic completeness, as geodesics can be extended a uniform amount by translation). Thus, in order to apply Corollaries~\ref{cor:WZ rates compact RM} and~\ref{cor:WZ as rate RM} in the setting of reductive homogeneous spaces, all we need is the following lemma:

\begin{lemma}\label{lem:submersion}
Let $M\cong G/H$ be a reductive homogeneous space with reductive decomposition
$\mathfrak{g}=\mathfrak{h}\oplus\mathfrak{m}$, let $B$ be an
$\operatorname{Ad}_H$-invariant inner product on $\mathfrak{m}$, and let $M$ be
endowed with the $G$-invariant Riemannian metric $\bm{g}_B$ induced by $B$ (see~\eqref{eq:gB}). Let
$Q_{\mathfrak{h}}$ be an arbitrary inner product on $\mathfrak{h}$ and endow
$G$ with the left-invariant Riemannian metric $\bm{g}_G$ induced by the inner product
$Q = Q_{\mathfrak{h}}\oplus B$ on $\mathfrak{g}$ for which
$\mathfrak{h}\perp_Q\mathfrak{m}$. Then the natural projection
$\pi\colon G \rightarrow M$ is a Riemannian submersion. In particular, $\pi$ is distance-nonincreasing:
\begin{equation}\label{eq:pi_Lipschitz}
    d_{\bm{g}_B}(\pi(g_1),\pi(g_2)) \leq d_{\bm{g}_G}(g_1,g_2), \quad g_1,g_2\in G.
\end{equation}
\end{lemma}

\begin{proof}
Recall that $\pi\colon G\rightarrow G/H$ denotes the natural projection. For all $g\in G$ the vertical subspace of $T_g G$ with respect to $\pi$ is $\operatorname{ker} d\pi_g = (dL_g)_e\mathfrak{h}$ and,
by the orthogonality of the splitting and the left-invariance of the metric on
$G$, the horizontal subspace is $(dL_g)_e\mathfrak{m}$. Since
$\pi \circ L_g = \Phi_g \circ \pi$, we have for all $X\in \mathfrak{m}$ that
\begin{equation*}
    \| d\pi_g (dL_g)_e X \|_{T_{\pi(g)}M}
    = \| d(\Phi_g)_o \, d\pi_e X\|_{T_{\pi(g)}M}
    = \| d\pi_e X\|_{T_oM}
    = \|X\|_{B}
    = \|(dL_g)_e X\|_{T_gG},
\end{equation*}
where the second equality holds by the $G$-invariance of the metric on $M$ (see~\eqref{eq:gB}) and
the final equality by the left-invariance of the metric on $G$. Thus
$d\pi_g$ maps the horizontal subspace isometrically onto $T_{\pi(g)}M$, i.e.,
$\pi$ is a Riemannian submersion. Finally, if
$\gamma\colon [0,1]\rightarrow G$ is a smooth path connecting $g_1$ and $g_2$,
then $\pi\circ\gamma$ connects $\pi(g_1)$ and $\pi(g_2)$ and, decomposing
$\dot{\gamma}$ into horizontal and vertical components, one has
$\operatorname{length}(\pi\circ \gamma)\leq \operatorname{length}(\gamma)$;
taking the infimum over $\gamma$ yields~\eqref{eq:pi_Lipschitz}.
\end{proof}

By the discussion at the beginning of this section and the lemma above we thus obtain:

\begin{corollary}\label{cor:WZ_rates_homogeneous}
Let $M\cong G/H$ be a reductive homogeneous space with reductive decomposition
$\mathfrak{g}=\mathfrak{h}\oplus\mathfrak{m}$, let $B$ be an
$\operatorname{Ad}_H$-invariant inner product on $\mathfrak{m}$, and let $M$ be
endowed with the $G$-invariant Riemannian metric $\bm{g}_{B}$ induced by $B$ (see~\eqref{eq:gB}). Let
$Q_{\mathfrak{h}}$ be an arbitrary inner product on $\mathfrak{h}$ and endow
$G$ with the left-invariant Riemannian metric $\bm{g}_G$ induced by the inner product
$Q_{\mathfrak{h}}\oplus B$ on $\mathfrak{g}$ for which
$\mathfrak{h}\perp_Q\mathfrak{m}$. Let $V_0,V_1,\ldots,V_r \in \mathfrak{X}(G)$, let $(X(t))_{t\in [0,\tau^{\max})}$ denote the maximal solution to~\eqref{eq:RMsde} with $(M,\bm{g})=(G,\bm{g}_{G})$ and for $h>0$ let $(X_{h}(t))_{t \in [0,\tau^{\max}_{h})}$ denote the maximal solution to~\eqref{eq:WZ-RMsde}. Then
\begin{enumerate}
\item Assume that $X$ is a global solution, i.e., $\tau^{\max}=\infty$ a.s. Then for all $T>0$, $\eta \in (0,\frac{1}{2})$, and all $(h_n)_{n\in \N}$ converging to $0$ there exist almost surely finite random variables $N_0\colon \Omega \rightarrow \N\cup \{\infty\}$ and $K_0\colon \Omega \rightarrow [0,\infty]$ such that for all $\omega\in \Omega$ and all $n\geq N_0(\omega)$
it holds that $\tau_{h_n}^{\max}(\omega) >T $ and
\begin{equation}\label{eq:WZ as rate HS}
\sup_{t\in [0,T]} d_{\bm{g}_B}( \pi(X(t,\omega)) , \pi(X_{h_n}(t,\omega) ) ) \leq K_0(\omega) h_n^{\eta}.
\end{equation}
If moreover~\eqref{eq:WZ-RMsde} allows for a global solution for all $h=h_n$, $n\in \N$ (as is the case when $V_0,\ldots,V_r$ are left-invariant, see~\eqref{eq:sol_WZ_LG} below), then one can take $N_0\equiv 1$.
\item If $G$ is compact, then $X$ and $X_h$ exist globally for all $h>0$, and for all $T>0$, $\eta \in (0,\frac{1}{2})$, and $p\in [1,\infty)$ there exists a constant $C$ such that
\begin{equation}
\Big\| \sup_{t\in [0,T]} d_{\bm{g}_B}( \pi(X(t)) , \pi(X_{h}(t) ) ) \Big\|_{L^p(\Omega)}
\leq C h^{\eta}, \quad h>0.
\end{equation}
\end{enumerate}
\end{corollary}

\subsubsection{Representation of the Wong--Zakai approximation in a Lie group and in reductive homogeneous spaces}
Assume that the vector fields $V_0,\ldots,V_r \in \mathfrak{X}(G)$ in~\eqref{eq:RMsde} are \textbf{left-invariant vector fields}, so that they can be identified with elements $V_0,\ldots,V_r \in \mathfrak{g}$. For notational brevity we define $\Delta W_i^{k,h} = W_i(kh)-W_i((k-1)h)$, $k\in \N$ and $h>0$. Recall again that the flow of a left-invariant vector field on a Lie group is given by right-multiplication with the Lie exponential map (see, e.g.,~\cite[Chapter 20]{lee2012introductionsmooth}).
We thus obtain the following recursive formula for $X_h$, the Wong--Zakai approximation of $X$ (see~\eqref{eq:WZ-RMsde})
\begin{equation}\label{eq:sol_WZ_LG}
    X_h(t) = X_h(kh)\exp\Big( (t-kh) V_0+ \sum_{i=1}^{r}(t/h-k) \Delta W_i^{k+1,h}  V_i \Big),\quad t\in (kh,(k+1)h],\, k\in \N_0.
\end{equation}
Note that the exponential above is the Lie group exponential, and that it immediately follows from this representation that the Wong--Zakai approximations exist globally, even when $G$ is non-compact. The expression in~\eqref{eq:sol_WZ_LG} is the noncommutative analogue of approximating a Stratonovich stochastic exponential by piecewise constant ODE flows driven by Brownian increments. In particular, as already mentioned in the introduction, almost sure convergence rates of $X_{h_n}$, $n\in \N$, \emph{along a dyadic grid} ($h_n = 2^{-n}$) have already been established in~\cite[Section 4.8]{mckean1969stochastic} (i.e., in the Lie group setting, when it comes to almost sure convergence, the only improvement obtained from Corollary~\ref{cor:WZ as rate RM} when compared to~\cite[Section 4.8]{mckean1969stochastic}, is that almost sure convergence holds along arbitrary sequences $h_n\searrow 0$). \par
Now suppose $M\cong G/H$ is a reductive homogeneous space with decomposition $\mathfrak{g}=\mathfrak{h}\oplus \mathfrak{m}$, endowed with a $G$-invariant metric arising from an $\operatorname{Ad}_H$-invariant inner product $B$ on $\mathfrak{m}$. In Section~\ref{ssec:RM_natredHS} we demonstrated that a Riemannian Brownian motion on $M$ is given by $t\mapsto \pi(x g(t))$, where $(g(t))_{t\geq 0}$ is the solution to the SDE in $G$ given by~\eqref{eq:stratsde}, provided $G$ is unimodular or $B$ renders $M$ naturally reductive, see Proposition~\ref{prop:unimodularredhombm}. \par
Rephrasing~\eqref{eq:stratsde} as an equation of the form~\eqref{eq:RMsde} we have $\xi=e$ (the identity element of $G$), $r=\operatorname{dim}\mathfrak{m}$, $V_0 = 0$, and $V_i(g) = (\textnormal{d}L_g)_e e_i = \tilde{e}_i(g)$, $g\in G$, with $e_1,\ldots, e_{r}$ an orthonormal basis of $\mathfrak{m}\subseteq \mathfrak{g}$ (see~\eqref{eq:def_tilde_e}), and we obtain
\begin{equation}
    X_h(t) = X_h(kh)\exp\Big( \sum_{i=1}^{r}(t/h-k) \Delta W_i^{k+1,h}  e_i \Big),\quad t\in (kh,(k+1)h],\, k\in \N_0,
\end{equation}
and thus
\begin{equation}\label{eq:prewongzakai}
    X_h(kh) = \exp\left(\sum_{i=1}^{r}  \Delta W_i^{1,h} e_i\right) \cdot \ldots \cdot \exp\left(\sum_{i=1}^{r}  \Delta W_i^{k,h} e_i\right),\quad k\in \N.
\end{equation}

Finally, we note that the order in which the increments appear in~\eqref{eq:prewongzakai} deserves a comment.

\begin{remark}\label{rem:projection_on_M}
For notational brevity, set $g_k = \exp\big(\sum_{i=1}^{r} \Delta W_i^{k,h} e_i\big)$, so
that~\eqref{eq:prewongzakai} reads $X_h(kh) = g_1 g_2 \cdots g_k$. Since $\Phi$ is a left
action, projecting to $M$ gives
\begin{equation*}
    \pi(X_h(kh)) = \Phi_{g_1}\circ \cdots \circ \Phi_{g_k}(o),
\end{equation*}
in which the \emph{most recent} increment is the innermost map, applied to $o$ first. This
is not a reversal of time. The update $X_h((k+1)h) = X_h(kh)\,g_{k+1}$ is a \emph{right}
translation on $G$, which does not descend to a map of $G/H$: the increment $g_{k+1}$ is a
step expressed at the base point (in the frame $\mathfrak{m}\cong T_oM$), and right
multiplication by $g_{k+1}$ transports this step to the current location via the
accumulated history $X_h(kh)$ --- the development (``rolling'') picture. To read the
process as a chronological composition of maps on $M$, one first converts each increment
to its ``space-frame'' version by conjugation,
$\tilde{g}_{k+1} := X_h(kh)\, g_{k+1}\, X_h(kh)^{-1}$, so that
$X_h((k+1)h) = \tilde{g}_{k+1} X_h(kh)$ and hence
\begin{equation*}
    \pi(X_h(kh)) = \Phi_{\tilde{g}_k}\circ \cdots \circ \Phi_{\tilde{g}_1}(o),
\end{equation*}
with the increments now appearing in chronological order. The two expressions are
algebraically identical: a left-to-right chronological product of body-frame increments
equals the right-to-left product of the corresponding space-frame increments, as familiar
from the composition of intrinsic versus extrinsic rotations in rigid-body kinematics.
\end{remark}

\begin{remark}\label{rem:WZ_is_GRW}
For naturally reductive homogeneous spaces, geodesics through the basepoint $o$ are projections of one-parameter subgroups (see page~\ref{def:natredHS}). In particular, the Wong--Zakai approximation of \eqref{eq:sol_WZ_LG} can be interpreted as a geodesic random walk. As properly rescaled geodesic random walks have been shown to converge in law to a Riemannian Brownian motion (see \cite{jorgensen1975central}), this provides yet another way to establish that the process $\pi(g)$ in Proposition~\ref{prop:unimodularredhombm} is indeed a Riemannian Brownian motion when $M=G/H$ is a naturally reductive homogeneous space.
\end{remark}

\begin{remark}
The advantage of considering left-invariant vector fields in~\eqref{eq:RMsde} when $M$ is a Lie group, is that these allow for an expression of the Wong--Zakai approximation that is explicit up to Lie group exponential maps, see~\eqref{eq:sol_WZ_LG}. Such an explicit expression does not exist for general vector fields. However, more general vector fields are of interest in the context of sampling using Langevin dynamics, as recently explored in~\cite[Section 5]{luesink2025symplectic}. For such equations, splitting schemes may provide a way out; the convergence of such schemes is discussed in~\cite{LewisLuesinkSchieven:2026}.
\end{remark}

\subsubsection{Replacing the exponential by a cheaper map into \texorpdfstring{$G$}{G}}\label{ssec:retraction}
For a matrix group $G\subseteq\R^{n\times n}$, the only non-trivial operation in~\eqref{eq:prewongzakai} is the matrix exponential of the increments $\Delta_k := \sum_{i=1}^{r}\Delta W_i^{k,h}e_i\in\mathfrak{m}$, and it is natural to ask whether $\exp$ may be replaced by a cheaper map $\mathcal{R}\colon\mathfrak{g}\rightarrow\R^{n\times n}$ without affecting the convergence rate (such approximations are called retractions in the context of Riemannian manifolds). Since $|\Delta_k| = O(h^{1/2})$, a map satisfying
\begin{equation}\label{eq:retraction_condition}
    \mathcal{R}(A) = \exp(A) + O(|A|^3)\quad\text{as } A\to 0
\end{equation}
changes each factor in~\eqref{eq:prewongzakai} by $O(h^{3/2})$; writing the difference of the two products $\mathcal{R}(\Delta_1)\cdots\mathcal{R}(\Delta_k)$ and $\exp(\Delta_1)\cdots\exp(\Delta_k)$ as a telescoping sum of $k\leq T/h$ such terms, one obtains that the modified scheme differs from the Wong--Zakai approximation by $O(h^{1/2})$ (and in fact by $O(h)$, as the leading error term $\frac16\Delta_k^3$ has mean zero), so that the rates of Corollary~\ref{cor:WZ_rates_homogeneous} are retained; in our experiments the error constants were indistinguishable from those of~\eqref{eq:prewongzakai}. Condition~\eqref{eq:retraction_condition} cannot be weakened to a first-order approximation: $\mathcal{R}(A)=I_n+A$ yields the Euler--Maruyama scheme for the Stratonovich equation~\eqref{eq:stratsde}, which converges to the solution of the It\^o equation with the same coefficients, i.e., to the wrong process. Besides the cost, the second requirement on $\mathcal{R}$ is that it should map $\mathfrak{m}$ into $G$, so that the approximation stays on $M$ by construction, as the exponential does. First of all we note that for the spheres and hyperbolic spaces (the examples considered in Sections~\ref{sec:spheres} and~\ref{sec:hyperbolics}) there is nothing to gain, as the exponential of an element of $\mathfrak{m}$ is available in closed form, see~\eqref{eq:sphere_exp} and~\eqref{eq:hyperbolic_exp}. For the other examples we consider, natural choices for $\mathcal{R}$ are:
\begin{enumerate}[(i)]
\item \emph{The Taylor polynomial} $\mathcal{R}(A) = I_n + A + \tfrac12A^2$: for $G=\operatorname{GL}(n)$ and $\mathfrak{m}=\operatorname{Sym}(n)$ (the setting in Section~\ref{sec:spd}) this choice of $\mathcal{R}$ maps $\mathfrak{m}$ into $G$: indeed, for symmetric $S$ the matrix $I_n+S+\tfrac12S^2$ has eigenvalues $1+s+\tfrac12s^2\geq\tfrac12$, hence is positive definite, so that all iterates are invertible and $X_h(kh)=g_kg_k^{\intercal}$ is symmetric positive definite. For the examples considered in Sections~\ref{sec:spheres} and~\ref{sec:hyperbolics}, the Taylor polynomial fails to map $\mathfrak{m}$ into the group $G$.  (Note that the scheme resulting from this retraction is precisely the stochastic Heun scheme, i.e., the explicit trapezoidal rule~\cite{Ruemelin:1982}, applied to~\eqref{eq:stratsde} on the vector space $\R^{n\times n}$: as the vector fields $g\mapsto ge_i$ are linear, predictor $\tilde{g}=g_k+g_k\Delta_k$ and corrector $g_{k+1}=g_k+\tfrac12(g_k+\tilde{g})\Delta_k$ combine to $g_{k+1}=g_k(I_n+\Delta_k+\tfrac12\Delta_k^2)$.)
\item \emph{The Cayley transform} $\mathcal{R}(A) = (I_n-\tfrac12A)^{-1}(I_n+\tfrac12A)$, a classical ingredient of Lie group integrators~\cite{IserlesEtAl:2000}, which costs one linear solve per step. It satisfies~\eqref{eq:retraction_condition} and maps the Lie algebra of any quadratic group into the group: if $A^{\intercal}J=-JA$, then $\mathcal{R}(A)^{\intercal}J\mathcal{R}(A)=J$. It is therefore the natural substitute for $\exp$ in the examples of Sections~\ref{sec:flag} and~\ref{sec:stiefel}, where $G=\operatorname{O}(n)$ and the increments are general skew-symmetric matrices. (For $G=\operatorname{GL}(n)$ the matrix $I_n-\tfrac12\Delta_k$ is ill-conditioned whenever an eigenvalue of $\Delta_k$ is close to $2$, so that choice~(i) is preferable there.)
\end{enumerate}
\section{Examples}\label{sec:examples}
Having introduced an efficient method for the pathwise approximation of a Riemannian Brownian motion on a reductive homogeneous space, we now turn to concrete examples. In the sections below we discuss five examples in increasing order of geometric complexity, each example introducing an additional feature in the associated metric compared to the previous one. For each of these examples we
\begin{enumerate}
    \item provide an explicit representation of the space as a reductive homogeneous space $M\cong G/H$ with reductive decomposition $\mathfrak{g}=\mathfrak{h}\oplus\mathfrak{m}$ (see Definition~\ref{def:rhs});
    \item identify the family of $G$-invariant Riemannian metrics on $M$, i.e., the family of $\operatorname{Ad}_H$-invariant inner products $B$ on $\mathfrak{m}$ (see~\eqref{eq:gB});
    \item verify that Proposition~\ref{prop:unimodularredhombm} applies, i.e., that $G$ is unimodular or that $B$ renders $M$ naturally reductive, so that the projection $\pi(g)$ of the solution to~\eqref{eq:stratsde} is a Riemannian Brownian motion on $(M,\bm{g}_B)$;
    \item provide plots of an (approximate) simulation of the Riemannian Brownian motion on $M$, obtained by means of the Wong--Zakai approximation~\eqref{eq:prewongzakai} discussed in Section~\ref{ssec:WZ_LG_HS}.
\end{enumerate}
Recall that the solution to~\eqref{eq:stratsde} exists globally (see Section~\ref{ssec:RM_natredHS}), so that Corollary~\ref{cor:WZ_rates_homogeneous}(i) guarantees almost sure pathwise convergence rates for the above-mentioned Wong--Zakai approximations in all five examples. In three of the five examples, namely the spheres, the flag manifolds and the Stiefel manifolds, the transitive action is by a compact group, and for these Corollary~\ref{cor:WZ_rates_homogeneous}(ii) moreover guarantees convergence rates in $L^p(\Omega)$.

The simplest example we consider is the sphere (Section~\ref{sec:spheres}): the action is by the compact group $\operatorname{SO}(n+1)$, the isotropy representation is irreducible so that the $G$-invariant metric is unique up to scaling, and the resulting space is a naturally reductive homogeneous space of constant positive curvature. The second example is the hyperbolic space (Section~\ref{sec:hyperbolics}). This is again a naturally reductive homogeneous space with a $G$-invariant metric that is unique up to scaling, now of constant negative curvature, but the acting group is no longer compact. Our third example is the manifold of symmetric positive-definite (SPD) matrices (Section~\ref{sec:spd}). This manifold is non-compact, its isotropy representation splits into two irreducible components, and accordingly it admits a two-parameter family of $G$-invariant metrics, all of which render it naturally reductive; the curvature is nonpositive but no longer constant. Nevertheless, the geodesic distance can be computed in closed form. As a fourth example we consider flag manifolds (Section~\ref{sec:flag}). Here the isotropy representation splits into pairwise non-isomorphic irreducible components, one for each pair of blocks, and there is one free scaling parameter per component. The normal metric renders the flag manifold naturally reductive, but for general choices of the scaling parameters this is no longer the case; the flag manifold is then merely a reductive homogeneous space, and it is the unimodularity of the compact group $\operatorname{O}(n)$ that guarantees that Proposition~\ref{prop:unimodularredhombm} still applies. Finally, Stiefel manifolds (Section~\ref{sec:stiefel}) are similar in structure to flag manifolds, but their isotropy representation contains \emph{isomorphic} irreducible components, so that the $G$-invariant metrics need not be diagonal with respect to the decomposition of $\mathfrak{m}$ into irreducible components.

\subsection{Spheres}\label{sec:spheres}
\paragraph{The $n$-sphere as a reductive homogeneous space.}
For $n\in \N$, the $n$-dimensional unit sphere $S^n=\{x\in\R^{n+1}\,|\, \|x\|=1\}$ is a homogeneous space: the special orthogonal group $G=\operatorname{SO}(n+1)$ acts transitively on $S^n\subseteq\R^{n+1}$ by matrix multiplication,
\begin{equation}
    \Phi_g(x) = gx, \quad g\in \operatorname{SO}(n+1),\, x\in S^n.
\end{equation}
Let $\bm{e}_1,\ldots, \bm{e}_{n+1}$ denote the standard basis of $\R^{n+1}$ (throughout this section, standard basis vectors of Euclidean spaces are written in bold, in order to distinguish them from the basis $e_1,\ldots,e_r$ of $\mathfrak{m}$ used in Sections~\ref{sec:redhomspaces} and~\ref{sec:WZ}). We fix the base point $o=\bm{e}_{n+1}\in S^n$ (the `north pole'). The isotropy group $H$ of $o$ consists of the block matrices
\begin{equation}\label{eq:sphere_h}
H = \left\{
    \begin{pmatrix}
        h & \bm{0}\\
        \bm{0}^{\intercal} & 1
    \end{pmatrix}\,\Bigg|\, h\in \operatorname{SO}(n)\right\},
\end{equation}
so that $H$ is isomorphic to $\operatorname{SO}(n)$ and $S^n\cong G/H \cong \operatorname{SO}(n+1)/\operatorname{SO}(n)$.

Recall that the Lie algebra $\mathfrak{so}(n+1)$ of $\operatorname{SO}(n+1)$ is the space of skew-symmetric $(n+1)\times (n+1)$ matrices. The Lie algebra of $H$ is
\begin{equation}
    \mathfrak{h}=\left\{
    \begin{pmatrix}
        A & \bm{0}\\
        \bm{0}^{\intercal} & 0
    \end{pmatrix}\,\Bigg|\, A\in\mathfrak{so}(n)\right\}\cong \mathfrak{so}(n),
\end{equation}
and, writing every skew-symmetric matrix in block form, one sees that $\mathfrak{so}(n+1)=\mathfrak{h}\oplus\mathfrak{m}$ with
\begin{equation}\label{eq:snmatrep}
\mathfrak{m}=\left\{
    X_{\bm{v}}:=\begin{pmatrix}
        0_{n\times n} & \bm{v}\\
        -\bm{v}^{\intercal} & 0
    \end{pmatrix}\,\Bigg|\, \bm{v}\in\R^n\right\}.
\end{equation}
In particular $\mathfrak{m}\cong \R^n$ via $X_{\bm{v}}\mapsto \bm{v}$. Moreover, the isomorphism $\mathfrak{m}\cong T_oS^n$ of~\eqref{eq:link_m_ToM} is given explicitly by
\begin{equation}\label{eq:sphere_ToSn}
    X_{\bm{v}}^*(o) = \frac{\bm{d}}{\bm{d}t}\Big|_{t=0}\exp(tX_{\bm{v}})\,\bm{e}_{n+1} = X_{\bm{v}}\bm{e}_{n+1} = \begin{pmatrix} \bm{v}\\ 0\end{pmatrix}\in T_oS^n = \bm{e}_{n+1}^{\perp}.
\end{equation}
We now verify that the decomposition $\mathfrak{so}(n+1)=\mathfrak{h}\oplus\mathfrak{m}$ is reductive (see Definition~\ref{def:rhs}).

\begin{proposition}\label{prop:sphere_brackets}
With $\mathfrak{h}$ and $\mathfrak{m}$ as above, the following inclusions hold:
\begin{equation}
    [\mathfrak{h},\mathfrak{h}]\subseteq \mathfrak{h}, \quad [\mathfrak{h},\mathfrak{m}]\subseteq\mathfrak{m}, \quad [\mathfrak{m},\mathfrak{m}]\subseteq \mathfrak{h}.
\end{equation}
\end{proposition}

\begin{proof}
    The first inclusion holds as $\mathfrak{h}$ is the Lie algebra of the subgroup $H$ (explicitly: the upper left block of a commutator of two elements of $\mathfrak{h}$ is a commutator in $\mathfrak{so}(n)$). For the second inclusion, we compute for $A\in\mathfrak{so}(n)$ and $\bm{v}\in\R^n$
    \begin{equation}\label{eq:sphere_hm_bracket}
        \left[\begin{pmatrix}
            A & \bm{0} \\
            \bm{0}^{\intercal} & 0
        \end{pmatrix},\begin{pmatrix}
            0_{n\times n} & \bm{v}\\
            -\bm{v}^{\intercal} & 0
        \end{pmatrix}\right] = \begin{pmatrix}
            0_{n\times n} & A\bm{v}\\
            \bm{v}^{\intercal}A & 0
        \end{pmatrix}
        = \begin{pmatrix}
            0_{n\times n} & A\bm{v}\\
            -(A\bm{v})^{\intercal} & 0
        \end{pmatrix} = X_{A\bm{v}},
    \end{equation}
    where we used $A^{\intercal}=-A$. Hence $[\mathfrak{h},\mathfrak{m}]\subseteq \mathfrak{m}$. For the final inclusion, we take $\bm{u},\bm{v}\in\R^n$ and compute
\begin{equation}
        \left[\begin{pmatrix}
            0_{n\times n} & \bm{u} \\
            -\bm{u}^{\intercal} & 0
        \end{pmatrix},\begin{pmatrix}
            0_{n\times n} & \bm{v}\\
            -\bm{v}^{\intercal} & 0
        \end{pmatrix}\right] = \begin{pmatrix}
            -\bm{u}\bm{v}^{\intercal} + \bm{v}\bm{u}^{\intercal} & \bm{0} \\
            \bm{0}^{\intercal} & 0
        \end{pmatrix}.
    \end{equation}
    Since the upper left block is skew-symmetric, it follows that $[\mathfrak{m},\mathfrak{m}]\subseteq\mathfrak{h}$.
\end{proof}

As $H\cong \operatorname{SO}(n)$ is connected and $[\mathfrak{h},\mathfrak{m}]\subseteq\mathfrak{m}$, Lemma~\ref{lem:connected_H_reductive} shows that $S^n\cong \operatorname{SO}(n+1)/\operatorname{SO}(n)$ is a reductive homogeneous space with reductive decomposition $\mathfrak{so}(n+1)=\mathfrak{h}\oplus\mathfrak{m}$. (Alternatively, this follows from Lemma~\ref{prop:properaction_redhomospace}, as $H$ is compact.) Moreover, since $[\mathfrak{m},\mathfrak{m}]\subseteq\mathfrak{h}$, Lemma~\ref{lemma: LTS} shows that \emph{every} $\operatorname{Ad}_H$-invariant inner product on $\mathfrak{m}$ renders $S^n$ a naturally reductive homogeneous space (see Definition~\ref{def: naturally reductive space}). In fact, $\mathfrak{h}$ and $\mathfrak{m}$ are the $+1$ and $-1$ eigenspaces of the involutive automorphism $X\mapsto sXs$ of $\mathfrak{so}(n+1)$, $s=\operatorname{diag}(I_n,-1)$, so $S^n$ is a symmetric space (cf.\ the discussion following Lemma~\ref{lemma: LTS}).

\paragraph{The $n$-sphere as a naturally reductive homogeneous space: identifying the metric.}
We next address the question of which $G$-invariant metrics $S^n$ admits, i.e., which $\operatorname{Ad}_H$-invariant inner products $\mathfrak{m}$ admits. Intuitively it is clear that there is, up to scaling, only one candidate. This can be made rigorous using the following proposition. Note that for $h=\operatorname{diag}(A,1)\in H$ and $\bm{v}\in\R^n$ one has
\begin{equation}\label{eq:sphere_Ad}
    \operatorname{Ad}_h X_{\bm{v}} = hX_{\bm{v}}h^{-1} = \begin{pmatrix}
            A & \bm{0} \\
            \bm{0}^{\intercal} & 1
        \end{pmatrix}\begin{pmatrix}
            0_{n\times n} & \bm{v}\\
            -\bm{v}^{\intercal} & 0
        \end{pmatrix}\begin{pmatrix}
            A^{\intercal} & \bm{0} \\
            \bm{0}^{\intercal} & 1
        \end{pmatrix} = X_{A\bm{v}},
\end{equation}
i.e., under the identification $\mathfrak{m}\cong\R^n$ the isotropy representation~\eqref{eq:isotropyrep} is the standard representation of $\operatorname{SO}(n)$ on $\R^n$ (this is the group-level version of~\eqref{eq:sphere_hm_bracket}).

\begin{proposition}\label{prop:irrepsphere}
    The isotropy representation $\rho\colon H\cong\operatorname{SO}(n) \rightarrow \operatorname{GL}(\mathfrak{m})\cong \operatorname{GL}(\R^n)$ (see~\eqref{eq:isotropyrep} and~\eqref{eq:sphere_Ad}) is irreducible for every $n\geq 1$.
\end{proposition}

\begin{proof}
    For $n=1$ the representation space is one-dimensional, hence irreducible. Let $n\geq 2$ and let $W\subseteq\R^{n}$ be a nonzero subspace that is invariant under $\rho$. Choose $\bm{v}\in W$ with $\|\bm{v}\|=1$. As the action of $\operatorname{SO}(n)$ on the unit sphere $S^{n-1}\subseteq\R^{n}$ is transitive, for every $\bm{u}\in S^{n-1}$ there exists $A\in \operatorname{SO}(n)$ such that $A\bm{v} = \bm{u}$, and by invariance of $W$ we obtain $\bm{u}\in W$. Thus $W$ contains every unit vector, hence every vector, of $\R^{n}$, i.e., $W=\R^{n}$. Consequently there are no nontrivial proper invariant subspaces and $\rho$ is irreducible.
\end{proof}

Proposition~\ref{prop:irrepsphere} implies that the $\operatorname{Ad}_H$-invariant inner product on $\mathfrak{m}$ is unique up to a positive multiplicative constant. Indeed, if $B$ and $B'$ are two such inner products, then $B'(X,Y)=B(SX,Y)$ for a linear map $S\colon\mathfrak{m}\to\mathfrak{m}$ that is $B$-self-adjoint and positive definite; the $\operatorname{Ad}_H$-invariance of both $B$ and $B'$ implies that $S$ commutes with $\rho(h)$ for all $h\in H$, so that every eigenspace of $S$ is $\rho$-invariant and hence, by irreducibility, equal to $\mathfrak{m}$. Thus $S=\lambda\,\operatorname{id}_{\mathfrak{m}}$ for some $\lambda>0$ (this is Schur's lemma). Consequently, the $\operatorname{SO}(n+1)$-invariant Riemannian metric on $S^n$ is unique up to a positive constant multiple.

As our $\operatorname{Ad}_H$-invariant inner product on $\mathfrak{m}$ we take the normalised Frobenius inner product
\begin{equation}\label{eq:sphere_B}
    B(X,Y) := \tfrac{1}{2}\tr(X^{\intercal} Y), \quad X,Y\in\mathfrak{m}.
\end{equation}
This inner product is $\operatorname{Ad}_H$-invariant since $\tr\big((hXh^{-1})^{\intercal}hYh^{-1}\big)=\tr(hX^{\intercal}h^{\intercal}hYh^{\intercal})=\tr(X^{\intercal}Y)$ for orthogonal $h$. In fact, $B$ is the restriction to $\mathfrak{m}$\ of the $\operatorname{Ad}_{\operatorname{SO}(n+1)}$-invariant inner product $Q(X,Y)=\frac12\tr(X^{\intercal}Y)$ on $\mathfrak{so}(n+1)$, and $\mathfrak{m}=\mathfrak{h}^{\perp}$ with respect to $Q$ by the block structure, so $B$ is the normal metric of Proposition~\ref{prop:normalhomogeneous}. For $X=X_{\bm{u}}$ and $Y=X_{\bm{v}}$ as in~\eqref{eq:snmatrep} we compute
\begin{equation}
    B(X_{\bm{u}},X_{\bm{v}}) = \tfrac{1}{2}\tr\left[\begin{pmatrix}
        0_{n\times n} & -\bm{u}\\
        \bm{u}^{\intercal} &  0
    \end{pmatrix} \begin{pmatrix}
        0_{n\times n} & \bm{v}\\
        -\bm{v}^{\intercal} & 0
    \end{pmatrix}\right] = \tfrac{1}{2}\tr\begin{pmatrix}
        \bm{u}\bm{v}^{\intercal} & \bm{0}\\
        \bm{0}^{\intercal} &  \bm{u}^{\intercal}\bm{v}
    \end{pmatrix} = \bm{u}\cdot\bm{v},
\end{equation}
so that, under the identifications $\mathfrak{m}\cong\R^n\cong T_oS^n$ of~\eqref{eq:snmatrep} and~\eqref{eq:sphere_ToSn}, $B$ is the standard Euclidean inner product on $T_oS^n=\bm{e}_{n+1}^{\perp}\subseteq\R^{n+1}$; the factor $\frac12$ in~\eqref{eq:sphere_B} is chosen precisely to achieve this. The $G$-invariant metric $\bm{g}_B$ induced by $B$ (see~\eqref{eq:gB}) is obtained by transporting $B$ from $o$ to an arbitrary point $x=\Phi_g(o)=g \bm{e}_{n+1}$: since the action is linear, $\textnormal{d}(\Phi_{g^{-1}})_x V = g^{-1}V$ for $V\in T_xS^n\subseteq\R^{n+1}$, and thus, as $g$ is orthogonal,
\begin{equation}
    \bm{g}_B(x;V,W) = (g^{-1}V)\cdot(g^{-1}W) = V\cdot W, \quad V,W\in T_xS^n.
\end{equation}
In other words, $\bm{g}_B$ is the round metric on the unit sphere, i.e., the metric induced by the Euclidean metric on $\R^{n+1}$, with geodesic distance $d(x,y) = \arccos(x\cdot y)$, $x,y\in S^n$.

\paragraph{Constructing the Riemannian Brownian motion and simulations.}
We conclude that both conditions of Proposition~\ref{prop:unimodularredhombm} are satisfied: $G=\operatorname{SO}(n+1)$ is compact and hence unimodular, and $(S^n,B)$ is naturally reductive. To construct the Riemannian Brownian motion on $(S^n,\bm{g}_B)$ as in Section~\ref{ssec:RM_natredHS}, we note that $e_i':=X_{\bm{e}_i}$, $i=1,\ldots,n$ (with $\bm{e}_1,\ldots,\bm{e}_n$ the standard basis of $\R^n$) is a $B$-orthonormal basis of $\mathfrak{m}$. Taking $n$ independent standard real-valued Brownian motions $\bm{W} = (W_1, \ldots, W_n)$, the $\mathfrak{m}$-valued Brownian motion~\eqref{eq:mbrownianmotion} is given by
\begin{equation}
    W_\mathfrak{m}(t) = \sum_{i=1}^n e_i' W_i(t) =
    \begin{pmatrix}
        0_{n\times n} & \bm{W}(t)\\
        -\bm{W}(t)^{\intercal} & 0
    \end{pmatrix},\quad t\geq 0.
\end{equation}
By Proposition~\ref{prop:unimodularredhombm}, if $(g(t))_{t\geq 0}$ is the $\operatorname{SO}(n+1)$-valued solution to~\eqref{eq:stratsde} driven by $W_{\mathfrak{m}}$, then $X(t)=\pi(g(t))=g(t)\bm{e}_{n+1}$, the last column of $g(t)$, is a Riemannian Brownian motion on $S^n$ starting at $o$. For the simulation we use the Wong--Zakai approximation~\eqref{eq:prewongzakai}: with $g_k=\exp(X_{\Delta\bm{W}^{k,h}})$, $\Delta\bm{W}^{k,h}=\bm{W}(kh)-\bm{W}((k-1)h)$, we set $X_h(kh)=\pi(g_1\cdots g_k) = g_1\cdots g_k\,\bm{e}_{n+1}$ (here and in the following examples, $X_h$ denotes the \emph{projected} Wong--Zakai approximation, i.e., the image under $\pi$ of the $G$-valued process in~\eqref{eq:prewongzakai}). Note moreover that the Lie group exponential of an element of $\mathfrak{m}$ is the matrix exponential of $X_{\bm{v}}$, which is available in closed form: as $X_{\bm{v}}^3=-\|\bm{v}\|^2X_{\bm{v}}$, one has, for $\theta=\|\bm{v}\|>0$,
\begin{equation}\label{eq:sphere_exp}
    \exp(X_{\bm{v}}) = I_{n+1} + \frac{\sin\theta}{\theta}X_{\bm{v}} + \frac{1-\cos\theta}{\theta^2}X_{\bm{v}}^2.
\end{equation}
In particular,
\begin{equation}
    \exp(X_{\bm{v}})\,\bm{e}_{n+1} = \begin{pmatrix} \frac{\sin\theta}{\theta}\,\bm{v}\\ \cos\theta\end{pmatrix},
\end{equation}
which confirms that $t\mapsto \pi(\exp(tX_{\bm{v}}))$ is the great circle through $o$ with initial velocity $(\bm{v},0)^{\intercal}$, traversed at unit speed if $\|\bm{v}\|=1$, as it must be for a naturally reductive homogeneous space (see Section~\ref{ssec:nrHS}). This confirms that the Wong--Zakai approximation is indeed a geodesic random walk on $S^n$. A resulting trajectory is shown in Figure~\ref{fig:spherebmtraj}, and the associated distance process $t\mapsto d(o,X_h(t))$ in Figure~\ref{fig:spherebmdist}. The distance process is bounded by the diameter $\pi$ of $S^n$, which is attained at the south pole $-\bm{e}_{n+1}$; this is to be contrasted with the hyperbolic case treated in the next section, where the distance process grows without bound.

\begin{figure}[H]
    \centering
    \includegraphics[width=0.9\textwidth]{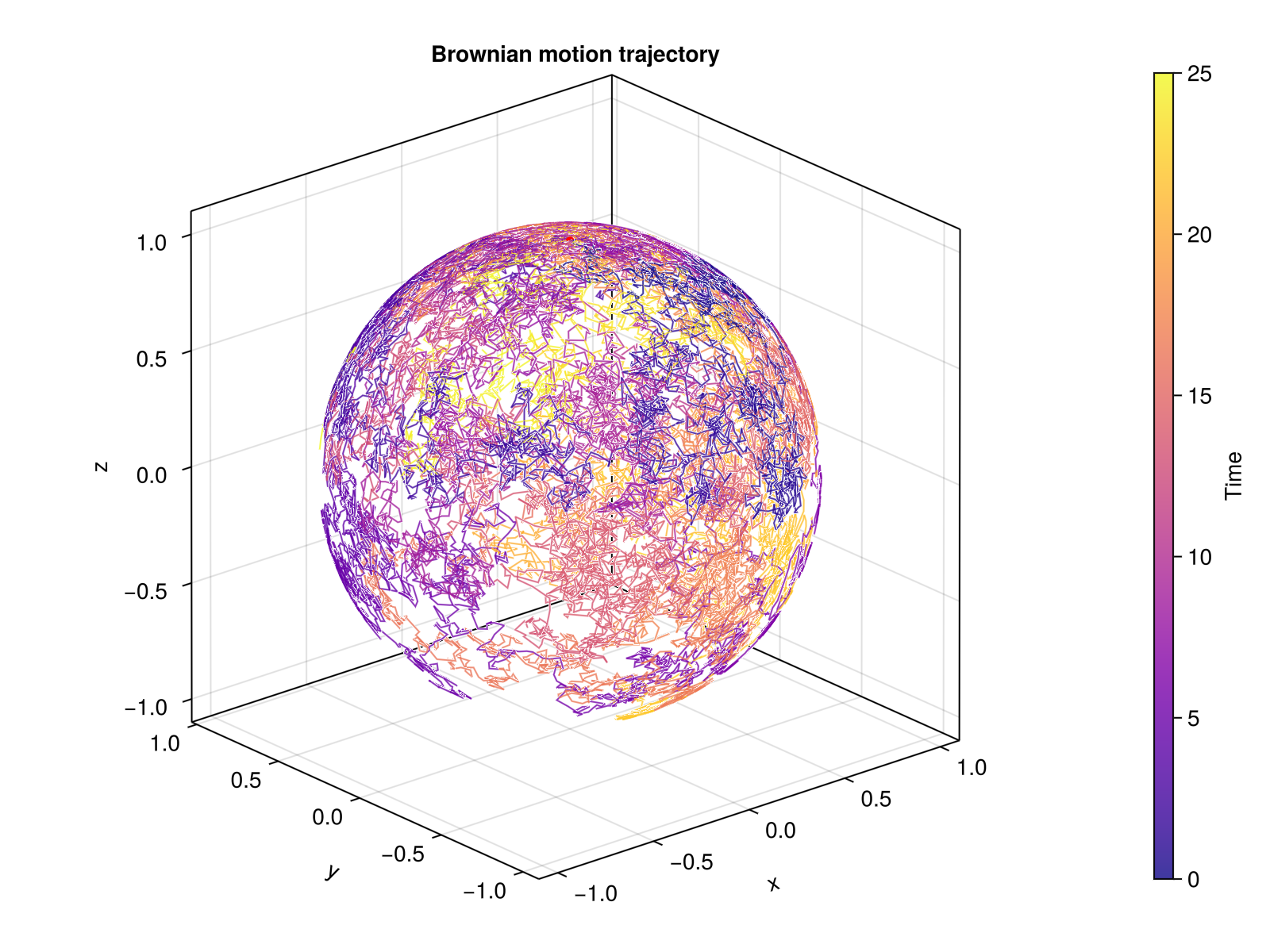}
    \caption{A realisation of the Wong--Zakai approximation of the Riemannian Brownian motion on the sphere $S^2$, started at the north pole $o=\bm{e}_3$. The step-size is $h=0.001$ for this simulation.}
    \label{fig:spherebmtraj}
\end{figure}

\begin{figure}[H]
    \centering
    \includegraphics[width=0.9\textwidth]{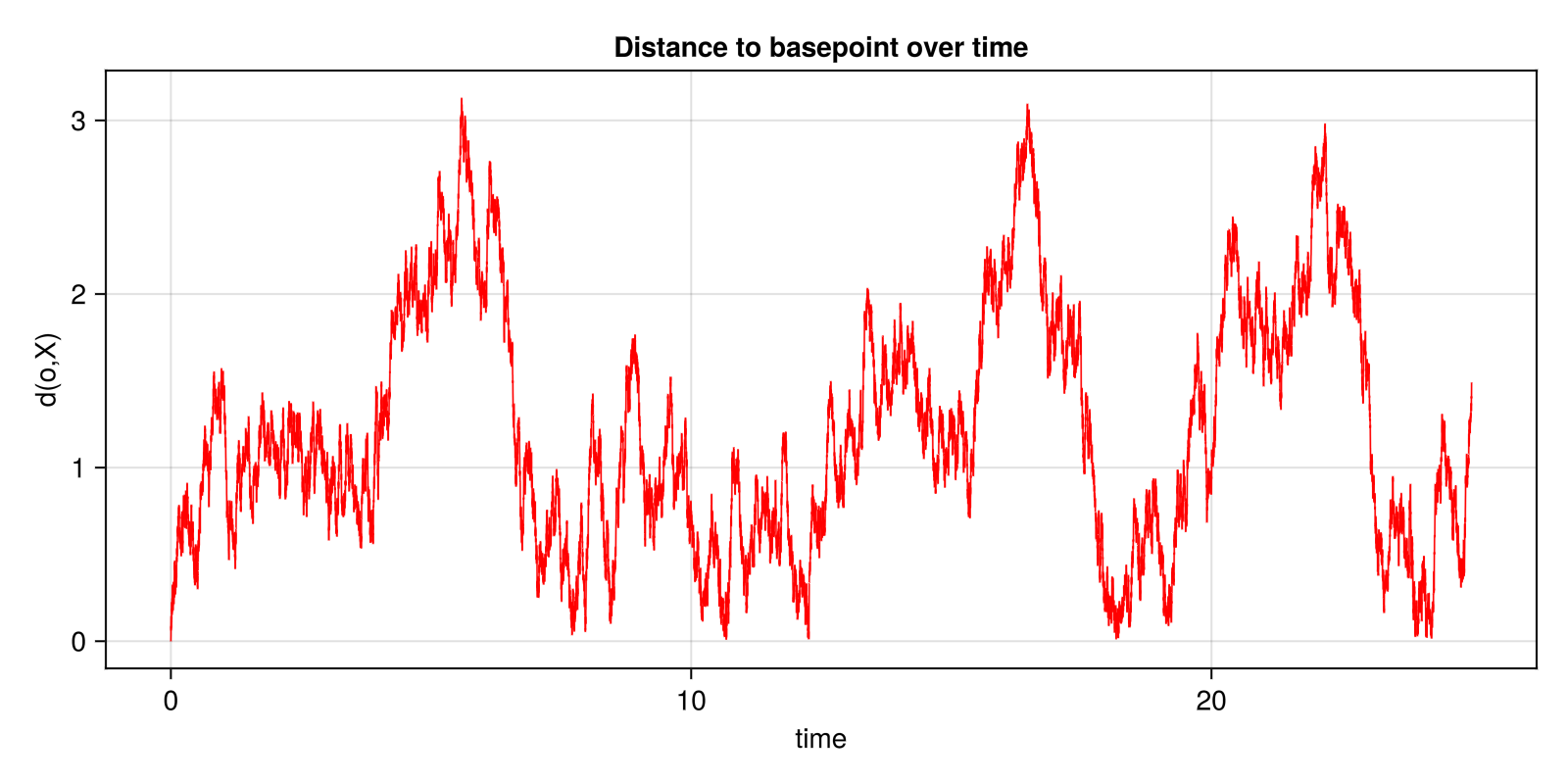}
    \caption{The geodesic distance process $t\mapsto d(o,X_h(t))$ associated with the realisation of Figure~\ref{fig:spherebmtraj}. The distance process is bounded by $\pi$, the maximal distance from the north pole, which is attained at the south pole. The step-size is $h=0.001$.}
    \label{fig:spherebmdist}
\end{figure}

\subsection{Hyperbolic spaces}\label{sec:hyperbolics}

\paragraph{The hyperbolic space as a naturally reductive homogeneous space.} For $n\in\N$, the $n$-dimensional hyperbolic space is the upper sheet of the two-sheeted hyperboloid,
\begin{equation}
\mathbb{H}^n = \{x=(x_0,x_1,\ldots,x_n)\in \mathbb{R}^{n+1}\,|\, \langle x,x\rangle_{1,n} = -1,\,\, x_0>0\},\qquad \langle x,y\rangle_{1,n} := -x_0y_0 + x_1y_1 + \cdots + x_ny_n,
\end{equation}
where $\langle\cdot,\cdot\rangle_{1,n}$ denotes the Lorentzian inner product on $\R^{n+1}$; the ambient space is written $\mathbb{R}^{1,n}$ to emphasise the presence of this inner product. The hyperbolic space is an important example of a non-compact naturally reductive homogeneous space. One may verify that the Lorentzian inner product is positive definite when restricted to $T\mathbb{H}^n$, thus rendering $\mathbb{H}^n$ a Riemannian manifold. The resulting geodesic distance is known to be given by
\begin{equation}\label{eq:distancehyperbolic}
    d(x,y)=\operatorname{arccosh}(-\langle x, y\rangle_{1,n}),\quad x,y\in \mathbb{H}^n.
\end{equation}

Let $\operatorname{O}(1,n)=\{g\in \operatorname{GL}(n+1)\,|\, g^{\intercal}Jg = J\}$, $J:=\operatorname{diag}(-1,I_n)$, denote the Lorentz group, i.e., the group of linear maps of $\R^{1,n}$ preserving $\langle\cdot,\cdot\rangle_{1,n}$, and let $G=\operatorname{SO}^+(1,n)$ be its identity component, which consists of the elements of $\operatorname{O}(1,n)$ with determinant $1$ that map the upper sheet $\mathbb{H}^n$ to itself. The group $\operatorname{SO}^+(1,n)$ thus acts on $\mathbb{H}^n$ by matrix multiplication, $\Phi_g(x)=gx$, and one easily verifies that this action is transitive. Writing $\bm{e}_0,\bm{e}_1,\ldots,\bm{e}_{n}$ for the standard basis of $\mathbb{R}^{1,n}$, we fix the base point $o=\bm{e}_0$. The isotropy group of $o$ is
\begin{equation}
H = \left\{
    \begin{pmatrix}
        1 & \bm{0}^{\intercal}\\
        \bm{0} & h
    \end{pmatrix}\,\Bigg|\, h\in \operatorname{SO}(n)\right\}\cong \operatorname{SO}(n),
\end{equation}
so that $\mathbb{H}^n \cong \operatorname{SO}^+(1,n)/\operatorname{SO}(n)$.
The Lie algebra of $\operatorname{SO}^+(1,n)$ is $\mathfrak{so}(1,n)=\{X\in \R^{(n+1)\times(n+1)}\,|\, X^{\intercal}J+JX=0\}$; writing $X$ in block form one sees that $\mathfrak{so}(1,n)=\mathfrak{h}\oplus\mathfrak{m}$ with
\begin{equation}
    \mathfrak{h}=\left\{
    \begin{pmatrix}
        0 & \bm{0}^{\intercal}\\
        \bm{0} & A
    \end{pmatrix}\,\Bigg|\, A\in\mathfrak{so}(n)\right\}\cong \mathfrak{so}(n)
\end{equation}
and
\begin{equation}\label{eq:hyperbolicmatrep}
\mathfrak{m} = \left\{ X_{\bm{v}} :=
    \begin{pmatrix}
        0 & \bm{v}^{\intercal}\\
        \bm{v} & 0_{n\times n}
    \end{pmatrix}\,\Bigg|\, \bm{v}\in\mathbb{R}^n\right\}.
\end{equation}
In particular $\mathfrak{m}\cong\mathbb{R}^n$ via $X_{\bm{v}}\mapsto \bm{v}$ and, as in the case of the sphere, the isomorphism $\mathfrak{m} \cong T_o \mathbb{H}^n$ of~\eqref{eq:link_m_ToM} is given explicitly by
\begin{equation}\label{eq:hyperbolic_ToH}
    X_{\bm{v}}^*(o) = \frac{\bm{d}}{\bm{d}t}\Big|_{t=0}\exp(tX_{\bm{v}})\,\bm{e}_0 = X_{\bm{v}}\bm{e}_0 = \begin{pmatrix}  0 \\ \bm{v}\end{pmatrix}\in T_o \mathbb{H}^n = \{x\in \R^{1,n}\,|\,x_0=0\}.
\end{equation}
The proof of the following proposition is entirely analogous to the proof of Proposition~\ref{prop:sphere_brackets}: for $A\in\mathfrak{so}(n)$ and $\bm{u},\bm{v}\in\R^n$ one computes $[\operatorname{diag}(0,A),X_{\bm{v}}]=X_{A\bm{v}}$ and $[X_{\bm{u}},X_{\bm{v}}]=\operatorname{diag}(0,\bm{u}\bm{v}^{\intercal}-\bm{v}\bm{u}^{\intercal})$.
\begin{proposition}\label{prop:hyperbolic_brackets}
    With $\mathfrak{h}$ and $\mathfrak{m}$ as above, the following inclusions hold:
    \begin{equation}
        [\mathfrak{h},\mathfrak{h}]\subseteq\mathfrak{h}, \quad [\mathfrak{h},\mathfrak{m}]\subseteq \mathfrak{m}, \quad [\mathfrak{m},\mathfrak{m}]\subseteq\mathfrak{h}.
    \end{equation}
\end{proposition}
We obtain the same conclusions as in the case of the sphere: as $H\cong \operatorname{SO}(n)$ is connected and $[\mathfrak{h},\mathfrak{m}]\subseteq\mathfrak{m}$, Lemma~\ref{lem:connected_H_reductive} shows that $\mathbb{H}^n\cong \operatorname{SO}^+(1,n)/\operatorname{SO}(n)$ is a reductive homogeneous space with reductive decomposition $\mathfrak{so}(1,n)=\mathfrak{h}\oplus\mathfrak{m}$ (alternatively, Lemma~\ref{prop:properaction_redhomospace} applies as $H$ is compact). Moreover, since $[\mathfrak{m},\mathfrak{m}]\subseteq\mathfrak{h}$, Lemma~\ref{lemma: LTS} shows that \emph{every} $\operatorname{Ad}_H$-invariant inner product on $\mathfrak{m}$ renders $\mathbb{H}^n$ a naturally reductive homogeneous space; in fact, $\mathfrak{h}$ and $\mathfrak{m}$ are the $+1$ and $-1$ eigenspaces of the involutive automorphism $X\mapsto JXJ$ of $\mathfrak{so}(1,n)$, so $\mathbb{H}^n$ is a symmetric space. Furthermore, for $h=\operatorname{diag}(1,A)\in H$ one computes $\operatorname{Ad}_hX_{\bm{v}} = hX_{\bm{v}}h^{-1} = X_{A\bm{v}}$, so that under the identification $\mathfrak{m}\cong\R^n$ the isotropy representation is again the standard representation of $\operatorname{SO}(n)$ on $\R^n$, which is irreducible by Proposition~\ref{prop:irrepsphere}. By the discussion following Proposition~\ref{prop:irrepsphere}, the $\operatorname{Ad}_H$-invariant inner product on $\mathfrak{m}$ is therefore unique up to a positive multiplicative constant, and as for the sphere we take the normalised Frobenius inner product
\begin{equation}\label{eq:hyperbolic_B}
    B(X,Y) := \tfrac12\tr(X^{\intercal}Y),\quad X,Y\in\mathfrak{m},\qquad\text{for which}\qquad B(X_{\bm{u}},X_{\bm{v}}) = \bm{u}\cdot\bm{v},\quad \bm{u},\bm{v}\in\R^n.
\end{equation}
Under the identification~\eqref{eq:hyperbolic_ToH}, $B$ thus coincides with the Lorentzian inner product restricted to $T_o\mathbb{H}^n$. As $\operatorname{SO}^+(1,n)$ preserves $\langle\cdot,\cdot\rangle_{1,n}$, the Riemannian metric on $\mathbb{H}^n$ induced by the Lorentzian inner product is $G$-invariant, and since it agrees with $\bm{g}_B$ at $o$, it coincides with $\bm{g}_B$ everywhere (see~\eqref{eq:gB}): $\bm{g}_B$ is the standard hyperbolic metric, with geodesic distance~\eqref{eq:distancehyperbolic}. Note that, in contrast to the sphere, for $n\geq 2$ the inner product $B$ is \emph{not} the restriction of an $\operatorname{Ad}_G$-invariant inner product on $\mathfrak{g}$: the non-compact semisimple group $\operatorname{SO}^+(1,n)$ does not admit a bi-invariant metric. Nevertheless, both conditions of Proposition~\ref{prop:unimodularredhombm} are satisfied: $\operatorname{SO}^+(1,n)$ is unimodular (it is abelian for $n=1$ and semisimple for $n\geq 2$), and $(\mathbb{H}^n,B)$ is naturally reductive.

\paragraph{Constructing the Riemannian Brownian motion and simulations.}
The construction of the Riemannian Brownian motion is now identical to the one for the sphere. The matrices $e_i':=X_{\bm{e}_i}$, $i=1,\ldots,n$, form a $B$-orthonormal basis of $\mathfrak{m}$, so taking $n$ independent standard real-valued Brownian motions $\bm{W}= (W_1,\hdots, W_n)$, the $\mathfrak{m}$-valued Brownian motion~\eqref{eq:mbrownianmotion} is given by
\begin{equation}
    W_\mathfrak{m}(t)= \sum_{i=1}^n e_i'W_i(t) = \begin{pmatrix}
        0 & \bm{W}(t)^{\intercal} \\
        \bm{W}(t) & 0_{n\times n}
    \end{pmatrix}, \quad t\geq 0.
\end{equation}
By Proposition~\ref{prop:unimodularredhombm}, if $(g(t))_{t\geq 0}$ is the $\operatorname{SO}^+(1,n)$-valued solution to~\eqref{eq:stratsde} driven by $W_{\mathfrak{m}}$, then $X(t)=\pi(g(t))=g(t)\bm{e}_{0}$, the first column of $g(t)$, is a Riemannian Brownian motion on $\mathbb{H}^n$ starting at $o=\bm{e}_0$. For the simulation we use the Wong--Zakai approximation~\eqref{eq:prewongzakai}, i.e., $X_h(kh)=\pi(g_1\cdots g_k)=g_1\cdots g_k\,\bm{e}_0$ with $g_k=\exp(X_{\Delta\bm{W}^{k,h}})$, $\Delta\bm{W}^{k,h}=\bm{W}(kh)-\bm{W}((k-1)h)$. The matrix exponential of an element of $\mathfrak{m}$ is again available in closed form: as $X_{\bm{v}}^3=\|\bm{v}\|^2X_{\bm{v}}$, one has, for $\theta=\|\bm{v}\|>0$,
\begin{equation}\label{eq:hyperbolic_exp}
    \exp(X_{\bm{v}}) = I_{n+1} + \frac{\sinh \theta}{\theta}X_{\bm{v}} + \frac{\cosh\theta-1}{\theta^2}X_{\bm{v}}^2.
\end{equation}
In particular,
\begin{equation}
    \exp(X_{\bm{v}})\,\bm{e}_{0} = \begin{pmatrix} \cosh\theta\\ \frac{\sinh\theta}{\theta}\,\bm{v}\end{pmatrix}\in \mathbb{H}^n,
\end{equation}
which is the geodesic through $o$ with initial velocity $(0,\bm{v})^{\intercal}$, traversed at unit speed if $\|\bm{v}\|=1$, as it must be for a naturally reductive homogeneous space.

Figure \ref{fig:hyperbolicbmtraj} shows a realisation of the Wong--Zakai approximation of the Riemannian Brownian motion on $\mathbb{H}^2$, represented in the Poincar\'e disk model, i.e., $\mathbb{H}^2$ is identified with the open unit disk $\mathbb{D}^2\subseteq \R^2$ via the projection $\mathbb{H}^2 \ni (x_0,x_1,x_2) \mapsto \big(\tfrac{x_1}{1+x_0},\tfrac{x_2}{1+x_0}\big)$. In contrast with the case of the sphere in Figure \ref{fig:spherebmtraj}, the Riemannian Brownian motion on $\mathbb{H}^2$ is transient; it converges to a random point on the boundary of the disk (and so does the geodesic random walk approximating it), see \cite[Theorem 2.2]{Ichihara:1982}, \cite[Theorems 2.7, 2.8]{ArmstrongKing:2023}. This transient behaviour is also visible in the distance process $t\mapsto d(o,X_h(t))$ shown in Figure \ref{fig:hyperbolicbmdist}, which grows without bound over time.

\begin{figure}[H]
    \centering
    \includegraphics[width=0.9\textwidth]{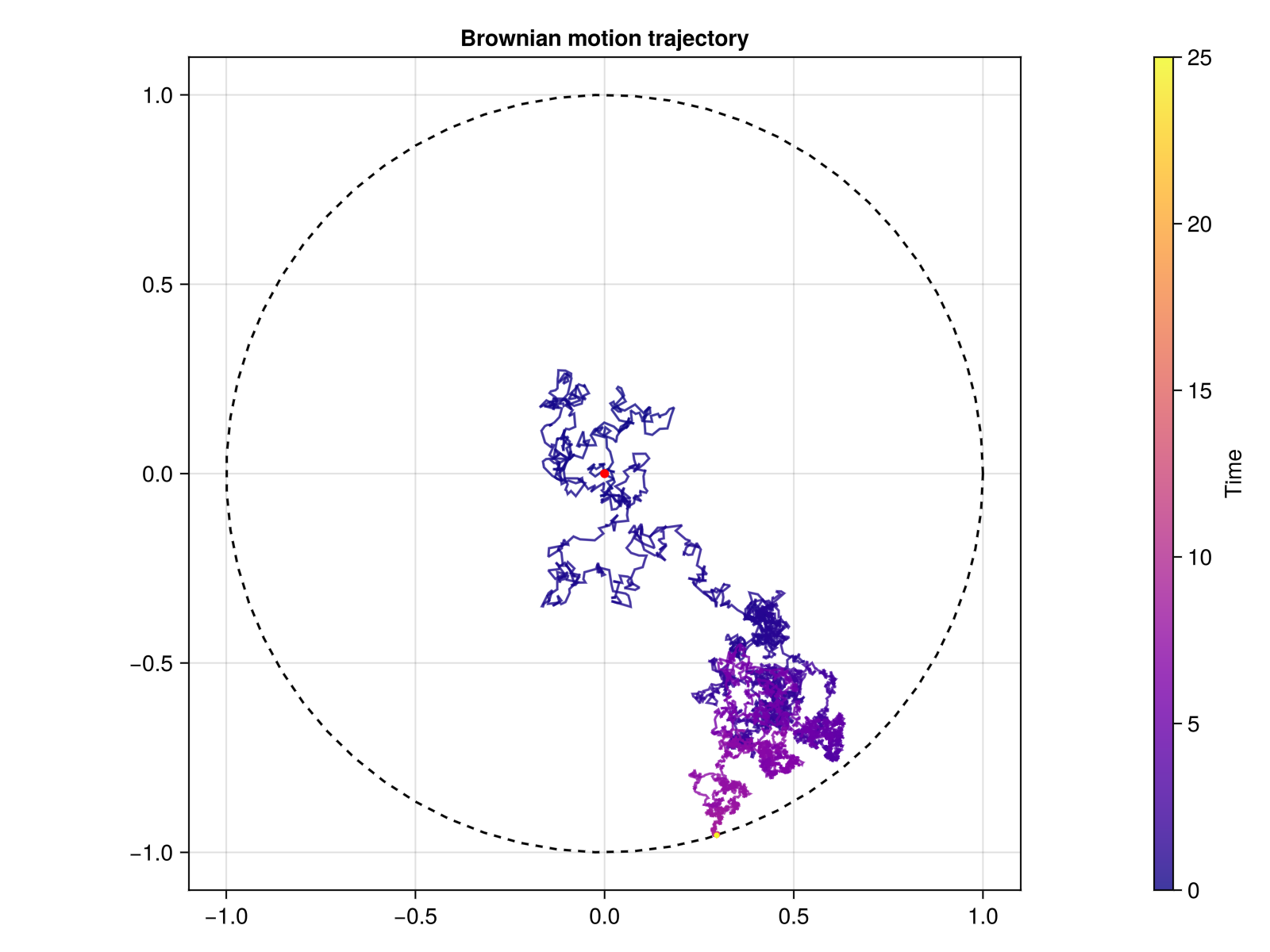}
    \caption{A realisation of the Wong--Zakai approximation of the Riemannian Brownian motion on $\mathbb{H}^2$, started at $o=\bm{e}_0$ and represented in the Poincar\'e disk. This simulation is obtained with step-size $h=0.001$.}
    \label{fig:hyperbolicbmtraj}
\end{figure}

\begin{figure}[H]
    \centering
    \includegraphics[width=0.9\textwidth]{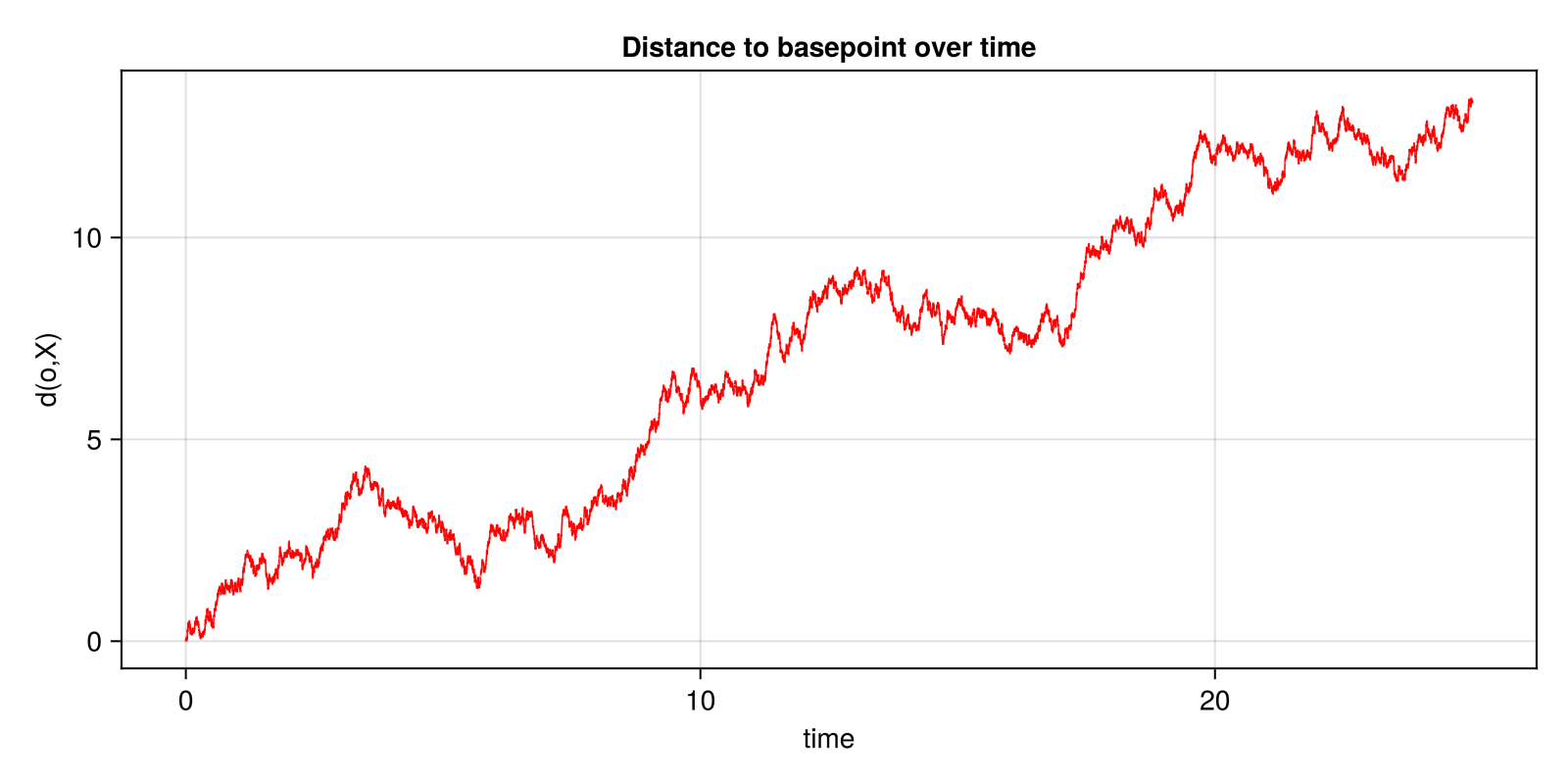}
    \caption{The geodesic distance process $t\mapsto d(o,X_h(t))$ on $\mathbb{H}^2$ associated with the realisation of Figure~\ref{fig:hyperbolicbmtraj}, with the distance function given by \eqref{eq:distancehyperbolic}. This simulation is obtained with step-size $h=0.001$.}
    \label{fig:hyperbolicbmdist}
\end{figure}

\subsection{Symmetric positive-definite matrices}\label{sec:spd}
\paragraph{Symmetric positive-definite matrices as a reductive homogeneous space.}
The space of symmetric positive-definite matrices
\begin{equation}
\operatorname{SPD}(n) = \{ \Sigma\in \operatorname{GL}(n)\,|\, \Sigma=\Sigma^{\intercal} \text{ and } \Sigma\succ 0\}
\end{equation}
arises naturally in many areas: covariance matrices in statistics and machine learning, diffusion tensors in physics and medical imaging, metric tensors in differential geometry, and Hessians in convex optimisation. Note that $\operatorname{SPD}(n)$ is an open subset of the vector space of symmetric matrices $\operatorname{Sym}(n)$; the inclusion $\operatorname{SPD}(n)\hookrightarrow \operatorname{Sym}(n)$ provides a global chart (in particular, $\operatorname{SPD}(n)$ is a smooth manifold of dimension $\frac12 n(n+1)$) that allows us to identify $T_{\Sigma} \operatorname{SPD}(n)$ with $\operatorname{Sym}(n)$ for all $\Sigma\in \operatorname{SPD}(n)$. The group $G=\operatorname{GL}(n)$ acts on $\operatorname{SPD}(n)$ by congruence: for $A\in \operatorname{GL}(n)$ we define $\Phi_{A} \colon \operatorname{SPD}(n)\rightarrow \operatorname{SPD}(n)$ by
\begin{equation}\label{eq:spd_action}
    \Phi_{A}(\Sigma) = A\Sigma A^{\intercal}, \quad \Sigma\in \operatorname{SPD}(n).
\end{equation}
This is a left action, and it is transitive as $\Phi_{\Sigma^{1/2}}(I_n)=\Sigma$ for all $\Sigma\in \operatorname{SPD}(n)$, where $\Sigma^{1/2}\in \operatorname{SPD}(n)$ denotes the symmetric positive-definite square root of $\Sigma$. Under the identification $T_\Sigma\operatorname{SPD}(n)=\operatorname{Sym}(n)$ the differential of $\Phi_A$ is given by $\textnormal{d}(\Phi_A)_\Sigma V = AVA^{\intercal}$, $V\in \operatorname{Sym}(n)$. Fixing the base point $o=I_n$, the identity matrix, we see that the isotropy group associated with this action is $H=\{A\in \operatorname{GL}(n)\,|\,AA^{\intercal}=I_n\}=\operatorname{O}(n)$, i.e., $\operatorname{SPD}(n) \cong \operatorname{GL}(n)/\operatorname{O}(n)$. \par
The Lie algebra $\mathfrak{gl}(n)=\R^{n\times n}$ admits the decomposition
\begin{equation}\label{eq:spd_decomp}
    \mathfrak{gl}(n) = \mathfrak{so}(n)\oplus \operatorname{Sym}(n),\qquad X = \tfrac12(X-X^{\intercal}) + \tfrac12(X+X^{\intercal}),
\end{equation}
into skew-symmetric and symmetric matrices, where $\mathfrak{h}=\mathfrak{so}(n)$ is the Lie algebra of $H=\operatorname{O}(n)$; we set $\mathfrak{m}:=\operatorname{Sym}(n)$. Since $\operatorname{Ad}_h X = hXh^{-1} = hXh^{\intercal}$ is symmetric for all $h\in \operatorname{O}(n)$ and $X\in \operatorname{Sym}(n)$, the subspace $\mathfrak{m}=\operatorname{Sym}(n)$ is $\operatorname{Ad}_H$-invariant, whence $\operatorname{SPD}(n)\cong \operatorname{GL}(n)/\operatorname{O}(n)$ is a reductive homogeneous space in the sense of Definition~\ref{def:rhs} with reductive decomposition~\eqref{eq:spd_decomp}. (Here we verified the $\operatorname{Ad}_H$-invariance of $\mathfrak{m}$ directly: Lemma~\ref{lem:connected_H_reductive} does not apply as $H=\operatorname{O}(n)$ is not connected, and Lemma~\ref{prop:properaction_redhomospace}, although applicable as $H$ is compact, does not specify the reductive complement.)

\begin{proposition}\label{prop:spd_brackets}
    The following inclusions hold:
    \begin{equation}
        [\mathfrak{so}(n),\mathfrak{so}(n)]\subseteq\mathfrak{so}(n), \quad [\mathfrak{so}(n), \operatorname{Sym}(n)]\subseteq \operatorname{Sym}(n), \quad [\operatorname{Sym}(n),\operatorname{Sym}(n)]\subseteq \mathfrak{so}(n).
    \end{equation}
\end{proposition}

\begin{proof}
    Since $\mathfrak{so}(n)$ is a Lie subalgebra of $\mathfrak{gl}(n)$, the first inclusion is immediate. For the second inclusion, let $Y\in \mathfrak{so}(n)$ and $X\in \operatorname{Sym}(n)$ and observe that the transpose of the commutator is given by $[Y,X]^{\intercal} = (YX-XY)^{\intercal} = X^{\intercal}Y^{\intercal} - Y^{\intercal}X^{\intercal} = -XY + YX = [Y,X]$. Hence $[Y,X]\in \operatorname{Sym}(n)$, and the second inclusion follows. For the third inclusion, let $X, Y\in \operatorname{Sym}(n)$. The transpose of the commutator is $[X,Y]^{\intercal} =(XY - YX)^{\intercal} = Y^{\intercal}X^{\intercal} - X^{\intercal}Y^{\intercal} = YX - XY = -[X,Y]$, and thus the third inclusion follows.
\end{proof}

Since $[\mathfrak{m},\mathfrak{m}]\subseteq\mathfrak{h}$ by Proposition~\ref{prop:spd_brackets}, Lemma~\ref{lemma: LTS} shows that \emph{every} $\operatorname{Ad}_{\operatorname{O}(n)}$-invariant inner product on $\operatorname{Sym}(n)$ renders $\operatorname{SPD}(n)$ a naturally reductive homogeneous space; as in the previous two examples, $\mathfrak{h}$ and $\mathfrak{m}$ are the $+1$ and $-1$ eigenspaces of the involutive automorphism $X\mapsto -X^{\intercal}$ of $\mathfrak{gl}(n)$, so $\operatorname{SPD}(n)$ is a symmetric space. Moreover, the group $\operatorname{GL}(n)$ is unimodular: for $A\in \operatorname{GL}(n)$ the linear map $\operatorname{Ad}_A\colon X\mapsto AXA^{-1}$ on $\R^{n\times n}$ is the composition of $X\mapsto AX$ and $X\mapsto XA^{-1}$, which have determinants $\det(A)^n$ and $\det(A^{-1})^n$, so that $|\det \operatorname{Ad}_A|=1$ and Proposition~\ref{prop:unimodular_Ad} applies. Hence \emph{both} conditions of Proposition~\ref{prop:unimodularredhombm} are satisfied for every $\operatorname{Ad}_{\operatorname{O}(n)}$-invariant inner product $B$ on $\operatorname{Sym}(n)$.

Note that, unlike for the sphere and the hyperbolic space, the identification~\eqref{eq:link_m_ToM} of $\mathfrak{m}=\operatorname{Sym}(n)$ with $T_o\operatorname{SPD}(n)=\operatorname{Sym}(n)$ is \emph{not} the identity map: for $X\in \operatorname{Sym}(n)$ one has
\begin{equation}\label{eq:spd_ToM}
    X^*(o) = \frac{\bm{d}}{\bm{d}t}\Big|_{t=0} \exp(tX)\,I_n\,\exp(tX)^{\intercal} = X+X^{\intercal} = 2X.
\end{equation}
This factor $2$ has to be taken into account when translating between inner products on $\mathfrak{m}$ and Riemannian metrics on $\operatorname{SPD}(n)$, see~\eqref{eq:spd_gB} below. Note also that $\pi(\exp(tX)) = \exp(tX)\exp(tX)^{\intercal}=\exp(2tX)$ for $X\in \operatorname{Sym}(n)$, in accordance with the fact that the geodesics through $I_n$ of the metrics considered below are of the form $t\mapsto \exp(tV)$, $V\in \operatorname{Sym}(n)$ (see \cite[Table 5]{ThanwerdasPennec:2023}).

\paragraph{$\operatorname{GL}(n)$-invariant metrics on SPD$(n)$.}
The space $\operatorname{SPD}(n)$ allows for multiple $\operatorname{O}(n)$-invariant metrics; an excellent overview of the literature on this topic can be found in~\cite{ThanwerdasPennec:2023}. The $\operatorname{GL}(n)$-invariant ones among these correspond to the $\operatorname{Ad}_{\operatorname{O}(n)}$-invariant inner products on $\mathfrak{m}=\operatorname{Sym}(n)$, i.e., to the inner products $\langle\cdot,\cdot\rangle$ on $\operatorname{Sym}(n)$ satisfying $\langle hXh^{\intercal},hYh^{\intercal}\rangle = \langle X,Y\rangle$ for all $h\in \operatorname{O}(n)$ (recall that $\operatorname{Ad}_hX=hXh^{\intercal}$ for $h\in\operatorname{O}(n)$). These are characterised in~\cite{ThanwerdasPennec:2023}:

\begin{theorem}[{\cite[Theorem 2.1]{ThanwerdasPennec:2023}}]\label{thm:ThP}
Let $\langle\cdot,\cdot\rangle\colon \operatorname{Sym}(n)\times \operatorname{Sym}(n)\rightarrow \R$ be an inner product. Then $\langle \cdot,\cdot \rangle$ is $\operatorname{Ad}_{\operatorname{O}(n)}$-invariant if and only if there exist $\alpha\in (0,\infty)$ and $\beta \in (-\nicefrac{\alpha}{n},\infty)$ such that
\begin{equation}\label{eq:innerproductalphabeta}
\langle X,Y \rangle = \langle X,Y \rangle_{\alpha,\beta} := \alpha \tr(XY) + \beta \tr(X)\tr(Y),\quad X,Y\in \operatorname{Sym}(n).
\end{equation}
Moreover, the linear map $F_{p,\sqrt{\alpha}}\colon \operatorname{Sym}(n)\rightarrow\operatorname{Sym}(n)$, $F_{p,\sqrt{\alpha}}(X) = \sqrt{\alpha} X + \frac{p-\sqrt{\alpha}}{n}\tr(X)I_n$, with $p=\sqrt{\alpha +n\beta}$, is an isometry from $(\operatorname{Sym}(n),\langle\cdot,\cdot\rangle_{\alpha,\beta})$ onto $(\operatorname{Sym}(n),\langle\cdot,\cdot\rangle_{1,0})$, i.e., $\langle X,Y\rangle_{\alpha,\beta} = \tr\big(F_{p,\sqrt{\alpha}}(X)F_{p,\sqrt{\alpha}}(Y)\big)$ for all $X,Y\in \operatorname{Sym}(n)$.
\end{theorem}

While we refer to~\cite{ThanwerdasPennec:2023} for a detailed proof, we do point out that this characterisation stems from the fact that, for $n\geq 2$, the isotropy representation of $\operatorname{O}(n)$ on $\operatorname{Sym}(n)$ decomposes into two non-isomorphic irreducible components,
\begin{equation}
    \operatorname{Sym}(n) = \mathbb{R}\cdot I_n \oplus \operatorname{Sym}_0(n),
\end{equation}
where $\operatorname{Sym}_0(n)$ denotes the space of trace-free symmetric matrices and $\mathbb{R}\cdot I_n$ denotes the one-dimensional space of real multiples of the identity matrix. By Schur's lemma, any $\operatorname{Ad}_{\operatorname{O}(n)}$-invariant inner product is therefore of the form $a\tr(X_0Y_0)+b\tr(X)\tr(Y)$ with $a,b>0$, where $X_0 = X-\frac{1}{n}\tr(X)I_n$ denotes the trace-free part of $X\in \operatorname{Sym}(n)$; expanding $\tr(X_0Y_0) = \tr(XY)-\frac1n\tr(X)\tr(Y)$ yields~\eqref{eq:innerproductalphabeta} with $\alpha = a$ and $\beta = b-\frac{a}{n}$. The trivial component $\mathbb{R}\cdot I_n$, and hence the parameter $\beta$ in~\eqref{eq:innerproductalphabeta}, appears because $\mathfrak{gl}(n)$ is reductive but not semisimple: its center $\mathbb{R}\cdot I_n$ is contained in $\mathfrak{m}$. This distinguishes $\operatorname{SPD}(n)$ from the spheres and the hyperbolic spaces of the previous sections.

We now determine the $\operatorname{GL}(n)$-invariant Riemannian metric on $\operatorname{SPD}(n)$ induced, in the sense of~\eqref{eq:gB}, by $B=\langle\cdot,\cdot\rangle_{\alpha,\beta}$ for $\alpha>0$ and $\beta > -\frac{\alpha}{n}$. By~\eqref{eq:gB}, $\bm{g}_{\langle\cdot,\cdot\rangle_{\alpha,\beta}}(o;X^*(o),Y^*(o))=\langle X, Y \rangle_{\alpha,\beta}$ for $X,Y\in\mathfrak{m}$, so that by~\eqref{eq:spd_ToM} one has $\bm{g}_{\langle\cdot,\cdot\rangle_{\alpha,\beta}}(I_n;X,Y) = \frac14\langle X,Y\rangle_{\alpha,\beta}$ for $X,Y\in T_{I_n}\operatorname{SPD}(n)=\operatorname{Sym}(n)$. Transporting this inner product to $T_\Sigma\operatorname{SPD}(n)$, $\Sigma\in \operatorname{SPD}(n)$, by means of $\Phi_{\Sigma^{-1/2}}$ (recall that $\textnormal{d}(\Phi_{\Sigma^{-1/2}})_{\Sigma}X=\Sigma^{-1/2}X\Sigma^{-1/2}$) we obtain
\begin{equation}\label{eq:spd_gB}
        \bm{g}_{\langle\cdot,\cdot\rangle_{\alpha,\beta}}(\Sigma; X,Y) = \tfrac14\Big(\alpha \tr (\Sigma^{-1}X \Sigma^{-1}Y) + \beta\tr(\Sigma^{-1} X) \tr(\Sigma^{-1}Y)\Big),\quad X,Y\in \operatorname{Sym}(n)=T_\Sigma \operatorname{SPD}(n).
\end{equation}

\begin{remark}
A word of warning: in~\cite{ThanwerdasPennec:2023}, the metric $\bm{g}_{\alpha,\beta}$ on $\operatorname{SPD}(n)$ is defined (for $\alpha\in (0,\infty)$, $\beta\in (-\nicefrac{\alpha}{n},\infty)$) by
\begin{equation}\label{eq:generalmetricspd}
    \bm{g}_{\alpha,\beta}(\Sigma;X,Y) := \alpha \tr (\Sigma^{-1}X \Sigma^{-1}Y) + \beta\tr(\Sigma^{-1} X) \tr(\Sigma^{-1}Y),\quad X,Y\in \operatorname{Sym}(n).
\end{equation}
Thus~\eqref{eq:spd_gB} states that
\begin{equation}\label{eq:spd_gB_vs_ThP}
    \bm{g}_{\langle\cdot,\cdot\rangle_{\alpha,\beta}} = \bm{g}_{\alpha/4,\beta/4}.
\end{equation}
(The factor $4$ is the same one that appears in the left-invariant metric on $\operatorname{GL}^+(n)$ in \cite[Proposition 3.1]{ThanwerdasPennec:2023}.)
As we rely on~\cite[Section 3.3]{ThanwerdasPennec:2023} for various geometric properties of the affine-invariant metrics (distance function, curvature, geodesics, etc.), we shall also use the notation $\bm{g}_{\alpha,\beta}$ in what follows. Note that~\cite{ThanwerdasPennec:2023} refers to this class of metrics as the \emph{affine-invariant} metrics, and provides a literature review on this class, observing in particular that $\bm{g}_{\nicefrac{1}{2},0}$ is often referred to as the \emph{Fisher--Rao metric}.
\end{remark}

By \cite[Table 5]{ThanwerdasPennec:2023}, the geodesic distance associated with $\bm{g}_{\alpha,\beta}$ between $\Sigma,\Theta \in \operatorname{SPD}(n)$ is given by
\begin{equation}\label{eq:distancespd}
    d_{\alpha,\beta}(\Sigma,\Theta) = \sqrt{\alpha \tr\big((\log (\Sigma^{-1/2}\Theta \Sigma^{-1/2}))^2\big) + \beta \big(\log\det(\Sigma^{-1}\Theta)\big)^2 },
\end{equation}
or, expressed in terms of the eigenvalues $\lambda_1,\hdots,\lambda_n$ of $\Sigma^{-\nicefrac{1}{2}}\Theta \Sigma^{-\nicefrac{1}{2}}$,
\begin{equation}\label{eq:distancespd_eig}
    d_{\alpha,\beta}(\Sigma,\Theta) = \sqrt{\alpha \sum_{i=1}^n (\log\lambda_i)^2 + \beta \Big(\sum_{i=1}^n \log\lambda_i \Big)^2}.
\end{equation}
Note that $d_{\alpha,\beta}(I_n,\exp(X))^2 = \langle X,X\rangle_{\alpha,\beta}$ for $X\in \operatorname{Sym}(n)$, in accordance with the fact that $t\mapsto \exp(tX)$ is the $\bm{g}_{\alpha,\beta}$-geodesic through $I_n$ with initial velocity $X$ and that the Riemannian logarithm of $\bm{g}_{\alpha,\beta}$ at $I_n$ is the matrix logarithm (see \cite[Table 5]{ThanwerdasPennec:2023}).

\paragraph{Constructing the Riemannian Brownian motion.}
In order to construct a Riemannian Brownian motion on $(\operatorname{SPD}(n),\bm{g}_{\alpha,\beta})$ as in Section~\ref{ssec:RM_natredHS}, we need, by~\eqref{eq:spd_gB_vs_ThP}, an orthonormal basis of $(\mathfrak{m},B)=(\operatorname{Sym}(n),\langle\cdot,\cdot\rangle_{4\alpha,4\beta})$. We first construct an orthonormal basis of $(\operatorname{Sym}(n),\langle\cdot,\cdot\rangle_{\alpha,\beta})$ for arbitrary $\alpha\in (0,\infty)$ and $\beta\in (-\nicefrac{\alpha}{n},\infty)$. Let $\bm{e}_1,\ldots,\bm{e}_n$ denote the standard basis of $\R^n$ and let $E_{i,j}\in \operatorname{Sym}(n)$, $1\leq i\leq j\leq n$, be given by
\begin{equation}\label{eq:spd_Eij}
    E_{i,j} = \begin{cases}
        \bm{e}_i\bm{e}_i^{\intercal},& i=j;\\
        \frac{1}{\sqrt{2}} ( \bm{e}_i \bm{e}_j^{\intercal} + \bm{e}_j \bm{e}_i^{\intercal}), & i< j;
    \end{cases}
\end{equation}
this is an orthonormal basis of $\operatorname{Sym}(n)$ with respect to the Frobenius inner product $\langle X,Y\rangle_{1,0}=\tr(XY)$. Next, let $F_{p,\sqrt{\alpha}}\colon \operatorname{Sym}(n)\rightarrow \operatorname{Sym}(n)$ be the isometry from Theorem~\ref{thm:ThP}. Its inverse is given by $F_{p,\sqrt{\alpha}}^{-1}(Y)=\frac{1}{\sqrt{\alpha}}Y-\frac{p-\sqrt{\alpha}}{np\sqrt{\alpha}}\tr(Y)I_n$, so that (recalling that $p= \sqrt{\alpha + n\beta}$) we obtain
\begin{equation}\label{eq:generaldoubleindexbasis}
    E_{i,j}^{(\alpha,\beta)} := F_{p,\sqrt{\alpha}}^{-1}(E_{i,j}) = \frac{1}{\sqrt{\alpha}} E_{i,j} - \frac{\delta_{i,j}}{n\sqrt{\alpha}}\Big( 1 - \sqrt{\frac{\alpha}{\alpha+n\beta}}\Big) I_n, \quad 1\leq i\leq j\leq n,
\end{equation}
is an orthonormal basis\footnote{Alternatively, it can be verified by hand that this is indeed an orthonormal basis. Note that there appears to be a sign error in the definition of $E_{ij}^{\beta}$ in~\cite[Table 5]{ThanwerdasPennec:2023}: in their notation one should take $E_{ij}^{\beta} = E_{ij} - \frac{p-1}{np}\delta_{ij}I$ instead of $E_{ij} - \frac{1-p}{np}\delta_{ij}I$. Indeed, for $\beta>0$ the trace part of $E_{ii}$ has to be \emph{shrunk}, i.e., a positive multiple of the identity has to be subtracted, in order to compensate for the positive trace term.} of $(\operatorname{Sym}(n),\langle\cdot,\cdot\rangle_{\alpha,\beta})$. Note that $E_{i,j}^{(\alpha,\beta)} = \frac{1}{\sqrt{\alpha}} E_{i,j}$ whenever $i\neq j$, and that $E_{i,j}^{(\alpha,0)} = \frac{1}{\sqrt{\alpha}} E_{i,j}$ for all $1\leq i\leq j\leq n$. We conclude that $(\frac{1}{2}E_{i,j}^{(\alpha,\beta)})_{1\leq i \leq j\leq n}$ is an orthonormal basis of $(\mathfrak{m},B) = (\operatorname{Sym}(n),\langle\cdot,\cdot\rangle_{4\alpha,4\beta})$, the inner product on $\mathfrak{m}$ inducing the affine-invariant metric $\bm{g}_{\alpha,\beta}$.

We can now make the construction of Section~\ref{ssec:RM_natredHS} explicit. Let $(W_{i,j})_{1\leq i\leq j\leq n}$ be $\frac12 n(n+1)$ independent standard real-valued Brownian motions. The $\mathfrak{m}$-valued Brownian motion~\eqref{eq:mbrownianmotion} is given by
\begin{equation}\label{eq:spd_Wm}
    W_{\mathfrak{m}}(t) = \frac{1}{2} \sum_{1\leq i\leq j\leq n} E_{i,j}^{(\alpha,\beta)} W_{i,j}(t) = \frac{1}{2\sqrt{\alpha}}\Big( \hat{W}(t) - \frac{1}{n}\Big(1-\sqrt{\frac{\alpha}{\alpha+n\beta}}\Big)\tr\big(\hat{W}(t)\big)\,I_n\Big),\quad t\geq 0,
\end{equation}
where $\hat{W}(t) := \sum_{1\leq i\leq j\leq n} E_{i,j}W_{i,j}(t)$ is the symmetric matrix with diagonal entries $\hat{W}_{i,i}(t)=W_{i,i}(t)$ and off-diagonal entries $\hat{W}_{i,j}(t)=\hat{W}_{j,i}(t)=W_{i,j}(t)/\sqrt{2}$, $i<j$; in other words, $\hat{W}$ is the standard Brownian motion in the Euclidean space $(\operatorname{Sym}(n),\langle\cdot,\cdot\rangle_{1,0})$, and $W_{\mathfrak{m}} = \frac12 F_{p,\sqrt{\alpha}}^{-1}(\hat{W})$ is the standard Brownian motion in $(\operatorname{Sym}(n),\langle\cdot,\cdot\rangle_{4\alpha,4\beta})$. As $(\textnormal{d}L_g)_e X = gX$ for the matrix group $\operatorname{GL}(n)$, the $\operatorname{GL}(n)$-valued process $g$ of~\eqref{eq:stratsde} is the solution to the matrix-valued Stratonovich SDE
\begin{equation}\label{eq:spd_sde}
    \bm{d}g(t) = g(t)\circ \bm{d}W_{\mathfrak{m}}(t),\quad t\geq 0;\qquad g(0)=I_n,
\end{equation}
and since both conditions of Proposition~\ref{prop:unimodularredhombm} are satisfied,
\begin{equation}\label{eq:spd_BM}
    X(t) = \pi(g(t)) = \Phi_{g(t)}(I_n) = g(t)g(t)^{\intercal},\quad t\geq 0,
\end{equation}
is a Riemannian Brownian motion on $(\operatorname{SPD}(n),\bm{g}_{\alpha,\beta})$ starting at $I_n$; note that $X(t)$ is symmetric positive definite by construction. By Section~\ref{ssec:RM_natredHS}, a Riemannian Brownian motion starting at an arbitrary $\Sigma_0\in \operatorname{SPD}(n)$ is given by $\pi(\Sigma_0^{1/2}g(t)) = \Sigma_0^{1/2}g(t)g(t)^{\intercal}\Sigma_0^{1/2}$, $t\geq 0$.

For the simulation we use the Wong--Zakai approximation~\eqref{eq:prewongzakai}: setting $\Delta W_{\mathfrak{m}}^{k,h} := W_{\mathfrak{m}}(kh)-W_{\mathfrak{m}}((k-1)h)$ and $g_k := \exp\big(\Delta W_{\mathfrak{m}}^{k,h}\big)$, $k\in \N$, one has
\begin{equation}\label{eq:spd_WZ}
    X_h(kh) = \pi(g_1\cdots g_k) = g_1\cdots g_k\, g_k^{\intercal}\cdots g_1^{\intercal} = \Phi_{g_1}\circ\cdots\circ\Phi_{g_k}(I_n),\quad k\in \N.
\end{equation}
Here $\exp$ is the matrix exponential of a symmetric matrix, which can be evaluated by diagonalisation: if $\Delta W^{k,h}_{\mathfrak{m}} = Q\operatorname{diag}(\mu_1,\ldots,\mu_n)Q^{\intercal}$ with $Q\in \operatorname{O}(n)$, then $g_k = Q\operatorname{diag}(e^{\mu_1},\ldots,e^{\mu_n})Q^{\intercal}$. In particular, each increment $g_k$ is itself symmetric positive definite, so that $X_h(kh)$ is symmetric positive definite for all $k\in\N$ without any projection step. Note that the $g_k$ do not commute, so that the order of the factors in~\eqref{eq:spd_WZ} matters; for $k=1$ one recovers $X_h(h)=\exp(2\Delta W_{\mathfrak{m}}^{1,h})$, a point on the geodesic through $I_n$ with initial velocity $2\Delta W_{\mathfrak{m}}^{1,h}$, in accordance with~\eqref{eq:spd_ToM}.

\paragraph{A characterisation without reference to $\operatorname{GL}(n)$.}
The Riemannian Brownian motion $X$ can also be characterised by an It\^o SDE on $\operatorname{Sym}(n)$ involving only matrix products, which is the form in which $\operatorname{SPD}(n)$-valued diffusions are usually specified in the literature (see e.g.~\cite{Bru:1991}). \par
To see this, we fix $\Sigma_0\in \operatorname{SPD}(n)$, $\alpha>0$ and $\beta >\frac{\alpha}{n}$, and let $g$ be the solution to~\eqref{eq:spd_sde} with initial value $g(0)=\Sigma_0^{1/2}$, so that $X=gg^{\intercal}$ is the Riemannian Brownian motion starting at $\Sigma_0$. Since $W_{\mathfrak{m}}$ is symmetric, the Stratonovich product rule gives $\bm{d}X = 2\,g\circ\bm{d}W_{\mathfrak{m}}\,g^{\intercal}$. Note that this is not closed in $X$, and replacing $g$ by $X^{1/2}$ in this Stratonovich equation is possible only at the cost of an additional drift term involving the derivative of $\Sigma\mapsto\Sigma^{1/2}$, which has no expression in terms of finitely many matrix products. We therefore pass to It\^o form. As the diffusion coefficient in~\eqref{eq:spd_sde} is linear in $g$ and $W_{\mathfrak{m}}(t)=\frac12\sum_{i\leq j}E_{i,j}^{(\alpha,\beta)}W_{i,j}(t)$ (see~\eqref{eq:spd_Wm}), the It\^o form of~\eqref{eq:spd_sde} reads $\bm{d}g = g\,\bm{d}W_{\mathfrak{m}} + \frac18\, g \sum_{1\leq i\leq j\leq n}(E_{i,j}^{(\alpha,\beta)})^2\,\bm{d}t$, and a direct computation from~\eqref{eq:generaldoubleindexbasis} yields
\begin{equation}\label{eq:Strat_to_Ito}
   \sum_{1\leq i \leq j \leq n} \big(E_{i,j}^{(\alpha,\beta)}\big)^2 = 4 m_{\alpha,\beta}\, I_n
   \qquad\text{with}\qquad
    m_{\alpha,\beta} := \frac{1}{4\alpha}\Big(\frac{n+1}{2}-\frac{\beta}{\alpha+n\beta}\Big).
\end{equation}
Hence $\bm{d}g = g\,\bm{d}W_{\mathfrak{m}} + \tfrac12 m_{\alpha,\beta}\,g\,\bm{d}t$, and It\^o's product rule yields $\bm{d}X = 2\,g\,\bm{d}W_{\mathfrak{m}}\,g^{\intercal} + 2m_{\alpha,\beta}\,X\,\bm{d}t$ (note that the extra `$+m_{\alpha,\beta}\,X\,\bm{d}t$' is the It\^o correction term). Set $R := X^{-1/2}g$ (so that $g=X^{1/2}R$) and $\widetilde{W}_{\mathfrak{m}} := \int_0^{\cdot} R\,\bm{d}W_{\mathfrak{m}}\,R^{\intercal}$. Then $X$ satisfies
\begin{equation}\label{eq:spd_ito_X1}
    \bm{d}X(t) = 2\,X^{1/2}(t)\,\bm{d}\widetilde{W}_{\mathfrak{m}}(t)\,X^{1/2}(t) + \frac{1}{2\alpha}\Big(\frac{n+1}{2}-\frac{\beta}{\alpha+n\beta}\Big)X(t)\,\bm{d}t,\qquad X(0)=\Sigma_0.
\end{equation}
The notation $\widetilde{W}_{\mathfrak{m}}$ is somewhat suggestive: indeed, we now show that $W_{\mathfrak{m}}$ and $\widetilde{W}_{\mathfrak{m}}$ are equal in law.
By the $\operatorname{Ad}_{\operatorname{O}(n)}$-invariance of $\langle\cdot,\cdot\rangle_{4\alpha,4\beta}$ and the fact that $R$ is a continuous adapted process taking values in $\operatorname{O}(n)$, the map $Z\mapsto R(t)ZR(t)^{\intercal}$ is an orthogonal transformation of $(\mathfrak{m},\langle\cdot,\cdot\rangle_{4\alpha,4\beta})$  for every $t\geq 0$.
Thus, in the orthonormal basis $(\frac12E_{i,j}^{(\alpha,\beta)})_{i\leq j}$ of $(\mathfrak{m},\langle\cdot,\cdot\rangle_{4\alpha,4\beta})$, one has $\widetilde{W}_{\mathfrak{m}}=\int_0^{\cdot}O(t)\,\bm{d}\bm{W}(t)$ with $\bm{W}=(W_{i,j})_{i\leq j}$ a standard Brownian motion in $\R^{n(n+1)/2}$ and $O$ an adapted, continuous process of orthogonal $\frac12n(n+1)\times\frac12n(n+1)$ matrices; such a process is a standard Brownian motion by L\'evy's characterisation. In particular, $\widetilde{W}_{\mathfrak{m}}$ and $W_{\mathfrak{m}}$ are equal in law, and it thus follows from~\eqref{eq:spd_ito_X1} that the Riemannian Brownian motion $X$ on $(\operatorname{SPD}(n),\bm{g}_{\alpha,\beta})$ is a \emph{(stochastically) weak} solution of the It\^o SDE
\begin{equation}\label{eq:spd_ito_X}
    \bm{d}X(t) = 2\,X^{1/2}(t)\,\bm{d}W_{\mathfrak{m}}(t)\,X^{1/2}(t) + \frac{1}{2\alpha}\Big(\frac{n+1}{2}-\frac{\beta}{\alpha+n\beta}\Big)X(t)\,\bm{d}t,\qquad X(0)=\Sigma_0.
\end{equation}
Note that having the It\^o form of~\eqref{eq:spd_sde} at our disposal, one could also approximate $g$ by an Euler--Maruyama scheme, whose iterates remain in $\operatorname{GL}(n)$ with probability one, and use this to approximate $X$.

The It\^o SDE~\eqref{eq:spd_ito_X} demonstrates that the Riemannian Brownian motion on $(\operatorname{SPD}(n),\bm{g}_{\alpha,\beta})$ can be viewed as an $\operatorname{SPD}(n)$-valued analogue of the geometric Brownian motion. In particular, it differs in character from the so-called Wishart process, studied in detail in~\cite{Bru:1991}: the Wishart process is the matrix analogue of a squared Bessel process, obtained from an \emph{additive} matrix Brownian motion $N$ as $N^{\intercal}N$. Accordingly, the eigenvalues of a Wishart process behave like squared Bessel processes, see~\cite[Theorem 2]{Bru:1991}, whereas the eigenvalues of our process $X$ behave like geometric Brownian motions, as we shall explain in the next section.

\paragraph{Further context: Dynkin's process and the eigenvalues.}
The Riemannian Brownian motion on $(\operatorname{SPD}(n),\bm{g}_{\nicefrac14,0})$ has been studied before, without the language of homogeneous spaces, by Norris, Rogers and Williams~\cite{NorrisRogersWilliams:1986}, who give an elementary account of some remarkable results of Dynkin~\cite{Dynkin:1961}. We recall the processes considered there (adding the superscript $\mathrm{NRW}$ to their notation, which does not match ours). Let $\mathcal{B}$ be an $n\times n$ matrix of independent standard Brownian motions and let $G$ be the \emph{right}-invariant Brownian motion on $\operatorname{GL}(n)$, i.e., the solution of $\bm{d}G=(\circ\bm{d}\mathcal{B})\,G$, $G(0)=I_n$. Two $\operatorname{SPD}(n)$-valued processes are studied in~\cite{NorrisRogersWilliams:1986}:
\begin{equation*}
    X^{\mathrm{NRW}} := GG^{\intercal}\qquad\text{and}\qquad Y^{\mathrm{NRW}} := G^{\intercal}G;
\end{equation*}
the latter is the process studied by Dynkin, the former is referred to as the `flow' process.

\begin{proposition}\label{prop:spd_dynkin}
The Riemannian Brownian motion on $(\operatorname{SPD}(n),\bm{g}_{\nicefrac{1}{4},0})$ starting at $I_n$ coincides in law with the process $Y^{\mathrm{NRW}}$.
\end{proposition}
\begin{proof}
Recall that the Riemannian Brownian motion on $(\operatorname{SPD}(n),\bm{g}_{\nicefrac{1}{4},0})$  is given by $\pi(g)$, where $\bm{d}g = g\,\bm{d}W_{\mathfrak{m}}$, $g(0)=I_n$, and $W_{\mathfrak{m}}$ is the Brownian motion given by~\eqref{eq:spd_Wm} with $\alpha=\nicefrac{1}{4}$ and $\beta = 0$. On the other hand, $G^{\intercal}$ solves the \emph{left}-invariant equation $\bm{d}G^{\intercal}=G^{\intercal}\circ\bm{d}\mathcal{B}^{\intercal}$, $G^{\intercal}(0)=I_n$. Thus $G^{\intercal}$ is driven by a standard Brownian motion in all of $\mathfrak{gl}(n)$, whereas $g$ is driven by a Brownian motion in $\mathfrak{m}=\operatorname{Sym}(n)$ only. In particular, $G^{\intercal}$ and $g$ are different processes, but we claim their projections have the same law. Indeed, as $\mathfrak{gl}(n)=\mathfrak{so}(n)\oplus\operatorname{Sym}(n)$ is an orthogonal splitting with respect to the Frobenius inner product, the generator of $G^{\intercal}$ is $\frac12\sum_{1\leq k<\ell\leq n}\tilde{A}_{k,\ell}^2+\frac12\sum_{1\leq i\leq j\leq n}\tilde{E}_{i,j}^2$ with $(A_{k,\ell})_{1\leq k<\ell\leq n}$ and $(E_{i,j}^{(\nicefrac{1}{4},0)})_{1\leq i\leq j\leq n}$ orthonormal bases of $\mathfrak{so}(n)$ and $\operatorname{Sym}(n)$ (recall that $\tilde{A}$ denotes the left-invariant vector field associated with $A\in\mathfrak{gl}(n)$, see~\eqref{eq:def_tilde_e}). Recall that $\frac12\sum_{1\leq i\leq j\leq n}(\tilde{E}_{i,j}^{(\nicefrac{1}{4},0)})^2$ is the generator of $g$. For $A\in\mathfrak{so}(n)$ the left-invariant vector field $\tilde{A}$ generates $h\mapsto he^{tA}$, which leaves $hh^{\intercal}$ unchanged, so $\tilde{A}(f\circ\pi)=0$ for all $f\in C^\infty(\operatorname{SPD}(n))$: the $\mathfrak{so}(n)$-part of the noise drops out on pull-backs, and the projected generators of $G^{\intercal}$ and $g$ coincide (cf.\ the proof of Proposition~\ref{prop:unimodularredhombm}). Hence $Y^{\mathrm{NRW}}=\pi(G^{\intercal})$ has the law of $X=\pi(g)$ with $B=\langle\cdot,\cdot\rangle_{1,0}$, which by~\eqref{eq:spd_gB_vs_ThP} is the Riemannian Brownian motion for $\bm{g}_{\nicefrac14,0}=\frac14\tr\big((\Sigma^{-1}\bm{d}\Sigma)^2\big)$.
\end{proof}

(The flow process $X^{\mathrm{NRW}}$, i.e., the right-invariant Brownian motion projected onto the space of \emph{left} cosets, is also a diffusion, but it is not $\operatorname{GL}(n)$-invariant, cf.\ the discussion preceding~\eqref{eq:stratsde}.)

By~\cite[Theorem A and \S 5]{NorrisRogersWilliams:1986} and the fact that $Y^{\mathrm{NRW}}(t)$ and $X^{\mathrm{NRW}}(t)$ have the same spectrum, we have that the eigenvalues of $Y^{\mathrm{NRW}}(t)$ are distinct for all $t>0$ (the eigenvalues `never collide') and, if labelled increasingly, satisfy $t^{-1}\log\lambda_i(t)\to 2i-(n+1)$ almost surely. It follows directly from our expression for $W_{\mathfrak{m}}=\frac{1}{2\sqrt{\alpha}}\hat{W}$, see~\eqref{eq:spd_Wm}, that the Riemannian Brownian motion on $(\operatorname{SPD}(n),\bm{g}_{\alpha,0})$ starting at $I_n$ has the law of $(Y^{\mathrm{NRW}}(t/(4\alpha)))_{t\geq 0}$. Thus, by Proposition~\ref{prop:spd_dynkin} and the discussion above, the eigenvalues $\lambda_1(t)<\ldots<\lambda_n(t)$ of the Riemannian Brownian motion $X$ on $(\operatorname{SPD}(n),\bm{g}_{\alpha,0})$ starting at $I_n$ almost surely never collide and satisfy
\begin{equation}\label{eq:spd_lyapunov}
    \lim_{t\to\infty}\frac{\log\lambda_i(t)}{t} = \frac{2i-(n+1)}{4\alpha}\quad\text{almost surely},\qquad i=1,\ldots,n.
\end{equation}
The exponential fanning out of the eigenvalues, with Lyapunov exponents in arithmetic progression, is clearly visible in Figure~\ref{fig:spd10eigs}. Note that the exponents in~\eqref{eq:spd_lyapunov} sum to zero, in accordance with the fact that $\log\det X(t) = \log\det\Sigma_0 + 2\tr\big(W_{\mathfrak{m}}(t)\big) = \log\det\Sigma_0 + \frac{1}{\sqrt{\alpha}}\sum_{i=1}^n W_{i,i}(t)$ is a Brownian motion with variance $\frac{n}{\alpha}t$ (by Jacobi's formula and the Stratonovich chain rule, $\bm{d}\log\det g = \tr(g^{-1}\circ\bm{d}g)=\tr(\circ\bm{d}W_{\mathfrak{m}})$). The same identity holds for the Wong--Zakai approximation, $\log\det X_h(kh) = \log\det\Sigma_0+2\tr(W_{\mathfrak{m}}(kh))$, without any discretisation error, which provides a simple exact test for implementations of~\eqref{eq:spd_WZ}.\par

\paragraph{Simulations.}
For our simulations we consider the metric $\bm{g}_{1,0}$ on $\operatorname{SPD}(10)$, i.e., $\alpha=1$ and $\beta=0$, for which the orthonormal basis of $(\mathfrak{m},B)$ is $(\frac12 E_{i,j})_{1\leq i\leq j\leq 10}$ and $W_{\mathfrak{m}}=\frac12\hat{W}$. In Figure~\ref{fig:spd10eigs} we show the evolution of the eigenvalues of the Wong--Zakai approximation~\eqref{eq:spd_WZ} of the Riemannian Brownian motion starting from $\Sigma_0=100\cdot I_{10}$, i.e., of $X_h(kh)=100\,g_1\cdots g_k\,g_k^{\intercal}\cdots g_1^{\intercal}$; by~\eqref{eq:spd_lyapunov}, the Lyapunov exponents of the eigenvalues are $\frac{2i-11}{4}$, $i=1,\ldots,10$, i.e., $-\frac94,-\frac74,\ldots,\frac94$. Together with the distance process $t\mapsto d_{1,0}(\Sigma_0,X_h(t))$, computed using~\eqref{eq:distancespd_eig} and shown in Figure~\ref{fig:spd10dist}, the non-compact nature of the space is clearly visible.
\begin{figure}[H]
    \centering
    \includegraphics[width=0.9\textwidth]{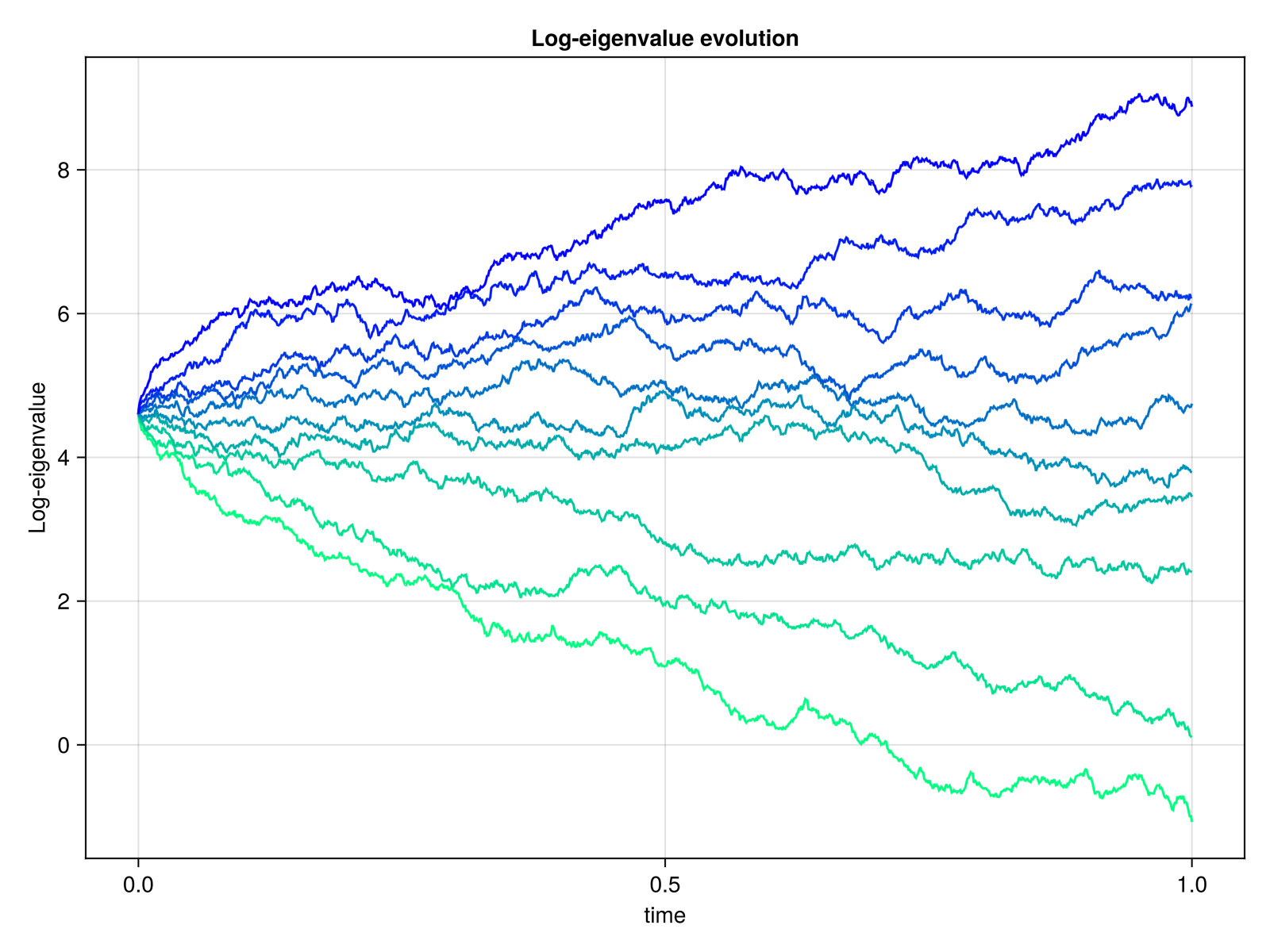}
    \caption{The natural logarithm of the eigenvalues over time of a realisation of the Wong--Zakai approximation of the Riemannian Brownian motion on $(\operatorname{SPD}(10),\bm{g}_{1,0})$, starting from $100\cdot I_{10}$, with step-size $h=0.001$. The eigenvalues fan out exponentially with Lyapunov exponents $\frac{2i-11}{4}$, $i=1,\ldots,10$, see equation~\eqref{eq:spd_lyapunov}.}
    \label{fig:spd10eigs}
\end{figure}

\begin{figure}[H]
    \centering
    \includegraphics[width=0.9\textwidth]{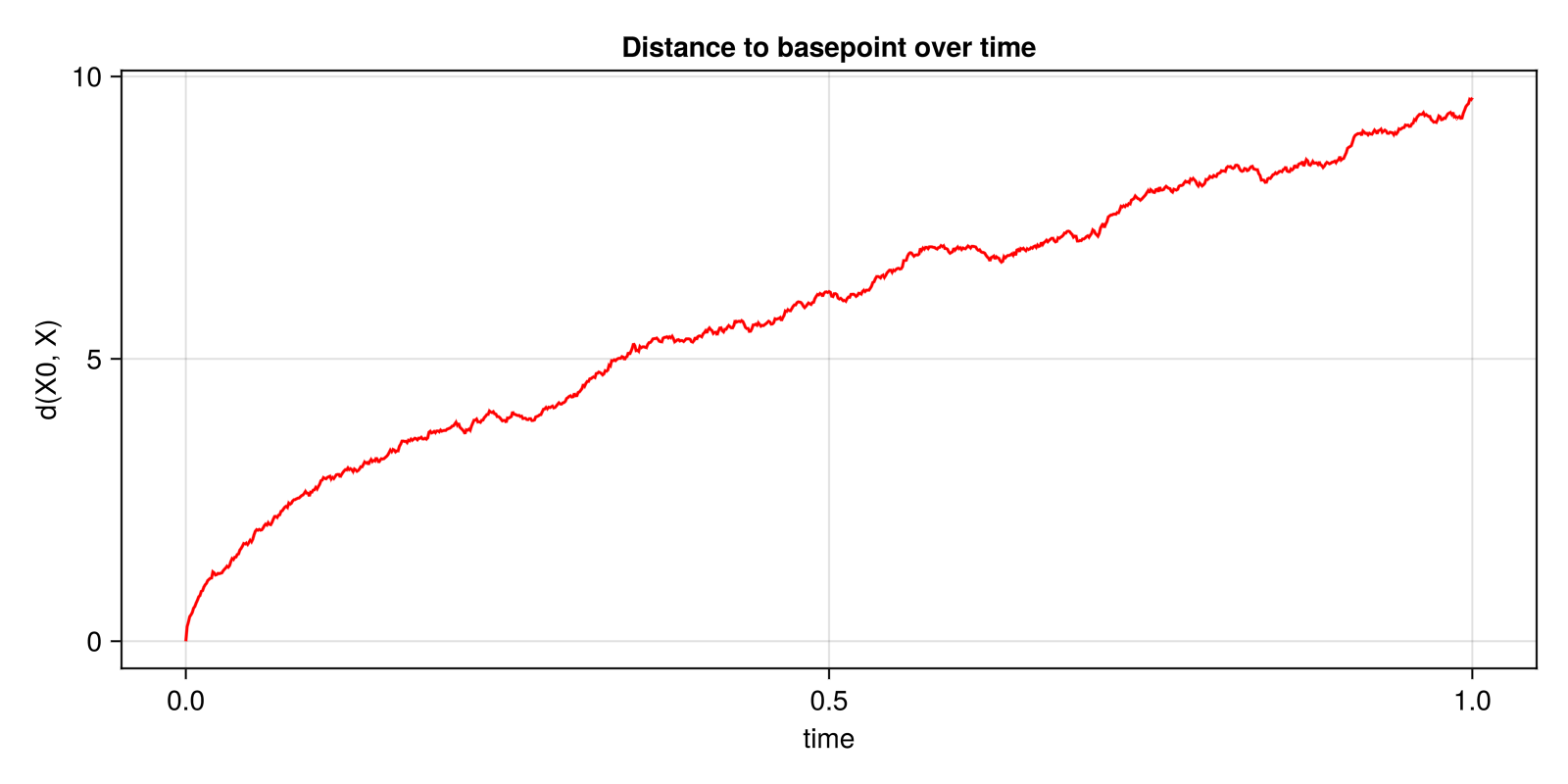}
    \caption{The geodesic distance process $t\mapsto d_{1,0}(\Sigma_0,X_h(t))$ on $(\operatorname{SPD}(10),\bm{g}_{1,0})$ associated with the realisation of Figure~\ref{fig:spd10eigs}, with the distance function given by~\eqref{eq:distancespd_eig}. The step-size is $h=0.001$.}
    \label{fig:spd10dist}
\end{figure}
\subsection{Flag manifolds}\label{sec:flag}
\paragraph{Flag manifolds as reductive homogeneous spaces.} A \emph{flag} is a nested sequence of (strict) subspaces of\footnote{Flags in other vector spaces than $\R^n$ have also been considered in the literature, but we restrict ourselves to the real setting.} $\R^n$, $n\in \{2,3,\ldots\}$, i.e., a sequence of the form
\begin{equation}\label{eq:flag}
\{0\} \subsetneq V_1 \subsetneq \ldots \subsetneq V_r \subsetneq \R^n,
\end{equation}
where necessarily $r< n$. The \emph{signature} of a flag is the vector $(\operatorname{dim}(V_1),\ldots,\operatorname{dim}(V_r))$. Given $1\leq d_1<d_2<\ldots<d_r<n$, we let $\operatorname{Fl}(d_1,\ldots,d_r;\R^n ) $ denote the set of all flags with signature $(d_1,\ldots,d_r)$. \par
In what follows, we fix $n\in \{2,3,\ldots\}$, $r<n$, and $1\leq d_1< \ldots < d_r <n$. In addition, we set $d_0=0$, $d_{r+1}=n$ and $n_i=d_i-d_{i-1}$, $i\in \{1,\ldots,r+1\}$. We show that Proposition~\ref{prop:HPviaH} allows us to identify $\operatorname{Fl}(d_1,\ldots,d_r;\R^n)$ as a homogeneous space. Indeed, let $\bm{e}_1,\ldots,\bm{e}_n$ denote the standard orthonormal basis of $\R^n$ and set $E_i = \operatorname{span}\{\bm{e}_{1},\ldots,\bm{e}_{d_i}\}$, $i\in \{1,\ldots,r\}$. The flag $o = (E_1 \subsetneq E_2 \subsetneq \ldots \subsetneq E_r)$ serves as our base point. For any flag~\eqref{eq:flag} in $\operatorname{Fl}(d_1,\ldots,d_r;\R^n)$ pick $A\in \operatorname{O}(n)$ such that $A E_i = V_i$, $i\in \{1,\ldots,r\}$. Next, set
\begin{equation}\label{eq:H for flags}
H = \left\{\begin{pmatrix}
            A_{1,1} & 0 & \cdots & 0\\
            0 & A_{2,2} & \ddots & \vdots\\
            \vdots & \ddots & \ddots & 0\\
            0 & \cdots & 0 & A_{r+1,r+1}
        \end{pmatrix} : A_{i,i}\in \operatorname{O}(n_i),\, i\in \{1,\ldots,r+1\}\right\}\cong \operatorname{O}(n_1)\times\ldots\times \operatorname{O}(n_{r+1}).
\end{equation}
Note that $A'E_i = V_i$ for all $i\in \{1,\ldots,r\}$ if and only if $A'\in AH$, so that we can identify the flag~\eqref{eq:flag} with the coset $AH=\{ A h \colon h\in H\}$. As $H$ is a closed subgroup of $\operatorname{O}(n)$, Proposition~\ref{prop:HPviaH} provides $\operatorname{O}(n)/H$, and thus $\operatorname{Fl}(d_1,\ldots,d_r;\R^n)$, with the structure of a smooth manifold on which $\operatorname{O}(n)$ acts smoothly and transitively; the natural projection is $\pi(A)=(AE_1,\ldots,AE_r)$, $A\in\operatorname{O}(n)$, i.e., $V_i=AE_i$ is spanned by the first $d_i$ columns of $A$.

The flag manifold $\operatorname{Fl}(1,2,\ldots,n-1;\R^n)$ is called the \emph{complete} (or full) flag manifold, note that
\begin{equation}
    \operatorname{Fl}(1,2,\hdots,n-1;\R^n) \cong \operatorname{O}(n)/(\operatorname{O}(1))^n\cong \operatorname{O}(n)/(\mathbb{Z}_2)^n.
\end{equation}
Non-complete flag manifolds are called \emph{partial} flag manifolds, a well-known subclass of the partial flag manifolds are those for which $r=1$, i.e., the \emph{Grassmannians} $\operatorname{Gr}_k(\R^n)$ which parametrise all $k$-dimensional subspaces in $\R^n$:
\begin{equation}
    \operatorname{Gr}_k(\R^n)\cong \operatorname{O}(n)/(\operatorname{O}(k)\times\operatorname{O}(n-k)).
\end{equation}
Note that $V\mapsto V^{\perp}$ is an $\operatorname{O}(n)$-equivariant diffeomorphism $\operatorname{Gr}_k(\R^n)\rightarrow \operatorname{Gr}_{n-k}(\R^n)$.

Having established that the flag manifold is a homogeneous space, we now turn to the question whether it is reductive (in the sense of Definition~\ref{def:rhs}). Recall that the Lie algebra of $\operatorname{O}(n)$ is $\mathfrak{so}(n)$, the space of skew-symmetric $n\times n$ matrices, and let us write $X\in \mathfrak{so}(n)$ in block form $X=(X_{i,j})_{1\leq i,j\leq r+1}$ with $X_{i,j}\in \R^{n_i\times n_j}$ according to the partition $n=n_1+\ldots+n_{r+1}$, so that $X_{j,i}=-X_{i,j}^{\intercal}$. The isotropy algebra $\mathfrak{h}$, i.e., the Lie algebra of $H$, consists of the block diagonal skew-symmetric matrices,
    \begin{equation}
        \mathfrak{h} = \left\{\begin{pmatrix}
            X_{1,1} & 0 & \cdots & 0\\
            0 & X_{2,2} & \ddots & \vdots\\
            \vdots & \ddots  & \ddots & 0\\
            0 & \cdots & 0 & X_{r+1,r+1}
        \end{pmatrix} : X_{i,i}\in \mathfrak{so}(n_i),\, i\in \{1,\ldots,r+1\}\right\}\cong \mathfrak{so}(n_1)\oplus \cdots \oplus\mathfrak{so}(n_{r+1}),
    \end{equation}
and as complement we take the space of skew-symmetric matrices whose diagonal blocks vanish,
    \begin{equation}\label{eq:flag_m}
        \mathfrak{m} = \left\{\begin{pmatrix}
            0 & X_{1,2} & \cdots  & X_{1,r+1}\\
            -X_{1,2}^{\intercal} & 0 & \ddots & \vdots \\
            \vdots & \ddots & \ddots & X_{r,r+1} \\
            -X_{1,r+1}^{\intercal} & \cdots & -X_{r,r+1}^{\intercal} & 0
        \end{pmatrix}: X_{i,j}\in \R^{n_i\times n_j},\, 1\leq i < j\leq r+1\right\},
    \end{equation}
so that clearly $\mathfrak{so}(n)=\mathfrak{h}\oplus \mathfrak{m}$ and $\dim \mathfrak{m} = \sum_{1\leq i<j\leq r+1}n_in_j$. For $h=\operatorname{diag}(A_{1,1},\ldots,A_{r+1,r+1})\in H$ and $X\in \mathfrak{so}(n)$ one has $\operatorname{Ad}_hX=hXh^{-1}=hXh^{\intercal}$, i.e.,
\begin{equation}\label{eq:flag_Ad}
    (\operatorname{Ad}_hX)_{i,j} = A_{i,i}X_{i,j}A_{j,j}^{\intercal},\quad 1\leq i,j\leq r+1.
\end{equation}
In particular, $\operatorname{Ad}_h$ maps $\mathfrak{m}$ into $\mathfrak{m}$ for every $h\in H$, so that $\operatorname{Fl}(d_1,\ldots,d_r;\R^n)\cong \operatorname{O}(n)/H$ is a reductive homogeneous space in the sense of Definition~\ref{def:rhs} with reductive decomposition $\mathfrak{so}(n)=\mathfrak{h}\oplus\mathfrak{m}$. (As for $\operatorname{SPD}(n)$ in Section~\ref{sec:spd}, we verified the $\operatorname{Ad}_H$-invariance directly, as $H$ is not connected.) Note that, unlike in the previous three examples, one does \emph{not} have $[\mathfrak{m},\mathfrak{m}]\subseteq \mathfrak{h}$ as soon as $r\geq 2$: if $X\in \mathfrak{m}$ has $X_{1,2}$ as its only nonzero block above the diagonal and $Y\in\mathfrak{m}$ has $Y_{2,3}$ as its only nonzero block above the diagonal, then $[X,Y]$ has the nonzero $(1,3)$-block $X_{1,2}Y_{2,3}$. For $r=1$, i.e., for the Grassmannians, one does have $[\mathfrak{m},\mathfrak{m}]\subseteq\mathfrak{h}$ (both $\mathfrak{h}$ and $\mathfrak{m}$ then consist of $2\times 2$ block matrices, and the product of two matrices with vanishing diagonal blocks is block diagonal), so that the Grassmannians are symmetric spaces and Lemma~\ref{lemma: LTS} applies, but flag manifolds with $r\geq 2$ are not symmetric.

\paragraph{$\operatorname{O}(n)$-invariant metrics on the flag manifolds.} In order to identify the $\operatorname{Ad}_H$-invariant inner products on $\mathfrak{m}$, we decompose $\mathfrak{m}$ according to its blocks: for $1\leq i<j\leq r+1$ let $\mathfrak{m}_{i,j}\subseteq\mathfrak{m}$ denote the subspace of those $X\in \mathfrak{m}$ for which $X_{k,\ell}=0$ whenever $\{k,\ell\}\neq \{i,j\}$, so that $\mathfrak{m}_{i,j}\cong \R^{n_i\times n_j}$ via $X\mapsto X_{i,j}$ and
\begin{equation}\label{eq:flag_m_decomp}
\mathfrak{m} = \bigoplus_{1\leq i<j\leq r+1} \mathfrak{m}_{i,j}.
\end{equation}
By~\eqref{eq:flag_Ad}, each $\mathfrak{m}_{i,j}$ is $\operatorname{Ad}_H$-invariant, and under the identification $\mathfrak{m}_{i,j}\cong\R^{n_i\times n_j}$ the isotropy representation~\eqref{eq:isotropyrep} restricted to $\mathfrak{m}_{i,j}$ is given by
\begin{equation}\label{eq:flag_isotropy_block}
\rho_{i,j}(h)X_{i,j} = A_{i,i}X_{i,j}A_{j,j}^{\intercal},\quad h=\operatorname{diag}(A_{1,1},\ldots,A_{r+1,r+1})\in H,\ X_{i,j}\in \R^{n_i\times n_j},
\end{equation}
i.e., $H$ acts on $\mathfrak{m}_{i,j}$ through the factors $\operatorname{O}(n_i)\times\operatorname{O}(n_j)$ only. The following proposition is the analogue of Proposition~\ref{prop:irrepsphere}.

\begin{proposition}\label{prop:flag_irreps}
Let $1\leq i<j\leq r+1$. The representation $\rho_{i,j}$ of $H$ on $\mathfrak{m}_{i,j}\cong \R^{n_i\times n_j}$ defined in~\eqref{eq:flag_isotropy_block} is irreducible. Moreover, if $1\leq k<\ell\leq r+1$ and $(k,\ell)\neq(i,j)$, then the representations $\rho_{i,j}$ and $\rho_{k,\ell}$ are not isomorphic.
\end{proposition}

\begin{proof}
Let $W\subseteq \R^{n_i\times n_j}$ be a nonzero $\rho_{i,j}$-invariant subspace and let $X\in W$ be nonzero. By the singular value decomposition there exist $U\in \operatorname{O}(n_i)$ and $V\in \operatorname{O}(n_j)$ such that $\Sigma:=U^{\intercal}XV$ has nonnegative entries on its diagonal (i.e., in the positions $(p,p)$, $p\leq \min\{n_i,n_j\}$) and zeros elsewhere; as $X\neq 0$ we may assume that $\Sigma_{1,1}=\sigma_1>0$. By invariance, $\Sigma\in W$ and $D\Sigma \in W$, where $D=\operatorname{diag}(1,-1,\ldots,-1)\in \operatorname{O}(n_i)$. Hence $\frac{1}{2}(\Sigma+D\Sigma) = \sigma_1 \bm{e}_1\bm{e}_1^{\intercal}\in W$, so that $\bm{e}_1\bm{e}_1^{\intercal}\in W$ (here $\bm{e}_1$ denotes the first standard basis vector of $\R^{n_i}$ and of $\R^{n_j}$, respectively). Invariance of $W$ now yields $(U'\bm{e}_1)(V'\bm{e}_1)^{\intercal}\in W$ for all $U'\in \operatorname{O}(n_i)$ and $V'\in \operatorname{O}(n_j)$, i.e., $W$ contains all matrices of the form $\bm{u}\bm{v}^{\intercal}$ with $\bm{u}\in \R^{n_i}$ and $\bm{v}\in\R^{n_j}$ unit vectors, and hence all rank-one matrices. As every matrix is a sum of rank-one matrices, $W=\R^{n_i\times n_j}$, which proves irreducibility.\par
Next, let $(k,\ell)\neq (i,j)$ with $1\leq k<\ell\leq r+1$. As $\{i,j\}\neq\{k,\ell\}$, at least one of the indices $i,j$, say $i$, does not belong to $\{k,\ell\}$. Let $h\in H$ be the block diagonal matrix with $A_{i,i}=-I_{n_i}$ and $A_{p,p}=I_{n_p}$ for $p\neq i$. Then $\rho_{i,j}(h)=-\operatorname{id}$ by~\eqref{eq:flag_isotropy_block}, whereas $\rho_{k,\ell}(h)=\operatorname{id}$. If $T\colon \mathfrak{m}_{i,j}\rightarrow\mathfrak{m}_{k,\ell}$ were an isomorphism of representations, then $-T = T\circ \rho_{i,j}(h) = \rho_{k,\ell}(h)\circ T = T$, whence $T=0$, a contradiction.
\end{proof}

By Proposition~\ref{prop:flag_irreps}, the decomposition~\eqref{eq:flag_m_decomp} of the isotropy representation into irreducible components contains each component exactly once. Arguing as in Section~\ref{sec:spheres} (see the discussion following Proposition~\ref{prop:irrepsphere}), Schur's lemma implies that any $\operatorname{Ad}_H$-invariant inner product on $\mathfrak{m}$ renders the components $\mathfrak{m}_{i,j}$ mutually orthogonal (for $(k,\ell)\neq (i,j)$, the restriction to $\mathfrak{m}_{k,\ell}$ of the $B$-orthogonal projection onto $\mathfrak{m}_{i,j}$ is an intertwiner between non-isomorphic irreducible representations, hence zero), and that its restriction to each $\mathfrak{m}_{i,j}$ is a positive multiple of the Frobenius inner product $\tr(X_{i,j}^{\intercal}Y_{i,j})$, the latter being an $\operatorname{Ad}_H$-invariant inner product on $\mathfrak{m}_{i,j}$ by~\eqref{eq:flag_isotropy_block}. Thus, the $\operatorname{Ad}_H$-invariant inner products on $\mathfrak{m}$ are precisely the inner products
\begin{equation}\label{eq:flaginner}
\langle X,Y \rangle_{\alpha} := \sum_{1\leq i<j\leq r+1} \alpha_{i,j}\tr(X_{i,j}^{\intercal} Y_{i,j}),\quad X,Y\in \mathfrak{m},
\end{equation}
where $\alpha = (\alpha_{i,j})_{1\leq i < j \leq r+1} \in (0,\infty)^{r(r+1)/2}$. By~\eqref{eq:gB}, to every such $\alpha$ corresponds a $G$-invariant metric $\bm{g}_{\alpha}:=\bm{g}_{\langle\cdot,\cdot\rangle_{\alpha}}$ on $\operatorname{Fl}(d_1,\ldots,d_r;\R^n)$, and these are all the $\operatorname{O}(n)$-invariant Riemannian metrics on the flag manifold. For $X,Y\in \mathfrak{m}$ each block above the diagonal appears twice in $\tr(X^{\intercal}Y)$, whence
\begin{equation}\label{eq:flag_normal}
    \langle X,Y\rangle_{\alpha} = \tfrac12 \tr(X^{\intercal}Y) = Q(X,Y),\quad X,Y\in \mathfrak{m},\qquad \text{if } \alpha_{i,j}=1 \text{ for all } 1\leq i<j\leq r+1,
\end{equation}
where $Q(X,Y)=\frac12\tr(X^{\intercal}Y)$ is the $\operatorname{Ad}_{\operatorname{O}(n)}$-invariant inner product on $\mathfrak{so}(n)$ already encountered in Section~\ref{sec:spheres}; moreover, $\mathfrak{m}=\mathfrak{h}^{\perp}$ with respect to $Q$ by the block structure. Hence, for $\alpha\equiv 1$ the metric $\bm{g}_{\alpha}$ is the normal metric of Proposition~\ref{prop:normalhomogeneous} and $(\operatorname{Fl}(d_1,\ldots,d_r;\R^n),\langle\cdot,\cdot\rangle_{\alpha})$ is naturally reductive. For other choices of $\alpha$ this is in general no longer the case; see~\eqref{eq:flag_not_natred} below for an example. Nevertheless, as $\operatorname{O}(n)$ is compact and hence unimodular, Proposition~\ref{prop:unimodularredhombm} applies to \emph{every} $\alpha\in (0,\infty)^{r(r+1)/2}$: if $(g(t))_{t\geq 0}$ is the $\operatorname{O}(n)$-valued solution to~\eqref{eq:stratsde} driven by the $\mathfrak{m}$-valued Brownian motion~\eqref{eq:mbrownianmotion} associated with a $\langle\cdot,\cdot\rangle_{\alpha}$-orthonormal basis of $\mathfrak{m}$, then $\pi(g)$ is a Riemannian Brownian motion on $(\operatorname{Fl}(d_1,\ldots,d_r;\R^n),\bm{g}_{\alpha})$ starting at the standard flag $o$. As $\pi(g(t))=(g(t)E_1,\ldots,g(t)E_r)$, the $i$-th subspace of this random flag is spanned by the first $d_i$ columns of $g(t)$.

\paragraph{Constructing the Riemannian Brownian motion.} A $\langle\cdot,\cdot\rangle_{\alpha}$-orthonormal basis of $\mathfrak{m}$ is obtained by rescaling the standard basis: for $1\leq i<j\leq r+1$, $k\in\{1,\ldots,n_i\}$ and $\ell\in \{1,\ldots,n_j\}$ let $F^{i,j}_{k,\ell}\in \mathfrak{m}_{i,j}$ be the element whose $(i,j)$-block is $\bm{e}_k\bm{e}_{\ell}^{\intercal}\in \R^{n_i\times n_j}$ (so that its $(j,i)$-block is $-\bm{e}_{\ell}\bm{e}_{k}^{\intercal}$ and all other blocks vanish). Then
\begin{equation}\label{eq:flag_ONB}
    \big(\alpha_{i,j}^{-\nicefrac{1}{2}}F^{i,j}_{k,\ell}\big)_{1\leq i<j\leq r+1,\ k\in \{1,\ldots,n_i\},\ \ell\in \{1,\ldots,n_j\}}
\end{equation}
is an orthonormal basis of $(\mathfrak{m},\langle\cdot,\cdot\rangle_{\alpha})$, as $\langle F^{i,j}_{k,\ell},F^{i',j'}_{k',\ell'}\rangle_{\alpha} = \alpha_{i,j}\delta_{i,i'}\delta_{j,j'}\delta_{k,k'}\delta_{\ell,\ell'}$ by~\eqref{eq:flaginner}. Let $W^{i,j}_{k,\ell}$, $1\leq i<j\leq r+1$, $k\in\{1,\ldots,n_i\}$, $\ell\in\{1,\ldots,n_j\}$, be $\dim\mathfrak{m}$ independent standard real-valued Brownian motions and, for $1\leq i<j\leq r+1$, let $\bm{W}^{i,j}:=(W^{i,j}_{k,\ell})_{k\leq n_i,\ \ell\leq n_j}$ be the corresponding $\R^{n_i\times n_j}$-valued Brownian motion. The $\mathfrak{m}$-valued Brownian motion~\eqref{eq:mbrownianmotion} associated with the orthonormal basis~\eqref{eq:flag_ONB} is then the skew-symmetric matrix with vanishing diagonal blocks and
\begin{equation}\label{eq:flag_Wm}
    \big(W_{\mathfrak{m}}(t)\big)_{i,j} = \alpha_{i,j}^{-\nicefrac{1}{2}}\,\bm{W}^{i,j}(t),\quad 1\leq i<j\leq r+1,\ t\geq 0,
\end{equation}
and, as $(\textnormal{d}L_g)_eX=gX$ for the matrix group $\operatorname{O}(n)$, the $\operatorname{O}(n)$-valued process $g$ of~\eqref{eq:stratsde} is the solution to $\bm{d}g(t)=g(t)\circ\bm{d}W_{\mathfrak{m}}(t)$, $g(0)=I_n$. For the simulation we use the Wong--Zakai approximation~\eqref{eq:prewongzakai}: setting $\Delta W_{\mathfrak{m}}^{k,h}:=W_{\mathfrak{m}}(kh)-W_{\mathfrak{m}}((k-1)h)$ and $g_k:=\exp(\Delta W_{\mathfrak{m}}^{k,h})\in\operatorname{SO}(n)$, $k\in\N$, where $\exp$ is the matrix exponential of a skew-symmetric matrix, one has
\begin{equation}\label{eq:flag_WZ}
    X_h(kh) = \pi(g_1\cdots g_k) = \big(g_1\cdots g_k\,E_1,\ldots,g_1\cdots g_k\,E_r\big),\quad k\in\N,
\end{equation}
i.e., the $i$-th subspace of the approximate flag at time $kh$ is spanned by the first $d_i$ columns of the orthogonal matrix $g_1\cdots g_k$. In particular, $X_h(kh)$ is a flag of the correct signature for every $k\in \N$ without any projection step. Since $\operatorname{O}(n)$ is compact, Corollary~\ref{cor:WZ_rates_homogeneous}(ii) guarantees that $X_h$ converges to $\pi(g)$ in $L^p(\Omega)$ at rate $h^{\eta}$ for every $\eta\in(0,\frac12)$ (uniformly on compact time intervals and with respect to the geodesic distance of $\bm{g}_{\alpha}$).

\paragraph{Distance concepts: geodesic and extrinsic.} In order to visualise the Riemannian Brownian motion, we would like to plot the geodesic distance process $t\mapsto d_{\bm{g}_{\alpha}}(o,X_h(t))$ as in the previous examples. Unlike for the previous examples, however, the geodesic distance on the flag manifold is not available in closed form. Indeed, for a flag $AH$, $A\in \operatorname{O}(n)$, Lemma~\ref{lem:submersion} yields $d_{\bm{g}_{\alpha}}(o,AH)\leq d_{\bm{g}_G}(I_n,Ah)$ for all $h\in H$, where $\bm{g}_G$ is the left-invariant metric on $\operatorname{O}(n)$ induced by $Q_{\mathfrak{h}}\oplus\langle\cdot,\cdot\rangle_{\alpha}$ for an arbitrary inner product $Q_{\mathfrak{h}}$ on $\mathfrak{h}$. Conversely, any piecewise smooth curve in $\operatorname{Fl}(d_1,\ldots,d_r;\R^n)$ from $o$ to $AH$ admits a horizontal lift to $\operatorname{O}(n)$ starting at $I_n$, which has the same length and ends in the fibre $AH$ (this is the horizontal lift of a curve in the base of the principal $H$-bundle $\pi\colon\operatorname{O}(n)\rightarrow \operatorname{O}(n)/H$ with respect to the connection given by the horizontal distribution $g\mapsto (\textnormal{d}L_g)_e\mathfrak{m}$). Hence
\begin{equation}\label{eq:geodesicflagdistance}
    d_{\bm{g}_{\alpha}}(o,AH) = \min_{h\in H} d_{\bm{g}_G}(I_n,Ah),\quad A\in \operatorname{O}(n).
\end{equation}
For the normal metric ($\alpha\equiv 1$) one can take $Q_{\mathfrak{h}}=Q|_{\mathfrak{h}\times\mathfrak{h}}$, so that $\bm{g}_G$ is the bi-invariant metric induced by $Q$; in that case the geodesics of $\bm{g}_G$ through $I_n$ are the one-parameter subgroups and the distance from $I_n$ to $g\in \operatorname{SO}(n)$ is the minimal $Q$-length of a logarithm of $g$, i.e., $d_{\bm{g}_G}(I_n,g) = \frac{1}{\sqrt{2}}\|\log g\|_F$, with $\log$ the principal matrix logarithm and $\|\cdot\|_F$ the Frobenius norm (and $d_{\bm{g}_G}(I_n,g)=\infty$ if $\det g=-1$). As $H$ contains elements of determinant $-1$, \eqref{eq:geodesicflagdistance} thus becomes
\begin{equation}\label{eq:geodesicflagdistance_normal}
    d_{\bm{g}_{\alpha}}(o,AH) = \min_{h\in H,\ Ah\in\operatorname{SO}(n)} \tfrac{1}{\sqrt 2}\|\log(Ah)\|_F\quad\text{if } \alpha\equiv 1.
\end{equation}
Evaluating~\eqref{eq:geodesicflagdistance_normal} requires minimising over the $r+1$ orthogonal blocks of $h$ simultaneously, which is computationally unwieldy, and for $\alpha\not\equiv1$ the group distance $d_{\bm{g}_G}$ in~\eqref{eq:geodesicflagdistance} is itself not available in closed form. We therefore resort to an \emph{extrinsic} distance instead, namely the \emph{chordal distance} of~\cite[Definition 4]{mankovich2023chordal} (which goes back to~\cite{PitavalTirkkonen:2013}). For a subspace $V\subseteq\R^n$ let $P_V\in \operatorname{Sym}(n)$ denote the orthogonal projection onto $V$. Given two flags $x=(V_1,\ldots,V_r)$ and $y=(V_1',\ldots,V_r')$ in $\operatorname{Fl}(d_1,\ldots,d_r;\R^n)$, set $V_0=V_0'=\{0\}$ and let $Q_i:=P_{V_i}-P_{V_{i-1}}$ and $Q_i':=P_{V_i'}-P_{V_{i-1}'}$, $i\in\{1,\ldots,r\}$, be the orthogonal projections onto the `successive differences' $V_i\cap V_{i-1}^{\perp}$ and $V_i'\cap (V_{i-1}')^{\perp}$ of the two flags. The chordal distance between $x$ and $y$ is defined as
\begin{equation}\label{eq:chordalflag}
d_{c}(x,y) := \Big( \tfrac12\sum_{i=1}^{r} \|Q_i-Q_i'\|_F^2\Big)^{\nicefrac{1}{2}}
= \Big( \sum_{i=1}^{r}  \big(n_i - \tr(Q_iQ_i')\big)\Big)^{\nicefrac{1}{2}},
\end{equation}
where we used $\|Q_i\|_F^2=\tr(Q_i)=n_i$.
As the map $x\mapsto (Q_1,\ldots,Q_r)$ is injective and $\operatorname{O}(n)$-equivariant, $d_c$ is an $\operatorname{O}(n)$-invariant distance on $\operatorname{Fl}(d_1,\ldots,d_r;\R^n)$, namely the extrinsic distance of the embedding $x\mapsto(Q_1,\ldots,Q_r)$ into $\operatorname{Sym}(n)^r$ equipped with the norm $\|(S_1,\ldots,S_r)\|:=\big(\tfrac12\sum_{i=1}^r\|S_i\|_F^2\big)^{\nicefrac12}$. As $\tr(Q_iQ_i')\geq \max\{0,2n_i-n\}$ for projections of rank $n_i$, it is bounded by
\begin{equation}\label{eq:chordalflag_bound}
    d_{c}(x,y)\leq \Big(\sum_{i=1}^r \min\{n_i,n-n_i\}\Big)^{\nicefrac12}.
\end{equation}
The chordal distance is related to the invariant metrics $\bm{g}_{\alpha}$ as follows. Let $Q_i^{o}:=P_{E_i}-P_{E_{i-1}}$ denote the projections associated with the base flag $o$, i.e., $Q_i^{o}$ is the block diagonal matrix with $I_{n_i}$ in the $i$-th diagonal block and zeros elsewhere. Differentiating $t\mapsto \exp(tX)Q_i^{o}\exp(-tX)$ at $t=0$ for $X\in\mathfrak{m}$ yields the tangent vector $[X,Q_i^{o}]$, so that $\|[X,Q_i^{o}]\|_F^2 = 2\sum_{k=1}^{i-1} \|X_{k,i}\|_F^2 + 2\sum_{k=i+1}^{r+1} \|X_{i,k}\|_F^2$. It now follows that the pullback to $T_o\operatorname{Fl}(d_1,\ldots,d_r;\R^n)\cong\mathfrak{m}$ of the metric on $\operatorname{Sym}(n)^r$ underlying~\eqref{eq:chordalflag} is
\begin{equation}\label{eq:chordal_vs_alpha}
    \tfrac12\sum_{i=1}^r \|[X,Q_i^{o}]\|_F^2 = \sum_{1\leq k<i \leq r}2\|X_{k,i}\|_F^2 + \sum_{k=1}^{r}\|X_{k,r+1}\|_F^2 = \langle X,X\rangle_{\alpha^{c}},\qquad \alpha^{c}_{k,\ell}:=\begin{cases} 2, & \ell\leq r,\\ 1, & \ell=r+1,\end{cases}
\end{equation}
i.e., the embedding $x\mapsto(Q_1,\ldots,Q_r)$ is an isometric embedding of $(\operatorname{Fl}(d_1,\ldots,d_r;\R^n),\bm{g}_{\alpha^{c}})$ (by $\operatorname{O}(n)$-equivariance it suffices to check this at $o$), and $d_{c}$ is the extrinsic distance of this embedding. In particular, by~\eqref{eq:metric_lowerbound}, $d_{c}\leq d_{\bm{g}_{\alpha^{c}}}$, and the two distances agree to first order for nearby flags. For the Grassmannians ($r=1$) one has $\alpha^c\equiv 1$, so that $d_c$ is the extrinsic distance of an isometric embedding of the normal metric.

\paragraph{Example 1: the Grassmannians $\operatorname{Gr}_1(\R^3)$ and $\operatorname{Gr}_2(\R^3)$.} As a first concrete example we consider $\operatorname{Fl}(1;\R^3)=\operatorname{Gr}_1(\R^3)$ and $\operatorname{Fl}(2;\R^3)=\operatorname{Gr}_2(\R^3)$, which are diffeomorphic via $V\mapsto V^{\perp}$. The Grassmannian $\operatorname{Gr}_1(\R^3)$ consists of all lines through the origin in $\R^3$, i.e., it is the real projective plane $\mathbb{R}P^2$, and the map $S^2\ni\bm{v}\mapsto \operatorname{span}\{\bm{v}\}\in \operatorname{Gr}_1(\R^3)$ is a two-fold covering. Here $n_1=1$, $n_2=2$, $\mathfrak{m}=\mathfrak{m}_{1,2}\cong\R^{1\times 2}$ is the space of matrices $\big(\begin{smallmatrix} 0 & \bm{v}^{\intercal}\\ -\bm{v} & 0_{2\times 2}\end{smallmatrix}\big)$, $\bm{v}\in\R^2$, and the only invariant metrics are positive multiples of $\tr(X_{1,2}^{\intercal}Y_{1,2})$, exactly as for the sphere (Section~\ref{sec:spheres}; up to a relabelling of the coordinates, $\mathfrak{m}$ is the same subspace of $\mathfrak{so}(3)$ as in~\eqref{eq:snmatrep}, and the covering map is $\operatorname{O}(3)$-equivariant). In particular, for $\alpha_{1,2}=1$ the covering map $S^2\rightarrow\operatorname{Gr}_1(\R^3)$ is a local isometry for the round metric on $S^2$, and the Riemannian Brownian motion $\pi(g)$ on $\operatorname{Gr}_1(\R^3)$ is the image under the covering map of the Riemannian Brownian motion $g(t)\bm{e}_1$ on $S^2$ driven by the same noise. Locally, the two processes are thus indistinguishable, but the global topology is different, and this is visible in the distance processes: the geodesic distance on $\operatorname{Gr}_1(\R^3)$ is $d(\operatorname{span}\{\bm{u}\},\operatorname{span}\{\bm{v}\})=\arccos|\bm{u}\cdot\bm{v}|\leq\frac{\pi}{2}$ for unit vectors $\bm{u},\bm{v}$, as opposed to $\arccos(\bm{u}\cdot\bm{v})\leq \pi$ on $S^2$, and the chordal distance~\eqref{eq:chordalflag} is $d_{c}(\operatorname{span}\{\bm{u}\},\operatorname{span}\{\bm{v}\}) = \tfrac{1}{\sqrt2}\|\bm{u}\bm{u}^{\intercal}-\bm{v}\bm{v}^{\intercal}\|_F=\sqrt{1-(\bm{u}\cdot\bm{v})^2}$, which is maximal, equal to $1$, for orthogonal lines, whereas the chordal distance $\|\bm{u}-\bm{v}\|$ on $S^2$ is maximal, equal to $2$, for antipodal points.

\paragraph{Example 2: $\operatorname{Fl}(1,2;\R^3)$ with an anisotropic metric.} Next, consider the complete flag manifold $\operatorname{Fl}(1,2;\R^3)$, for which $n_1=n_2=n_3=1$. The isotropy group $H=\operatorname{O}(1)^3\cong(\mathbb{Z}_2)^3$ consists of the eight diagonal matrices with entries $\pm 1$, so that $\mathfrak{h}=\{0\}$ and $\mathfrak{m}=\mathfrak{so}(3)$. The three components $\mathfrak{m}_{1,2}$, $\mathfrak{m}_{1,3}$ and $\mathfrak{m}_{2,3}$ of~\eqref{eq:flag_m_decomp} are one-dimensional and spanned by
\begin{equation}\label{eq:flag_F123}
    F_1 := F^{1,2}_{1,1} = \begin{pmatrix}
        0 & 1 & 0 \\
        -1 & 0 & 0\\
        0 & 0 & 0
    \end{pmatrix}, \quad F_2 := F^{1,3}_{1,1} = \begin{pmatrix}
        0 & 0 & 1 \\
        0 & 0 & 0\\
        -1 & 0 & 0
    \end{pmatrix}, \quad F_3 := F^{2,3}_{1,1} = \begin{pmatrix}
        0 & 0 & 0 \\
        0 & 0 & 1\\
        0 & -1 & 0
    \end{pmatrix},
\end{equation}
respectively, and the invariant inner products~\eqref{eq:flaginner} are $\langle X,Y\rangle_{\alpha}=\alpha_{1,2}X_{1,2}Y_{1,2}+\alpha_{1,3}X_{1,3}Y_{1,3}+\alpha_{2,3}X_{2,3}Y_{2,3}$, with orthonormal basis $\alpha_{1,2}^{-\nicefrac12}F_1$, $\alpha_{1,3}^{-\nicefrac12}F_2$, $\alpha_{2,3}^{-\nicefrac12}F_3$. One computes $[F_1,F_2]=-F_3$ and $[F_1,F_3]=F_2$, so that
\begin{equation}\label{eq:flag_not_natred}
    \langle F_2,[F_1,F_3]\rangle_{\alpha} + \langle [F_1,F_2],F_3\rangle_{\alpha} = \alpha_{1,3}-\alpha_{2,3};
\end{equation}
as $\mathfrak{h}=\{0\}$, condition~\eqref{eq:naturallyreductive} thus fails unless $\alpha_{1,3}=\alpha_{2,3}$, and by symmetry $(\operatorname{Fl}(1,2;\R^3),\langle\cdot,\cdot\rangle_{\alpha})$ is naturally reductive if and only if $\alpha_{1,2}=\alpha_{1,3}=\alpha_{2,3}$, i.e., if and only if $\bm{g}_{\alpha}$ is a multiple of the normal metric. For all other $\alpha$ it is the unimodularity of $\operatorname{O}(3)$ that guarantees, by Proposition~\ref{prop:unimodularredhombm}, that $\pi(g)$ is a Riemannian Brownian motion.

As $\mathfrak{m}=\mathfrak{so}(3)$, the SDE~\eqref{eq:stratsde} is non-degenerate, and its solution $g$, being continuous with $g(0)=I_3$, takes values in $\operatorname{SO}(3)$. In particular, $g$ is a Riemannian Brownian motion on $\operatorname{SO}(3)$ for the left-invariant metric induced by $\langle\cdot,\cdot\rangle_{\alpha}$ (by Proposition~\ref{prop:unimodularredhombm} with the trivial isotropy group). Since $\mathfrak{h}=\{0\}$, Lemma~\ref{lem:submersion} shows that $\pi|_{\operatorname{SO}(3)}\colon \operatorname{SO}(3)\rightarrow\operatorname{Fl}(1,2;\R^3)$ is a local isometry for this metric and $\bm{g}_{\alpha}$, so the Riemannian Brownian motion $\pi(g)$ on $\operatorname{Fl}(1,2;\R^3)$ is the image of a Riemannian Brownian motion on $\operatorname{SO}(3)$ under the map that forgets the signs of the columns of a rotation matrix.

For the simulation we choose $\alpha_{1,2}=\frac{1}{9}$, $\alpha_{1,3}=\frac{1}{6}$ and $\alpha_{2,3}=1$, for which $(\operatorname{Fl}(1,2;\R^3),\langle\cdot,\cdot\rangle_{\alpha})$ is not naturally reductive by~\eqref{eq:flag_not_natred}. The orthonormal basis~\eqref{eq:flag_ONB} is $3F_1$, $\sqrt{6}F_2$, $F_3$, so that
\begin{equation}
    W_{\mathfrak{m}}(t) = 3W_1(t)F_1+\sqrt{6}\,W_2(t)F_2+W_3(t)F_3 = \begin{pmatrix}
        0 & 3W_1(t) & \sqrt{6}\,W_2(t) \\
        -3W_1(t) & 0 & W_3(t)\\
        -\sqrt{6}\,W_2(t) & -W_3(t) & 0
    \end{pmatrix},\quad t\geq 0,
\end{equation}
for independent standard real-valued Brownian motions $W_1,W_2,W_3$, and the Wong--Zakai approximation~\eqref{eq:flag_WZ} reads, with $g_k=\exp(\Delta W^{k,h}_{\mathfrak{m}})\in \operatorname{SO}(3)$ and $A_k:=g_1\cdots g_k$,
\begin{equation}\label{eq:flag_WZ_example}
    X_h(kh) = \pi(g_1\cdots g_k) = \big(A_kE_1,\,A_kE_2\big) = \big(\operatorname{span}\{A_k\bm{e}_1\},\ \operatorname{span}\{A_k\bm{e}_1,A_k\bm{e}_2\}\big),\quad k\in \N,
\end{equation}
i.e., the line and the plane of the approximate flag at time $kh$ are spanned by the first column and by the first two columns of the rotation matrix $A_k$, respectively. The matrix exponentials $g_k$ can be evaluated using Rodrigues' formula~\eqref{eq:sphere_exp} (which applies to any $X\in \mathfrak{so}(3)$ with $\theta=\|X\|_F/\sqrt{2}$, as $X^3=-\theta^2X$). The resulting anisotropic Riemannian Brownian motion is the image of a Riemannian Brownian motion on $\operatorname{SO}(3)$ for a left-invariant metric that is not bi-invariant, and its trajectories differ markedly from those for the normal metric. A trajectory is shown in Figure~\ref{fig:flagtraj}: we plot the first column $A_k\bm{e}_1$ of the rotation matrix $A_k$ in~\eqref{eq:flag_WZ_example}, i.e., a unit vector spanning the line $V_1$ of the flag (the line itself only determines this vector up to sign, but the plotted curve is continuous as $A_k$ depends continuously on the increments). The chordal distance process $t\mapsto d_{c}(o,X_h(t))$ is shown in Figure~\ref{fig:flagdist}; here, by~\eqref{eq:chordalflag}, $d_c(o,X_h(kh))^2 = \big(1-(A_k)^2_{1,1}\big)+\big(1-(A_k)^2_{2,2}\big)$, which is bounded by $2$ by~\eqref{eq:chordalflag_bound}.

\begin{figure}[H]
    \centering
    \includegraphics[width=0.9\textwidth]{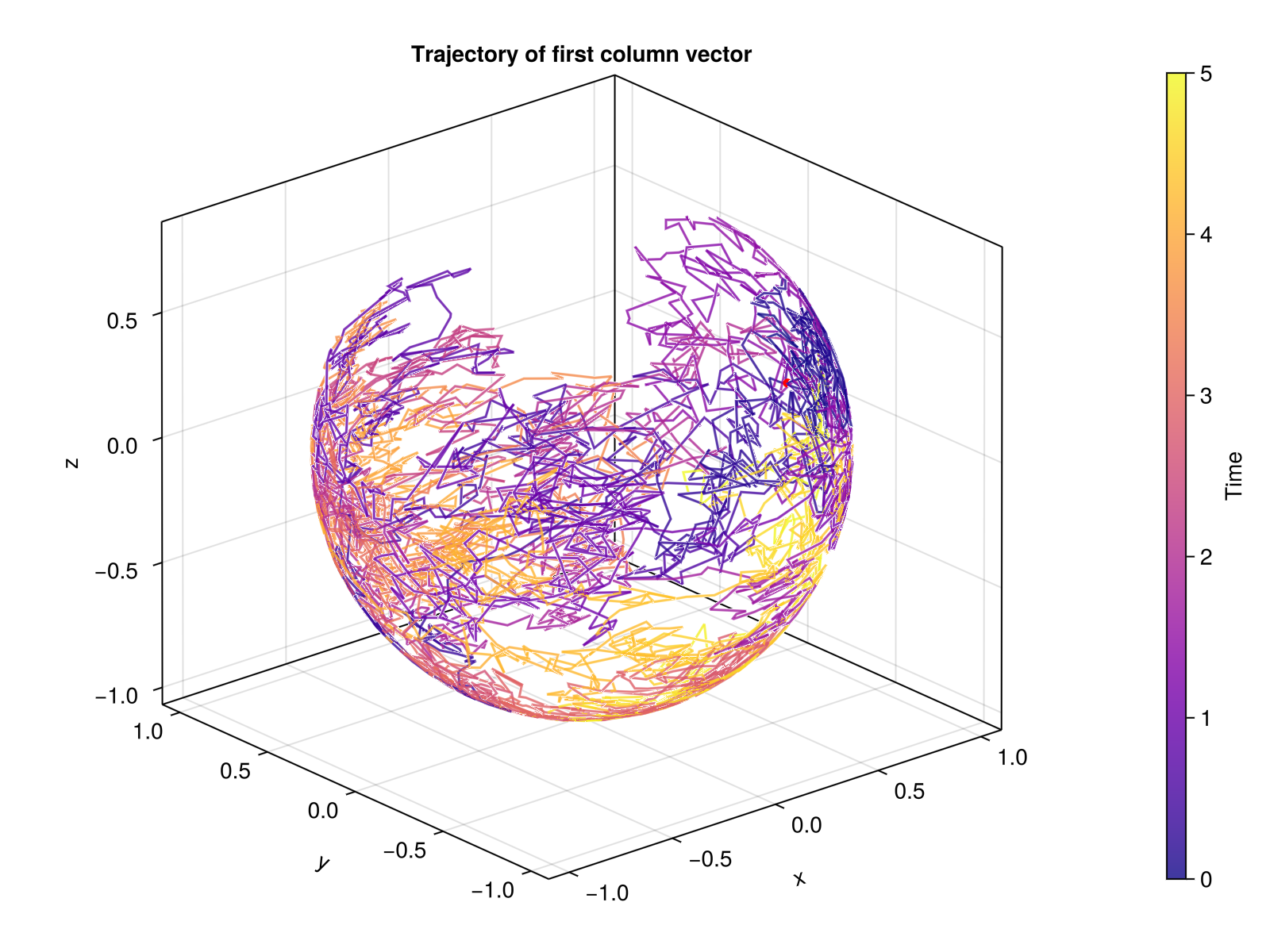}
    \caption{A realisation of the Wong--Zakai approximation of the anisotropic Riemannian Brownian motion on the complete flag manifold $\operatorname{Fl}(1,2;\R^3)$ with $\alpha_{1,2}=\frac19$, $\alpha_{1,3}=\frac16$, $\alpha_{2,3}=1$, started at the standard flag $o$ and simulated until time $T=5$. Plotted is the first column of the orthogonal matrix representing the flag, i.e., a unit vector spanning the line $V_1$ of the flag, which lives on $S^2$. The step-size is $h=0.001$ for this simulation.}
    \label{fig:flagtraj}
\end{figure}

\begin{figure}[H]
    \centering
    \includegraphics[width=0.9\textwidth]{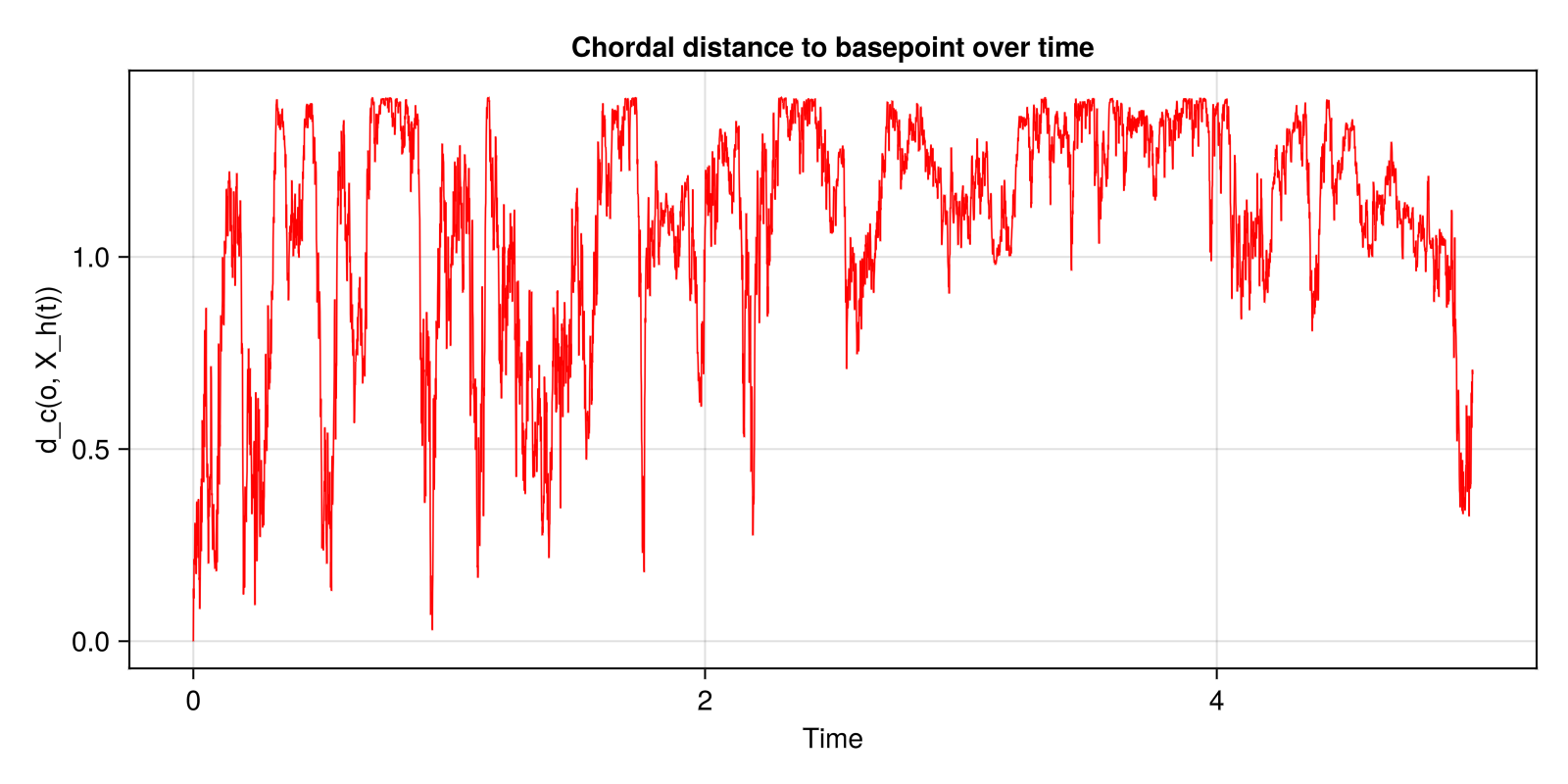}
    \caption{The chordal distance process $t\mapsto d_{c}(o,X_h(t))$, see~\eqref{eq:chordalflag}, associated with the realisation displayed in Figure~\ref{fig:flagtraj}. The simulation runs to final time $T=5$. The chordal distance is bounded by $\sqrt{2}$. The step-size is $h=0.001$ for this simulation.}
    \label{fig:flagdist}
\end{figure}

\subsection{Stiefel manifolds}\label{sec:stiefel}
\paragraph{Stiefel manifolds as reductive homogeneous spaces} Stiefel manifolds arise as parameter spaces in statistics of frame-valued data (matrix Bingham–von Mises–Fisher models, principal component analysis, and factor models with ordered components); Riemannian Brownian motion on them is the driving noise of the corresponding Langevin samplers and stochastic gradient flows. The Stiefel manifold $V_k(\mathbb{R}^n)$ (with $k,n\in \mathbb{N}$ and $k\leq n$) is the set of all ordered orthonormal $k$-frames in $\mathbb{R}^n$
\begin{equation}
    V_k(\mathbb{R}^n) = \{(\bm{v}_1,\hdots ,\bm{v}_k)\in (\mathbb{R}^n)^k\,|\, \bm{v}_i\cdot \bm{v}_j = \delta_{i,j}\}.
\end{equation}
Each element of a Stiefel manifold is an ordered $k$-tuple of orthonormal $n$-vectors, which can be represented by a matrix $Q=[\bm{v}_1\,\bm{v}_2\,\cdots\,\bm{v}_k]\in\mathbb{R}^{n\times k}$. These matrices satisfy $Q^{\intercal}Q=I_{k}$ and thus one can equivalently define Stiefel manifolds as
\begin{equation}\label{eq:stiefel_matrix}
V_k(\mathbb{R}^n) = \{Q\in\mathbb{R}^{n\times k}\,|\, Q^{\intercal}Q = I_{k}\}.
\end{equation}
Note that $V_n(\mathbb{R}^n)= \operatorname{O}(n)$, and $V_1(\mathbb{R}^n)= S^{n-1}$.

As in Section~\ref{sec:flag}, we use Proposition~\ref{prop:HPviaH} to identify $V_k(\R^n)$ as a homogeneous space. We select as base point $o\in V_k(\mathbb{R}^n)$ the frame consisting of the first $k$ standard basis vectors of $\mathbb{R}^n$
\begin{equation}
    o = (\bm{e}_1,\hdots, \bm{e}_k)\in V_k(\mathbb{R}^n),\qquad\text{i.e.,}\qquad o = \begin{pmatrix} I_{k}\\ 0_{(n-k)\times k}\end{pmatrix}\in \R^{n\times k}.
\end{equation}
For any $Q\in V_k(\R^n)$ pick $g\in \operatorname{O}(n)$ such that $g\,o=Q$, i.e., an orthogonal matrix whose first $k$ columns are the columns of $Q$. Next, set
\begin{equation}\label{eq:stiefel_H}
    H = \left\{\begin{pmatrix}
        I_{k} & 0_{k\times(n-k)}\\
        0_{(n-k)\times k} & h
    \end{pmatrix}\,\Bigg|\, h\in \operatorname{O}(n-k)\right\} \cong \operatorname{O}(n-k),
\end{equation}
which is a closed subgroup of $\operatorname{O}(n)$. Note that $g'o=Q$ if and only if $g'\in gH$, so that we can identify the frame $Q$ with the coset $gH=\{gh\colon h\in H\}$. Hence, Proposition~\ref{prop:HPviaH} provides $V_k(\R^n)\cong \operatorname{O}(n)/H$ with the structure of a smooth manifold on which $\operatorname{O}(n)$ acts smoothly and transitively, the action being given by matrix multiplication,
\begin{equation}\label{eq:stiefel_action}
    \Phi_g(\bm{v}_1,\hdots, \bm{v}_k) = (g\bm{v}_1, \hdots, g \bm{v}_k), \quad g\in \operatorname{O}(n).
\end{equation}
Note that the Stiefel manifold is also a compact embedded submanifold of $\R^{n\times k}$: as $I_k$ is a regular value of the map $Q\mapsto Q^{\intercal}Q$ from $\R^{n\times k}$ to $\operatorname{Sym}(k)$, $V_k(\mathbb{R}^n)$ is an embedded submanifold of $\R^{n\times k}$ of dimension
\begin{equation}
    \dim V_k(\mathbb{R}^n) = nk-\frac{k(k+1)}{2},
\end{equation}
with tangent space $T_QV_k(\R^n)=\{V\in \R^{n\times k}\,|\, Q^{\intercal}V+V^{\intercal}Q=0\}$. The two manifold structures coincide: the bijection $\operatorname{O}(n)/H\ni gH\mapsto g\,o\in V_k(\R^n)\subseteq \R^{n\times k}$ is smooth (as $\pi$ is a submersion and $g\mapsto g\,o$ is smooth on $\operatorname{O}(n)$), its differential at $o$ is injective (see~\eqref{eq:stiefel_ToM} below) and hence, by $\operatorname{O}(n)$-equivariance, injective everywhere, and both manifolds have the same dimension (see~\eqref{eq:stiefel_m} below); it is therefore a diffeomorphism. One can also view the Stiefel manifold as an $\operatorname{O}(k)$-principal bundle over the Grassmannian $\operatorname{Gr}_k(\mathbb{R}^n)$ of Section~\ref{sec:flag}, by mapping a frame to its span. \par
The quotient $V_k(\R^n)\cong \operatorname{O}(n)/H$ leads us to consider
\begin{equation}\label{eq:stiefelreddecomp}
    \mathfrak{so}(n) = \mathfrak{h}\oplus \mathfrak{m},
\end{equation}
where $\mathfrak{h}\cong \mathfrak{so}(n-k)$ is the Lie algebra of the isotropy group $H$, given by
\begin{equation}
    \mathfrak{h} = \left\{\begin{pmatrix}
        0_{k\times k} & 0_{k\times(n-k)}\\
        0_{(n-k)\times k} & A
    \end{pmatrix}\,\Bigg|\, A\in\mathfrak{so}(n-k)\right\}
\end{equation}
and the complementary subspace $\mathfrak{m}$ consists of matrices of the form
\begin{equation}\label{eq:stiefel_m}
    \mathfrak{m} = \left\{ X_{C,B}:=\begin{pmatrix}
        C & -B^{\intercal}\\
        B & 0_{(n-k)\times(n-k)}
    \end{pmatrix} \,\Bigg|\, B\in\mathbb{R}^{(n-k)\times k}, C\in\mathfrak{so}(k) \right\}.
\end{equation}
The isomorphism $\mathfrak{m}\cong T_oV_k(\R^n)$ of~\eqref{eq:link_m_ToM} is given explicitly by
\begin{equation}\label{eq:stiefel_ToM}
    X_{C,B}^*(o) = \frac{\bm{d}}{\bm{d}t}\Big|_{t=0}\exp(tX_{C,B})\,o = X_{C,B}\,o = \begin{pmatrix} C\\ B\end{pmatrix}\in T_oV_k(\R^n),
\end{equation}
i.e., $X_{C,B}^*(o)$ consists of the first $k$ columns of $X_{C,B}$. The following proposition shows that the decomposition~\eqref{eq:stiefelreddecomp} is reductive; the proof is straightforward and therefore omitted.
\begin{proposition}\label{prop:stiefel}
    The following inclusions hold
    \begin{equation}
        [\mathfrak{h},\mathfrak{h}]\subseteq\mathfrak{h}, \quad [\mathfrak{h},\mathfrak{m}]\subseteq \mathfrak{m}.
    \end{equation}
    Moreover, $\mathfrak{m}$ is $\operatorname{Ad}_H$-invariant: for $h\in H$ as in~\eqref{eq:stiefel_H} and $X_{C,B}\in \mathfrak{m}$ one has
    \begin{equation}\label{eq:stiefel_Ad}
    \operatorname{Ad}_{h}(X_{C,B}) = hX_{C,B}h^{\intercal} = \begin{pmatrix}
        C & -(hB)^{\intercal}\\
        hB & 0_{(n-k)\times(n-k)}
    \end{pmatrix} = X_{C,hB}.
\end{equation}
\end{proposition}

As $H\cong\operatorname{O}(n-k)$ is not connected, Lemma~\ref{lem:connected_H_reductive} does not apply, but~\eqref{eq:stiefel_Ad} shows directly that $V_k(\R^n)\cong \operatorname{O}(n)/\operatorname{O}(n-k)$ is a reductive homogeneous space in the sense of Definition~\ref{def:rhs} with reductive decomposition~\eqref{eq:stiefelreddecomp}. The inclusion $[\mathfrak{m},\mathfrak{m}]\subseteq\mathfrak{h}$ is not valid in general: for $X_{C,0},X_{0,B}\in \mathfrak{m}$ a direct calculation gives $[X_{C,0},X_{0,B}] = X_{0,-BC}$, which is a nonzero element of $\mathfrak{m}$ whenever $BC\neq 0$. Thus, for $1<k<n$ the decomposition~\eqref{eq:stiefelreddecomp} does not satisfy~\eqref{eq: LTS}, and Lemma~\ref{lemma: LTS} does not apply (for $k=1$ one has $\mathfrak{so}(1)=\{0\}$ and we recover the sphere of Section~\ref{sec:spheres}, up to a relabelling of the coordinates). Nevertheless, upon equipping the Stiefel manifold $V_k(\mathbb{R}^n)$ with the normal metric induced by the bi-invariant metric $Q(X,Y)=\frac12\tr(X^{\intercal}Y)$ on $\operatorname{O}(n)$ (note that $\mathfrak{m}=\mathfrak{h}^{\perp}$ with respect to $Q$), we obtain a normal homogeneous space, for which Proposition \ref{prop:normalhomogeneous} shows that this is naturally reductive. However, analogous to flag manifolds, we may equip Stiefel manifolds with much more general $\operatorname{O}(n)$-invariant metrics. Since Stiefel manifolds are obtained as quotients of $\operatorname{O}(n)$, which is compact and hence unimodular, Proposition \ref{prop:unimodularredhombm} applies: if $g$ is the solution to~\eqref{eq:stratsde}, where $G=\operatorname{O}(n)$ and $\mathfrak{m}$ is equipped with an $\operatorname{Ad}_{\operatorname{O}(n)}$-invariant metric $B$, then $\pi(g)$ is a Riemannian Brownian motion on $(V_k(\R^n),\bm{g}_{B})$ starting at $o$.

\paragraph{$\operatorname{O}(n)$-invariant metrics on Stiefel manifolds.} By~\eqref{eq:stiefel_Ad}, the isotropy representation~\eqref{eq:isotropyrep} of $H\cong \operatorname{O}(n-k)$ on $\mathfrak{m}$ reduces to the standard linear action of $\operatorname{O}(n-k)$ on $\mathbb{R}^{(n-k)\times k}$ by left multiplication $B\mapsto hB$, and $H$ acts trivially on the $C$ component. Since $B$ can be viewed as $k$ column vectors in $\mathbb{R}^{n-k}$, $\operatorname{O}(n-k)$ acts diagonally on these $k$ columns. This means that for $k=1$, the isotropy representation is irreducible (which coincides with Proposition \ref{prop:irrepsphere}).

For $1<k<n$, the isotropy representation is reducible: writing $\mathfrak{m}=\mathbb{R}^{(n-k)\times k}\oplus\mathfrak{so}(k)$ according to the blocks $B$ and $C$, it decomposes into a direct sum of $k$ copies of the standard representation of $\operatorname{O}(n-k)$ on $\mathbb{R}^{n-k}$ (one for each column of $B$) and $\frac12k(k-1)$ copies of the trivial representation (one for each entry of $C$ above the diagonal). The standard representation is irreducible by Proposition~\ref{prop:irrepsphere} (applied to $\operatorname{O}(n-k)\supseteq \operatorname{SO}(n-k)$; for $n-k=1$ it is the sign representation of $\operatorname{O}(1)=\{\pm1\}$ on $\R$), it is not isomorphic to the trivial representation, and its only self-intertwiners are the real multiples of the identity. Since the $k$ copies are isomorphic, Schur's lemma implies that the $\operatorname{Ad}_H$-invariant inner products on the $B$ part are precisely the inner products $(B,B')\mapsto \tr(B^{\intercal}B'S)$ with $S\in \operatorname{SPD}(k)$, whereas \emph{every} inner product on the $C$ part $\mathfrak{so}(k)$ is $\operatorname{Ad}_H$-invariant, and the two parts are orthogonal. Thus, in contrast to the flag manifolds of Section~\ref{sec:flag}, the invariant inner products on $\mathfrak{m}$ allow mixing across the columns in the $B$ part. For the $C$ part we restrict ourselves to inner products that are diagonal with respect to the standard basis of $\mathfrak{so}(k)$, and consider the family
\begin{equation}\label{eq:stiefelinnerproduct}
\langle X_{C,B}, X_{C',B'} \rangle_{S,\alpha} = \operatorname{tr}(B^{\intercal} B' S) + \sum_{1\leq p<q\leq k} \alpha_{p,q}\, C_{p,q}C'_{p,q},\quad X_{C,B},X_{C',B'}\in \mathfrak{m},
\end{equation}
where $S\in \operatorname{SPD}(k)$ and $\alpha=(\alpha_{p,q})_{1\leq p<q\leq k}\in (0,\infty)^{k(k-1)/2}$. Note that, as each entry of the skew-symmetric matrix $C$ above the diagonal appears twice in $\tr(C^{\intercal}C')$ and each entry of $B$ appears twice in $\tr(X_{C,B}^{\intercal}X_{C',B'})$, one has $\langle X_{C,B}, X_{C',B'} \rangle_{S,\alpha} = \frac12\tr(X_{C,B}^{\intercal}X_{C',B'})=Q(X_{C,B},X_{C',B'})$ when $S=I_k$ and all $\alpha_{p,q}=1$, i.e., this choice corresponds to the normal metric on the Stiefel manifold. By~\eqref{eq:gB} and~\eqref{eq:stiefel_ToM}, the $\operatorname{O}(n)$-invariant Riemannian metric $\bm{g}_{S,\alpha}:=\bm{g}_{\langle\cdot,\cdot\rangle_{S,\alpha}}$ on $V_k(\R^n)$ induced by~\eqref{eq:stiefelinnerproduct} is given as follows: for $x\in V_k(\R^n)$, pick $g\in \operatorname{O}(n)$ with $x=g\,o$ (i.e., $g$ is any orthogonal matrix whose first $k$ columns form the frame $x$); then for $V,W\in T_xV_k(\R^n)$ one has $g^{-1}V=g^{\intercal}V,\,g^{\intercal}W\in T_oV_k(\R^n)$, and writing $g^{\intercal}V=\big(\begin{smallmatrix} C_V\\ B_V\end{smallmatrix}\big)$ and $g^{\intercal}W=\big(\begin{smallmatrix} C_W\\ B_W\end{smallmatrix}\big)$ with $C_V,C_W\in \mathfrak{so}(k)$ and $B_V,B_W\in\R^{(n-k)\times k}$,
\begin{equation}\label{eq:stiefelmetric}
    \bm{g}_{S,\alpha}(x;V,W) = \operatorname{tr}\left(B_V^\intercal B_W S\right) + \sum_{1\leq p<q\leq k} \alpha_{p,q}\,(C_V)_{p,q}(C_W)_{p,q}
\end{equation}
(the right-hand side does not depend on the choice of $g$ by the $\operatorname{Ad}_H$-invariance of~\eqref{eq:stiefelinnerproduct}, cf.\ the discussion following~\eqref{eq:gB}). This metric has much more parametric freedom than any of the metrics considered in the previous examples, as it allows cross-coupling between the frame columns through the symmetric positive-definite matrix $S$.

\paragraph{Constructing the Brownian motion and simulations.}
To generate an orthonormal basis for the inner product~\eqref{eq:stiefelinnerproduct}, we write the Cholesky factorisation $S=LL^\intercal$ with $L$ lower triangular. Let the standard basis for the $B$ component in $\mathfrak{m}$ be denoted by $F_{p,q}:=X_{0,E_{p,q}}$, where $p\in \{1,\hdots,n-k\}$, $q\in \{1,\hdots,k\}$ and $E_{p,q}=\bm{e}_p\bm{e}_q^{\intercal}$ is the $(n-k)\times k$ matrix with a one at position $(p,q)$ and zeros elsewhere. In block matrix form, each generator is given by
\begin{equation}
F_{p,q} = \begin{pmatrix}
0_{k\times k} & -E_{p,q}^\intercal \\
E_{p,q} & 0_{(n-k)\times(n-k)}
\end{pmatrix}.
\end{equation}
Notice that the column index $q$ tracks which of the $k$ frame columns the generator affects. The inner product $\tr(B^{\intercal}B'S)$ mixes the columns according to the matrix $S$. To construct a set of orthonormal basis elements that accounts for this coupling, we transform the column index using the inverse Cholesky factor $R=L^{-\intercal}$ as follows
\begin{equation}\label{eq:stiefel_tildeF}
\tilde{F}_{p,j} = \sum_{q=1}^k R_{q,j} F_{p,q}, \quad p = 1, \dots, n-k, \quad j = 1, \dots, k,
\end{equation}
i.e., $\tilde{F}_{p,j}=X_{0,\bm{e}_p\bm{r}_j^{\intercal}}$ with $\bm{r}_j$ the $j$-th column of $R$. Indeed, $\langle \tilde{F}_{p,j},\tilde{F}_{p',j'}\rangle_{S,\alpha} = \tr(\bm{r}_j\bm{e}_p^{\intercal}\bm{e}_{p'}\bm{r}_{j'}^{\intercal}S) = \delta_{p,p'}\,\bm{r}_{j'}^{\intercal}S\bm{r}_j = \delta_{p,p'}(R^{\intercal}SR)_{j',j} = \delta_{p,p'}\delta_{j,j'}$, as $R=L^{-\intercal}$ and $S=LL^{\intercal}$. For the $C$ component we take the standard basis elements $F^{C}_{p,q}:=X_{\bm{e}_p\bm{e}_q^{\intercal}-\bm{e}_q\bm{e}_p^{\intercal},0}$, $1\leq p<q\leq k$, of $\mathfrak{so}(k)$, appropriately scaled by the inverse square root of $\alpha_{p,q}$: by~\eqref{eq:stiefelinnerproduct}, $\langle F^C_{p,q},F^C_{p',q'}\rangle_{S,\alpha}=\alpha_{p,q}\delta_{p,p'}\delta_{q,q'}$, so that
\begin{equation}\label{eq:stiefel_ONB}
    \big(\tilde{F}_{p,j}\big)_{p\in\{1,\ldots,n-k\},\ j\in\{1,\ldots,k\}}\quad\text{together with}\quad \big(\alpha_{p,q}^{-\nicefrac12}F^{C}_{p,q}\big)_{1\leq p<q\leq k}
\end{equation}
is an orthonormal basis of $(\mathfrak{m},\langle\cdot,\cdot\rangle_{S,\alpha})$.

The $\mathfrak{m}$-valued Brownian motion~\eqref{eq:mbrownianmotion} associated with the orthonormal basis~\eqref{eq:stiefel_ONB} can be written explicitly: let $\hat{W}=(W_{p,j})_{p\leq n-k,\ j\leq k}$ be an $\R^{(n-k)\times k}$-valued Brownian motion with independent standard real-valued entries, and let $(W^{C}_{p,q})_{1\leq p<q\leq k}$ be further independent standard real-valued Brownian motions. Then, by~\eqref{eq:stiefel_tildeF} and $R^{\intercal}=L^{-1}$,
\begin{equation}\label{eq:stiefel_Wm}
    W_{\mathfrak{m}}(t) = \sum_{p,j}W_{p,j}(t)\tilde{F}_{p,j}+\sum_{p<q}\alpha_{p,q}^{-\nicefrac12}W^{C}_{p,q}(t)F^{C}_{p,q},\quad t\geq 0.
\end{equation}
The $\operatorname{O}(n)$-valued process $g$ of~\eqref{eq:stratsde} solves $\bm{d}g(t)=g(t)\circ \bm{d}W_{\mathfrak{m}}(t)$, $g(0)=I_n$, and $X(t)=\pi(g(t))=g(t)\,o$, the matrix consisting of the first $k$ columns of $g(t)$, is a Riemannian Brownian motion on $(V_k(\R^n),\bm{g}_{S,\alpha})$ starting at $o$. For the simulation we use the Wong--Zakai approximation~\eqref{eq:prewongzakai}: with $g_\ell:=\exp(W_{\mathfrak{m}}(\ell h)-W_{\mathfrak{m}}((\ell-1)h))\in \operatorname{SO}(n)$, $\ell\in \N$ (we use $\ell$ as time index in this section, as $k$ denotes the frame dimension), one has
\begin{equation}\label{eq:stiefel_WZ}
    X_h(\ell h) = \pi(g_1\cdots g_\ell) = g_1\cdots g_\ell\, o,\quad \ell\in \N,
\end{equation}
i.e., $X_h(\ell h)$ consists of the first $k$ columns of the orthogonal matrix $g_1\cdots g_\ell$; in particular, $X_h(\ell h)$ is an orthonormal $k$-frame for every $\ell\in\N$ without any projection step.

For our simulations we consider Stiefel manifolds over $\R^3$. Note that $V_1(\mathbb{R}^3)\cong \operatorname{O}(3)/\operatorname{O}(2)= S^2$, $V_2(\mathbb{R}^3)\cong \operatorname{O}(3)/\operatorname{O}(1)\cong \operatorname{O}(3)/\mathbb{Z}_2\cong \operatorname{SO}(3)$ (every coset of $H=\{\operatorname{diag}(1,1,\pm1)\}$ contains exactly one rotation) and $V_3(\mathbb{R}^3)= \operatorname{O}(3)$. For the normal metric, the Riemannian Brownian motion on $V_1(\mathbb{R}^3)$ coincides with the Riemannian Brownian motion on the sphere of Section~\ref{sec:spheres} (started at $\bm{e}_1$ instead of $\bm{e}_3$). On $V_3(\R^3)=\operatorname{O}(3)$ the isotropy group is trivial, $\mathfrak{m}=\mathfrak{so}(3)$, and the Riemannian Brownian motion is the solution $g$ to~\eqref{eq:stratsde} itself; since it has continuous sample paths starting at the identity, it stays in the connected component $\operatorname{SO}(3)$ of the identity. Moreover, under the normal metric $V_2(\mathbb{R}^3)$ is isometric to $\operatorname{SO}(3)$ with the bi-invariant metric induced by $Q$: the map $\operatorname{SO}(3)\rightarrow V_2(\R^3)$, $A\mapsto A\,o$ (the first two columns of $A$), is a bijection that coincides with the restriction of $\pi$ to $\operatorname{SO}(3)$, and $\pi|_{\operatorname{SO}(3)}$ is a local isometry by Lemma~\ref{lem:submersion}, as $\mathfrak{h}=\{0\}$. In particular, the Riemannian Brownian motion on $V_2(\R^3)$ for the normal metric is the image of the one on $\operatorname{SO}(3)$ for the bi-invariant metric.

In view of the above, the interesting case to consider is that of an anisotropic metric: we consider $V_2(\mathbb{R}^3)$ with an anisotropic metric. As in the previous examples, we would like to plot the distance of the process to $o$ as a function of time. As in the setting of the flag manifolds, we do not have a closed form expression for the Riemannian distance in $V_2(\mathbb{R}^3)$ (even if we were to take the normal metric). Hence on Stiefel manifolds we simply compute the Frobenius distance $t\mapsto\|X_h(t)-o\|_F$ between the base frame $o$ and the current frame. Note that the Frobenius distance still detects the compactness of the underlying space: as $\|Q-Q'\|_F^2 = 2k-2\tr(Q^{\intercal}Q')$ for $Q,Q'\in V_k(\R^n)$, it is bounded by $2\sqrt{k}$, which is attained at $Q'=-Q$ (for $V_2(\R^3)$ this bound is $2\sqrt{2}$). Note that by~\eqref{eq:stiefel_ToM} the inclusion $V_k(\R^n)\subseteq\R^{n\times k}$ pulls the Frobenius inner product back to $\|X_{C,B}\,o\|_F^2 = \tr(B^{\intercal}B)+2\sum_{p<q}C_{p,q}^2$ on $\mathfrak{m}$, i.e., the inclusion is an isometric embedding of $(V_k(\R^n),\bm{g}_{S,\alpha})$ precisely when $S=I_k$ and $\alpha\equiv 2$, and only in that case is the Frobenius distance the extrinsic distance of an isometric embedding. To set up the anisotropic metric, we take the Cholesky factor
\begin{equation}
    L = \begin{pmatrix}
        1 & 0\\
        2 & 1
    \end{pmatrix},
\end{equation}
and construct $S=LL^\intercal = \big(\begin{smallmatrix} 1 & 2\\ 2 & 5\end{smallmatrix}\big)$. We further set $\alpha_{1,2}=16$. By~\eqref{eq:stiefel_Wm}, the $\mathfrak{m}$-valued Brownian motion is then
\begin{equation}
    W_{\mathfrak{m}}(t) = \begin{pmatrix} 0 & \tfrac14 W^{C}_{1,2}(t) & -B_{1,1}(t)\\ -\tfrac14 W^{C}_{1,2}(t) & 0 & -B_{1,2}(t)\\ B_{1,1}(t) & B_{1,2}(t) & 0\end{pmatrix},
\end{equation}
where
\begin{equation}
    \begin{aligned}
    (B_{1,1}(t),B_{1,2}(t)) &= (W_{1,1}(t),W_{1,2}(t))L^{-1}\\ &= \big(W_{1,1}(t)-2W_{1,2}(t),\,W_{1,2}(t)\big),
    \end{aligned}
\end{equation}
for independent standard real-valued Brownian motions $W_{1,1},W_{1,2},W^{C}_{1,2}$. In Figure \ref{fig:stiefeltraj}, a trajectory of the Riemannian Brownian motion with the anisotropic metric built from these parameters is shown: as for the flag manifold in Figure~\ref{fig:flagtraj}, we plot the first column of the frame $X_h(\ell h)$, which is a unit vector in $\R^3$. Figure \ref{fig:stiefeldist} shows the Frobenius distance process.

\begin{figure}[H]
    \centering
    \includegraphics[width=0.9\textwidth]{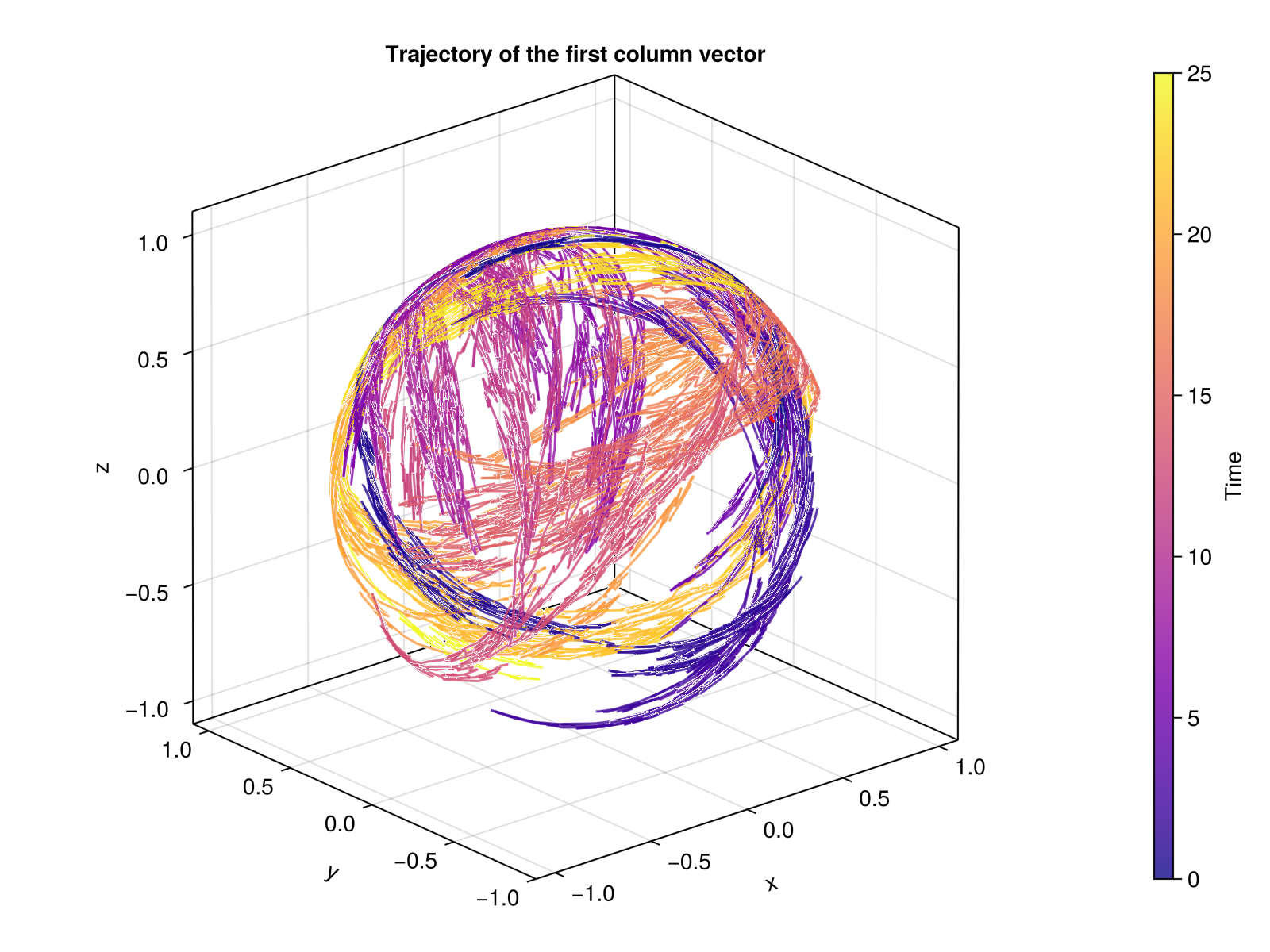}
    \caption{A realisation of the Wong--Zakai approximation of the anisotropic Riemannian Brownian motion on the Stiefel manifold $V_2(\mathbb{R}^3)$ with $S=LL^{\intercal}$, $L=\big(\begin{smallmatrix} 1 & 0\\ 2 & 1\end{smallmatrix}\big)$ and $\alpha_{1,2}=16$, started at $o=(\bm{e}_1,\bm{e}_2)$. Plotted is the first column vector of the frame, which lives on $S^2$. Compared to the anisotropic Riemannian Brownian motion on the complete flag manifold $\operatorname{Fl}(1,2;\mathbb{R}^3)$ in Figure~\ref{fig:flagtraj}, the trajectory shows strong correlations. The step-size is $h=0.001$ for this simulation.}
    \label{fig:stiefeltraj}
\end{figure}

\begin{figure}[H]
    \centering
    \includegraphics[width=0.9\textwidth]{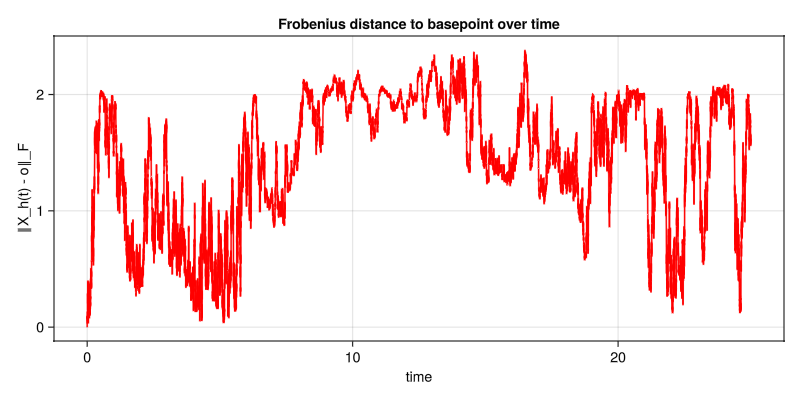}
    \caption{The Frobenius distance process $t\mapsto \|X_h(t)-o\|_F$ on $V_2(\mathbb{R}^3)$ associated with the realisation of Figure~\ref{fig:stiefeltraj}. The Frobenius distance is bounded by $2\sqrt{2}$. The step-size is $h=0.001$ for this simulation.}
    \label{fig:stiefeldist}
\end{figure}

\section*{AI disclosure statement}
GenAI tools such as Claude Fable and Gemini Pro have been used to perform or verify tedious calculations, for finding results in the literature, and for the final proofreading of this paper. GenAI was also used to generate the introduction based on a rough outline provided by the authors. All content (including all references) has been subsequently verified by the authors, who take full responsibility for the final version.

\section*{Acknowledgements}
The authors would like to thank Bas Janssens for help with Section~\ref{sec:redhomspaces} and Jasper Stokman for discussions on isotropy representations.
\bibliographystyle{amsalpha}
\bibliography{biblio}
\end{document}